\documentclass[11pt]{article}
\usepackage[top=1in, bottom=1in, left=1in, right=1in]{geometry}
\usepackage[linktocpage,colorlinks,linkcolor=blue,anchorcolor=blue,citecolor=blue,urlcolor=blue,pagebackref]{hyperref}
\usepackage{euscript,mdframed}
\usepackage{microtype,todonotes,relsize}
\usepackage{amsmath,amsthm}
\usepackage{algpseudocode}

\usepackage{accents}
\usepackage{comment,url,graphicx,relsize}
\usepackage{amssymb,amsfonts,amsmath,amsthm,amscd,dsfont,mathrsfs,mathtools,nicefrac, bm}
\usepackage{float,psfrag,epsfig,color,xcolor,url}
\usepackage{epstopdf,bbm,mathtools,enumitem}
\usepackage[ruled,vlined,linesnumbered]{algorithm2e}
\usepackage{tablefootnote}
\graphicspath{{Plots/}}
\usepackage{diagbox}
\usepackage{tikz}
\usetikzlibrary{calc,patterns,angles,quotes}
\usepackage{makecell}
\usepackage{multirow}
\usepackage{booktabs}
\usepackage[authoryear,round]{natbib}
\usepackage{color}
\usepackage{parskip}
\usepackage{subcaption}

\newtheorem{theorem}{Theorem}
\newtheorem{lemma}{Lemma}

\newtheorem{assumption}{Assumption}
\newtheorem{remark}{Remark}

\newtheorem{definition}{Definition}

\newcommand{\inner}[2]{\langle{#1},{#2}\rangle}

\newcommand{\R}{\mathbb{R}}

\newcommand{\cD}{\mathcal{D}}

\newcommand{\cM}{\mathcal{M}}

\newcommand{\cT}{\mathcal{T}}

\newcommand{\E}{\mathbb{E}}
\newcommand{\Exp}{{\mathtt{Exp}}}
\newcommand{\grad}{{\mathtt{grad}}}
\newcommand{\hess}{{\mathtt{Hess}}}
\newcommand{\Retr}{{\mathtt{Retr}}}
\newcommand{\dx}{{d_{\mathcal{M}_x}}}
\newcommand{\dy}{{d_{\mathcal{M}_y}}}

\title{Adaptive Single-Loop Methods for Stochastic Minimax Optimization on Riemannian Manifolds}
\author{
{Hongye Wang} \thanks{School of Information Management and Engineering, Shanghai University of Finance and Economics. \texttt{ishongyewang@gmail.com}}
\and
{Chang He} \thanks{School of Information Management and Engineering, Shanghai University of Finance and Economics. \texttt{ischanghe@gmail.com}}
\and 
{Bo Jiang} \thanks{School of Information Management and Engineering, Shanghai University of Finance and Economics. \texttt{isyebojiang@gmail.com}}
}

\begin{document}

\maketitle
\begin{abstract}
Stochastic minimax optimization on Riemannian manifolds has recently attracted significant attention due to its broad range of applications, such as 
robust training of neural networks and robust maximum likelihood estimation. Existing optimization methods for these problems typically require selecting stepsizes based on prior knowledge of specific problem parameters, such as Lipschitz-type constants  and (geodesic) strong concavity constants. Unfortunately, these parameters are often unknown in practice. To overcome this issue, we develop single-loop  adaptive methods that automatically adjust stepsizes using cumulative Riemannian (stochastic) gradient norms. We first propose a deterministic single-loop Riemannian adaptive gradient descent ascent method and show that it attains an $\epsilon$-stationary point within $\mathcal{O}(\epsilon^{-2})$ iterations. This deterministic method is of independent interest and lays the foundation for our subsequent stochastic method. In particular,
we propose the Riemannian stochastic adaptive gradient descent ascent method, which finds an $\epsilon$-stationary point in $\mathcal{O}(\epsilon^{-6})$ iterations. Under additional second-order smoothness, this iteration complexity is further improved to $\mathcal{O}(\epsilon^{-4})$, which even outperforms the corresponding complexity result in Euclidean space. Some numerical experiments on real-world applications are conducted, including the regularized robust maximum likelihood estimation problem, and the robust training of neural networks with orthonormal weights. The results are encouraging and demonstrate the effectiveness of adaptivity in practice.
\end{abstract}

\section{Introduction}
Over the past decades, 
optimization on 
Riemannian manifolds has attracted growing research interest in machine learning, scientific computing, and signal processing, due to the natural geometric structure inherent to many real-world applications in these domains \citep{absil2009optimization,boumal2011rtrmc, vandereycken2013low, huang2018orthogonal, chakraborty2020manifoldnet, cherian2016riemannian,boumal2023introduction}.
In order to accommodate a wider spectrum of practical applications, this paper considers the following \textit{stochastic minimax optimization problem}:
\begin{equation}\label{eq.stochastic problem}
\min_{x \in \cM_x}\max_{y\in\cM_y} \ f(x,y) \coloneqq \mathbb{E}_{\xi \sim \cD} [f(x,y; \xi)],
\end{equation}
where objective function $f(x,y)$ is nonconvex in $x$ and geodesically strongly concave in $y$. The two variables $x$ and $y$ lie on complete Riemannian manifolds $\cM_x$ and $\cM_y$, with dimensions $d_x$ and $d_y$ respectively. 
To further underscore the practical significance of the considered problem, in the following, we introduce two concrete machine learning tasks as motivating examples.
\subsection{Motivating examples}
\paragraph{Robust training of neural networks.} 
Given a input-label distribution $\mathcal{D}$, the forward mapping $h(\cdot)$ of neural networks and the loss function $\mathcal{L}(\cdot,\cdot)$, the training of neural networks can be formulated as:
\begin{equation}
    \min_{\{X_k\}_{k=1}^K : X_k \in \mathrm{St}(d_k, d_{k+1})}
\E_{(a_i,b_i)\in \cD} [\mathcal{L}\left(h(a_i; \{X_k\}_{k=1}^K), b_i\right)],
\end{equation}
where $\mathrm{St}(d_k, d_{k+1})=\left\{X_k\in\R^{d_k\times d_{k+1}}:X_k^\top X_k=I\right\}$ denotes the Stiefel manifold with  $X_k$ being the weights of the $k$-th layers. The reason to enforce the orthonormality on network weights is to  enhance both the stability and speed of convergence during training process \citep{bansal2018can, cogswell2015reducing, huang2018orthogonal, wang2020orthogonal}. In practice, the true distribution $\mathcal{D}$ is unknown and only approximated by a sampling distribution $\Tilde{\cD}$, which can make the model sensitive to small input perturbations. To promote robustness against such sensitivity, we introduce an auxiliary perturbation variable $y$ on a sphere and minimize the worst-case loss \citep{madry2017towards}:
\begin{equation}\label{eq.robust nn}
    \min_{\{X_k\}_{k=1}^K : X_k \in \mathrm{St}(d_k, d_{k+1})}
\max_{y \in \mathbb{S}^{d-1}(r)} 
\E_{(a_i,b_i)\in \Tilde{\cD}} [\mathcal{L}\left(h(a_i + y; \{X_k\}_{k=1}^K), b_i\right)],
\end{equation}
where $\mathbb{S}^{d-1}(r)=\left\{x\in \R^d: \|x\|_2=r\right\}$ is the sphere manifold of radius $r$. Thus, this formulation is a specific instance of the problem defined in \eqref{eq.stochastic problem}.

\paragraph{Robust Maximum Likelihood Estimation.}
The robust maximum likelihood estimation problem, introduced in \cite{bertsimas2019robust}, is to obtain statistically reliable estimators that remain stable under data uncertainty or perturbations. Later, \cite{hosseini2020alternative} proposed the following minimax reformulation to describe this problem:
\begin{equation}\label{eq.like}
    \min_{x\in \mathbb{S}^{d}}\max_{Y\in\mathcal{P}(d+1)}-\frac{n}{2}\log(\det(Y))-\frac{n}{2}\E_{a_i\in \cD}[([a_i,1]^\top-x)^\top Y^{-1}([a_i,1]^\top-x)].
\end{equation}
where $\mathbb{S}^{d}(1)=\{x\in \R^{d+1}:x^\top x=1\}$ denotes the sphere manifold of radius $1$, $a_i\in \R^d$ being the sampled data from distribution $\cD$, and $\mathcal{P}(d+1)=\{X\in \R^{(d+1)\times (d+1)}:X\succ 0 \}$ is the symmetric positive definite manifold. Here, $n$ and $d$ denote the number of samples and the feature dimension, respectively. Compared with the canonical maximum likelihood estimation, the robust formulation additionally minimizes over $x$ to account for possible perturbations in the data, making it a special case of problem \eqref{eq.stochastic problem}.

\subsection{Related works}
Given the powerful modeling capability of 
problem \eqref{eq.stochastic problem}, several recent studies have probed stochastic minimax optimization on Riemannian manifolds and proposed corresponding optimization methods. However, practical implementation of those methods remains challenging, primarily because they require various unknown problem parameters. For instance, \cite{han2023riemannian} proposed Hamiltonian gradient methods for objective functions whose Hamiltonian functions satisfy the Riemannian PL condition; however, this method requires knowledge of a Lipschitz-type constant.
Similarly, other methods such as the Riemannian corrected extragradient method and Riemannian gradient descent ascent depend on both Lipschitz-type constants and (geodesic) strong concavity constants \citep{jordan2022first,huang2023gradient,wu2023decentralized}.
Therefore, designing \textit{adaptive optimization methods} that do not rely on prior knowledge of these parameters is of great importance.

Recently, several adaptive methods have been developed for Riemannian minimization problems to remove the dependence on specific problem parameters, including \cite{becigneul2018riemannian, kasai2019riemannian, grapiglia2023adaptive, sakai2024general, yagishita2025simple, bento2025riemannian}. However, the development of adaptive methods for Riemannian minimax optimization remains largely unexplored. To the best of our knowledge, the most related work is the adaptive Riemannian hypergradient descent (AdaRHD) method in \cite{shi2025adaptive}, which was originally designed for bilevel optimization problems. In principle, AdaRHD could be applied to minimax optimization, but its convergence guarantees were established only for deterministic settings, making it inapplicable to the stochastic problem \eqref{eq.stochastic problem}.
In addition, it employs a double-loop scheme, which is computationally inefficient. This clear gap in the literature motivates our goal: to design a \textit{single-loop adaptive method} for \textit{stochastic} minimax optimization on Riemannian manifolds that has both strong theoretical guarantees and encouraging practical performance.

\subsection{Main contributions}
In this paper, we propose two adaptive single-loop methods for the Riemannian minimax problem. To eliminate the dependence on specific problem parameters, our methods automatically adjust their stepsizes using cumulative Riemannian (stochastic) gradient norms information
. In particular, we first design a Riemannian adaptive gradient descent ascent (RAGDA) method for the deterministic case of problem \eqref{eq.stochastic problem}. Building on this method, we then develop its stochastic version named Riemannian stochastic adaptive gradient descent ascent (RSAGDA) method for problem \eqref{eq.stochastic problem}. To the best of our knowledge, they are the first single-loop adaptive algorithms for solving Riemannian minimax optimization in both deterministic and stochastic settings.

Theoretically, we establish a suite of iteration complexity bounds for the proposed methods. Specifically, 
we show that under mild assumptions, RAGDA attains an $\epsilon$-stationary point for the deterministic version of problem \eqref{eq.stochastic problem} within $\mathcal{O}(\epsilon^{-2})$ iterations. In the stochastic setting, under standard stochasticity assumptions, RSAGDA finds an $\epsilon$-stationary point of problem \eqref{eq.stochastic problem} in $\mathcal{O}(\epsilon^{-6})$ iterations. 
Furthermore, under an additional second-order smoothness assumption, we show that the complexity of RSAGDA can be improved to
$\mathcal{O}(\epsilon^{-4})$, which even outperforms the corresponding complexity $\mathcal{O}(\epsilon^{-(4+\tilde{\epsilon})})$ for some small $\tilde{\epsilon}>0$ in Euclidean space \citep{li2022tiada}.
Such improvement stems from a sharper analysis that allows and optimizes a more flexible choice of algorithmic hyperparameters.
Numerically, we compare our methods against other baseline methods\footnote{See Section \ref{sec.experiments} for more details.} on the regularized robust maximum likelihood estimation problem and robust neural networks training with orthonormal weights. The experimental results demonstrate the effectiveness of the adaptive mechanism in our method.

\section{Preliminaries: Riemannian geometry}
In this section, we review the fundamental concepts and tools of Riemannian optimization. For a more comprehensive treatment of the underlying theory, we refer the reader to the works of \cite{absil2009optimization, boumal2023introduction}.

A Riemannian manifold $\cM$ is a smooth manifold endowed with a smooth inner product $\langle \cdot, \cdot \rangle_x:\cT_x \cM \times \cT_x \cM\rightarrow \R$ on the tangent space $\cT_x \cM$ (Definition 3.14 in \cite{boumal2023introduction}) for each point $x\in \cM$. The associated norm is denoted by $\lVert u \rVert_x := \sqrt{\langle u, u \rangle_x}$ for $u \in \cT_x \cM$, with the subscript omitted when it is clear from context. A retraction on a manifold $\cM$ is a smooth map $\Retr_x: \cT_x\cM\rightarrow\cM$ such that $\Retr_x(0) = x$ and $\operatorname{D}\Retr_x(0) = \operatorname{id}_{\cT_x\cM}$, {where $\operatorname{D}$ denotes the differential operator
and $\operatorname{id}_{\cT_x\cM}$ is the identity mapping on $\cT_x\cM$. The exponential map $\Exp_x : \cT_x\cM \to \cM$ is a canonical example of a retraction, as it maps a tangent vector to the point reached by moving along the geodesic (locally shortest curve). If $\cM$ is complete, the exponential map is globally defined. Given a geodesic $\gamma:[0,1]\to\cM$ with $\gamma(0) = x$, $\gamma(1) = y$ and $\dot{\gamma}(0) = g_x \in \cT_x \cM$, one has $y = \Exp_x(g_x) \in \cM$. If the smooth inverse of the exponential map exists, the Riemannian distance between two points $x,y \in \cM$ can be expressed as $d(x,y)=\|\Exp^{-1}_x(y)\|_x=\|\Exp^{-1}_y(x)\|_y$. The parallel transport  $\Gamma_{x_1}^{x_2} : \cT_{x_1} \mathcal{M} \to \cT_{x_2} \mathcal{M}$ 
is the linear isometry between tangent spaces such that
$\langle u, v \rangle_{x_1} = \langle \Gamma_{x_1}^{x_2} u, \Gamma_{x_1}^{x_2} v \rangle_{x_2}$. 
Moreover, the Cartesian product of Riemannian manifolds 
$\cM_x \times \cM_y$ is also a Riemannian manifold.

For a differentiable function $f:\cM\to\R$, the Riemannian gradient at $x \in \cM$ is the tangent vector $\grad f(x) \in \cT_x \mathcal{M}$ satisfying $\langle \grad f(x), u \rangle_x = \operatorname{D}f(x)[u]$ for all $u \in \cT_x \mathcal{M}$, where $\operatorname{D}f(x)[u]$ denotes the derivative of $f$ at $x$ in the direction of $u$. If $f$ is twice differentiable, the Riemannian Hessian $\hess f(x)$ is given by the covariant derivative of the Riemannian gradient. Consider a bifunction $f:\cM_x \times \cM_y \to \mathbb{R}$. We denote by $\grad_x f(x,y)$ and $\grad_y f(x,y)$ the Riemannian gradients with respect to $x$ and $y$, respectively, and by $\hess_x f(x,y)$ and $\hess_y f(x,y)$ their corresponding Riemannian Hessians.
The cross-derivatives are defined as linear maps $\grad_{xy}^2 f(x,y) : \cT_y \mathcal{M}_y \to \cT_x \mathcal{M}_x$, 
$\grad_{yx}^2 f(x,y) : \cT_x \mathcal{M}_x \to \cT_y \mathcal{M}_y$, defined, for example, by $\grad_{xy}^2 f(x,y)[v] = \operatorname{D}_y \grad_x f(x,y)[v]$ for any $v \in \cT_y \mathcal{M}_y$, where $\operatorname{D}_y$ denotes the differential with respect to $y$ (the definition of $\grad_{yx}^2 f$ is analogous). For a linear operator $\operatorname{T}: \cT_x\cM_x \to \cT_y\cM_y$, the operator norm of $\operatorname{T}$ is defined by 
$
\|\operatorname{T}\| := \sup_{u \in \cT_x \mathcal{M}_x : \|u\|_x = 1} \|\operatorname{T}[u]\|_y.
$

Having defined the notion of the Riemannian gradient, we now introduce the concept of an $\epsilon$-stationary point, which serves as the convergence criterion that our proposed methods aim to achieve. This definition is consistent with its Euclidean counterpart \citep{li2022tiada}. 
\begin{definition}[$\epsilon$-stationary point]\label{def.stationary}
    Given an accuracy $\epsilon > 0$, a point $(x,y)$ is called an $\epsilon$-stationary point of the deterministic minimax problem \eqref{eq.deterministic problem} if $ \|\grad_xf(x,y)\|\le \epsilon$ and $\|\grad_yf(x,y)\|\le \epsilon$. It is an $\epsilon$-stationary point of the stochastic minimax problem \eqref{eq.stochastic problem} if $\mathbb{E}[\|\grad_xf(x,y)\|^2]\le \epsilon^2$ and $\mathbb{E}[\|\grad_yf(x,y)\|^2]\le \epsilon^2$.
\end{definition}
To facilitate the convergence analysis of our methods, some standard properties of the objective function are required. 
In particular, we assume that the function \(f(x, y)\) is sufficiently smooth and exhibits a certain curvature structure with respect to \(y\). We formally state these properties below.

\begin{assumption}[Lipschitz Smoothness]\label{ass.l-smooth}
The function \( f(x, y) \) is \emph{\(L_1\)-Lipschitz smooth} with respect to both \(x \in \mathcal{M}_x\) and \(y \in \mathcal{M}_y\), i.e., it is continuously differentiable in both variables and there exists a constant \( L_1 > 0 \) such that, for all \( x_1, x_2 \in \mathcal{M}_x \) and \( y_1, y_2 \in \mathcal{M}_y \),
\begin{align*}
\max\{
&\|\Gamma_{x_1}^{x_2}\grad_x f(x_1,y_1) - \grad_x f(x_2,y_2)\|,\\
&\|\Gamma_{y_1}^{y_2}\grad_y f(x_1,y_1) - \grad_y f(x_2,y_2)\|
\} \\
\le &L_1\left(d_{\cM_x}(x_1,x_2) + d_{\cM_y}(y_1,y_2)\right),
\end{align*}
where \(\Gamma_{x_1}^{x_2}\) and \(\Gamma_{y_1}^{y_2}\) denote the parallel transport operators on \(\mathcal{M}_x\) and \(\mathcal{M}_y\), respectively.
\end{assumption}

The smoothness assumption is a standard requirement in the Riemannian optimization literature \citep{lin2020gradient, lin2020near, huang2023gradient}, as it ensures that the Riemannian gradient does not vary too rapidly. This property plays a central role in deriving descent-type inequalities, establishing stability of iterative updates, and controlling the accumulation of gradient norms in convergence analysis.

\begin{assumption}[Geodesic Strong Concavity]\label{ass.g-convex}
The function \( f(x, y) \) is \emph{\(\mu\)-geodesically strongly concave} in \(y\) for some constant \(\mu > 0\), i.e., for any fixed \( x \in \mathcal{M}_x \) and for all \( y_1, y_2 \in \mathcal{M}_y \),
\begin{align*}
    &f(x, y_2)\\ \le &f(x, y_1) + \langle \grad_y f(x, y_1), \Exp_{y_1}^{-1}(y_2) \rangle_{y_1}- \frac{\mu}{2} d_{\mathcal{M}_y}(y_1, y_2)^2.
\end{align*}
\end{assumption}

The geodesic strong concavity assumption, widely adopted in nonconvex minimax problems \citep{yang2022nest, li2022tiada, huang2023gradient}, extends the classical notion of strong concavity to the Riemannian setting. 
It is particularly well-suited for Hadamard manifolds, where the non-positive curvature ensures favorable concavity properties. 
Moreover, it guarantees the existence and uniqueness of the maximizer $y^*(x)$ for any fixed $x$ \citep{han2024framework, shi2025adaptive}.

Let $y^*(x) = \arg\min_{y \in \cM_y} f(x,y)$ denote the solution of the inner maximization problem, and define the corresponding primal function as $\Phi(x) = f(x, y^*(x))$.  To guarantee the well-posedness of the minimax problem and the validity of first-order optimality conditions at $y^*(x)$, we impose the following assumption.
\begin{assumption}\label{ass.optimal exist and stationary condition hold}
The value function $\Phi(x) = \max_{y \in \cM_y} f(x, y)$ is upper bounded by a constant $\Phi_{\max} \in \mathbb{R}$, and satisfies $\Phi^*=\min_{x \in \cM_x} \Phi(x) > -\infty$. Moreover, for any $x \in \cM_x$, the maximizer $y^*(x)$ lies in the interior of $\cM_y$, and the Riemannian gradient of $\Phi(x)$ is uniformly bounded: $\|\grad \Phi(x)\| \le G$ for some constant $G > 0$.
\end{assumption}

For manifolds, we adopt the following standard assumptions.
\begin{assumption}\label{ass.manifold}
The primal manifold $\cM_x$ is complete, and the dual manifold $\cM_y$ is a Hadamard manifold with sectional curvature lower bounded by $\tau < 0$. Moreover, we assume that for each iterate $y_t$, the curvature constant $\frac{\sqrt{|\tau|}d(y_t,y_t^*)}{\tanh(\sqrt{|\tau|}d(y_t,y_t^*))}$ is uniformly upper bounded by a constant $\zeta > 0$.
\end{assumption}
The completeness of $\cM_x$ and $\cM_y$ ensures the iterates remain in a geodesically complete space. The
Hadamard structure of $M_y$ allows for well-defined notions of geodesic concavity, which are essential
in our dual analysis \citep{ zhang2016first, zhang2016riemannian, han2024framework, shi2025adaptive}.

To justify the use of retractions in place of exponential maps, we adopt the following standard approximation condition, which controls the deviation between exponential and retraction mappings.
\begin{assumption}[Retraction Accuracy]\label{ass.retraction}
There exist constants $\bar{c} \ge 1$ and $c_R \ge 0$ such that for any point $z_1 \in \cM_x$ (or $\cM_y$) and tangent vector $u \in \cT_{z_1}\cM_x$ (or $\cT_{z_1}\cM_y$) with $z_2 = \Retr_{z_1}(u)$, we have:
\[
d^2(z_1, z_2) \le \bar{c} \|u\|^2, \quad 
\|\Exp_{z_1}^{-1}(z_2) - u\| \le c_R \|u\|^2.
\]
\end{assumption}
This assumption is reasonable since the retraction serves as a first-order approximation of the exponential map, which is commonly leveraged in theoretical analyses \citep{kasai2018riemannian, han2024framework, shi2025adaptive}.
\section{Riemannian adaptive gradient descent ascent}\label{sec.RAGDA}
Before addressing the stochastic setting, we first study the corresponding deterministic problem:
\begin{equation}\label{eq.deterministic problem}
\min_{x \in \cM_x}\max_{y\in\cM_y}f(x,y).
\end{equation}
and propose the Riemannian adaptive gradient descent ascent method, a single-loop scheme to solve it. 
We opt to begin with the deterministic setting for several reasons.
First, studying the 
deterministic problem yields fundamental design principles that underpin the subsequent development of our stochastic method. Second, the deterministic minimax formulation itself is of independent interest, with numerous applications in
robust maximum likelihood estimation \citep{bertsimas2019robust, zhang2023sion}, 
subspace robust Wasserstein distance \citep{lin2020projection, huang2021riemannian, jiang2023riemannian}, and distributionally robust principal component analysis \citep{zhang2023sion}. Finally, to the best of our knowledge, the only existing adaptive method for deterministic Riemannian minimax optimization adopts a double-loop structure \citep{shi2025adaptive}, which limits its practicality. Thus, designing a simple single-loop adaptive method that is computationally efficient for problem \eqref{eq.deterministic problem} constitutes a nontrivial and valuable contribution to the field.
\subsection{Algorithm design}
We now describe the update rules of the proposed single-loop adaptive method for solving the minimax problem \eqref{eq.deterministic problem}. 
The overall procedure is summarized in Algorithm~\ref{alg.RAGDA}.
At each iteration $t$, the algorithm performs a Riemannian gradient descent step on the primal variable $x \in \cM_x$ and a gradient ascent step on the dual variable $y \in \cM_y$, with gradients
$g^x_t = \grad_x f(x_t, y_t)$ and $g^y_t = \grad_y f(x_t, y_t)$ respectively. The accumulated squared gradient norms are tracked by $v^x_{t+1} = v^x_t + \|g^x_t\|^2$ and $v^y_{t+1} = v^y_t + \|g^y_t\|^2$ based on which the adaptive step sizes are computed. The updates are then given by
\[
x_{t+1} = \Retr_{x_t}\!\left(-\frac{\eta^x}{\max\{v^x_{t+1}, v^y_{t+1}\}^{\alpha}} g^x_t\right), \quad
y_{t+1} = \Retr_{y_t}\!\left(\frac{\eta^y}{(v^y_{t+1})^{\beta}} g^y_t\right).
\]
The use of $\max\{v^x_{t+1}, v^y_{t+1}\}$ in the primal update coordinates the learning rates of $x$ and $y$, preventing overly aggressive primal steps when the dual gradients remain large. This update scheme enables a single-loop implementation with theoretical guarantees, in contrast to AdaRHD \citep{shi2025adaptive}, which adapts steps using only $v^x_{t+1}$ and thus requires a double-loop scheme. The parameters $\eta^x$ and $\eta^y$ are introduced to enforce a suitable time-scale separation between the updates of the primal and dual variables, a well-established practice
in minimax optimization \citep{li2022tiada, lin2020gradient, lin2025two}. Distinct decay exponents $\alpha$ and $\beta$ are introduced to allow flexibility to regulate the rate of stepsize reduction for the primal and dual variables, generalizing the common choice of $1/2$ \citep{ward2020adagrad, shi2025adaptive}. As shown in our analysis, varying the values of these exponents yields different convergence guarantees. In addition, retractions instead of exponential maps are used to enhance computational efficiency, a common practice in Riemannian optimization \citep{absil2009optimization}. Notably, our adaptive update mechanism is inspired by \cite{li2022tiada} and aligns with the philosophy of cumulative-gradient-based adaptation in adaptive optimization methods \citep{ward2020adagrad, li2022tiada, shi2025adaptive}.
\begin{algorithm}[htbp]
\caption{RAGDA (Riemannian adaptive gradient descent ascent method)}
\label{alg.RAGDA}
\KwIn{$(x_0, y_0)$, $v^x_0 > 0$, $v^y_0 > 0$, $\eta^x > 0$, $\eta^y > 0$, $\alpha > 0$, $\beta > 0$.}
\For{$t = 0, 1, 2, \dots$}{
    Let $g^x_t = \grad_x f(x_t, y_t)$ and $g^y_t = \grad_y f(x_t, y_t)$\;
    $v^x_{t+1} = v^x_t + \|g^x_t\|^2$, $v^y_{t+1} = v^y_t + \|g^y_t\|^2$\;\label{alg.line3}
    $\eta_t = \frac{\eta^x}{\max\{v^x_{t+1}, v^y_{t+1}\}^\alpha}$\;
    $x_{t+1} = \Retr_{x_t}\left(- \eta_t g^x_t \right)$\;
    $\gamma_t = \frac{\eta^y}{(v^y_{t+1})^\beta}$\;
    $y_{t+1} = \Retr_{y_t}\left( \gamma_t g^y_t \right)$\;
}
\end{algorithm}
\subsection{Convergence analysis}
In this subsection, we aim to establish the convergence guarantees of the proposed algorithm RAGDA.


The central idea of our analysis is to control the growth of the accumulated the squared norms of the gradients $\sum_{t=0}^{T-1}\|\grad_x f(x_t, y_t)\|$, $\sum_{t=0}^{T-1}\|\grad_y f(x_t, y_t)\|^2$. Intuitively, if both of them can be uniformly bounded by some constant, then the corresponding gradient norms must diminish over time, implying that 
$\|\grad_x f(x_t, y_t)\| \to 0$ and $\|\grad_y f(x_t, y_t)\| \to 0$. This behavior aligns precisely with the notion of approaching an $\epsilon$-stationary point (Definition \ref{def.stationary}) as defined earlier. 

To formalize this intuition, we begin by controlling the dual auxiliary sequence $\{v_t^y\}$, which upper bounds the accumulated squared dual gradient norms, namely,
$
\sum_{t=0}^{T-1}\|\grad_y f(x_t, y_t)\|^2 \le v_T^y,
$
as defined in line~\ref{alg.line3} of Algorithm~\ref{alg.RAGDA}.
\begin{lemma}\label{le.bound sqaure norm y and vTy}
Under Assumptions~\ref{ass.l-smooth}, \ref{ass.g-convex}, \ref{ass.optimal exist and stationary condition hold}, \ref{ass.manifold}, and \ref{ass.retraction}, for any $T \ge t_0 + 1$, where $t_0$ denotes the first iteration such that $(v_{t_0+1}^y)^\beta > c_1$, where $c_1 := \max \left\{ \frac{4 \eta^y \mu L_1}{\mu + L_1},\; \eta^y (\mu + L_1)\left(\zeta \bar{c} + \frac{2 G c_R}{\mu} \right) \right\}$,
 the following holds:
\begin{equation}\label{eq.vTy bound}
    \begin{cases}
        \sum_{t=0}^{T-1}\gamma_t\|\grad_yf(x_t,y_t)\|^2\le c_4 + c_3 (\eta^x)^2 \left( \frac{1 + \log v_T^x - \log v_0^x}{(v_0^x)^{2\alpha - \beta - 1}} \cdot \mathbf{1}_{2\alpha - \beta < 1} + \frac{(v_T^x)^{1-2\alpha + \beta}}{1-2\alpha + \beta} \cdot \mathbf{1}_{2\alpha - \beta < 1} \right)\\
        v_T^y\le \left(\frac{c_4}{\eta^y} + \frac{c_3}{\eta^y} (\eta^x)^2 \left( \frac{1 + \log v_T^x - \log v_0^x}{(v_0^x)^{2\alpha - \beta - 1}} \cdot \mathbf{1}_{2\alpha - \beta < 1} + \frac{(v_T^x)^{1-2\alpha + \beta}}{1-2\alpha + \beta} \cdot \mathbf{1}_{2\alpha - \beta < 1} \right)\right)^{\frac{1}{1-\beta}}
    \end{cases}
    \end{equation}
Here, $c_2 = 4(\mu + L_1) \left( \frac{1}{\mu^2} + \frac{\bar{c}\eta^y}{(v_{t_0}^y)^\beta} \right) c_1^{1/\beta}$, $
c_3 = (\mu + L_1) \left( \frac{2\kappa^2\bar{c}}{(v_0^y)^\beta} + \frac{(\mu + L_1)\kappa^2\bar{c}}{\eta^y \mu L_1} \right)$, $
c_4 = c_2 + \frac{\eta^y c_1^{\frac{1-\beta}{\beta}}}{1-\beta}$, $\kappa = L_1/\mu$ is the condition number and $\mathbf{1}_A$ is the indicator function of event A, taking value 1 on A and 0 otherwise.
\end{lemma}
Lemma~\ref{le.bound sqaure norm y and vTy} shows that, after up to $t_0$ iterations, the accumulated primal gradients effectively control the growth of the accumulated dual gradients and yield an explicit bound on $v_T^y$. More importantly, the bound reveals how the dual dynamics are influenced by the primal updates through the term involving $v_T^x$.
This interaction highlights the necessity of coordinating the primal and dual stepsizes, which motivates the use of $\max\{v_{t+1}^x, v_{t+1}^y\}$ in the primal update rule.

We next derive explicit upper bounds on the accumulated squared primal gradient norms \\
$
\sum_{t=0}^{T-1}\|\grad_x f(x_t, y_t)\|^2
$
by controlling the associated primal auxiliary sequence $v_T^x$, following the same rationale as for $v_T^y$. Since the adaptive primal stepsize depends on $\max\{v_T^x, v_T^y\}$, the analysis naturally splits into two complementary cases depending on the relative magnitude of $v_T^x$ and $v_T^y$. We first consider the regime $v_T^y \le v_T^x$, in which $v_T^x$ can be bounded in a self-consistent manner, with the influence of the dual sequence appearing only through bounded constants. 
The resulting uniform bound on $v_T^x$ is stated in the following lemma. 
\begin{lemma}\label{le.bound vTx case 1}
    Under the same conditions as Lemma \ref{le.bound sqaure norm y and vTy}, if $v_T^y\le v_T^x$ and $0<\beta<\alpha<1$, it follows
    \begin{equation}\label{eq.vTx bound case 1}
    \begin{split}
        v_T^x \leq &7v_0^x + 7 \left( \frac{2 \Delta \Phi}{\eta^x} \right)^{\frac{1}{1-\alpha}} + 7 \left( \frac{\left(4Gc_R+2L_{\Phi}\bar{c}\right)\eta^x e^{(1-\alpha)(1-\log v_0^x)/2}}{e(1-\alpha)(v_0^x)^{2\alpha-1}} \right)^{\frac{2}{1-\alpha}} \cdot \mathbf{1}_{2\alpha \geq 1}\\
   &+ 7 \left( \frac{\left(2Gc_R+L_{\Phi}\bar{c}\right)\eta^x}{1-2\alpha} \right)^{\frac{1}{\alpha}} \cdot \mathbf{1}_{2\alpha < 1}+ 7 \left( \frac{c_4 c_5}{\eta^x} \right)^{\frac{1}{1-\alpha}} \\
&+ 7 \left( \frac{2c_3 c_5 \eta^x e^{(1-\alpha)(1-\log v_0^x)/2}}{e(1-\alpha)\left(v_0^x\right)^{2\alpha-\beta-1}} \right)^{\frac{2}{1-\alpha}} \cdot \mathbf{1}_{2\alpha-\beta \ge 1} 
+ 7 \left( \frac{c_3 c_5 \eta^x}{1-2\alpha+\beta} \right)^{\frac{1}{\alpha-\beta}} \cdot \mathbf{1}_{2\alpha-\beta < 1}.
    \end{split}
\end{equation}
where $\Delta \Phi = \Phi(x_0) - \Phi^*$, $L_{\Phi}=L_1+\kappa L_1$ and  $c_5 = \frac{\eta^x \kappa^2}{\eta^y (v_{t_0}^y)^{\alpha - \beta}}$.
\end{lemma}
The following Lemma~\ref{le.bound vTx case 2} addresses the complementary case in which the dual auxiliary dominates, i.e., $v_T^x < v_T^y$. 
In this case, the bound on $v_T^x$ is obtained indirectly by leveraging the explicit control of $v_T^y$ established in Lemma~\ref{le.bound sqaure norm y and vTy}.
\begin{lemma}\label{le.bound vTx case 2}
    Under the same conditions as Lemma \ref{le.bound sqaure norm y and vTy}, if $v_T^x< v_T^y$, it follows
    \begin{equation}\label{eq.vTx bound case 2-2}
    \begin{split}
     v_T^x
\leq &2v_0^x +2\left[\left( \frac{2\Delta \Phi + c_4 c_5}{\eta^x (v_0^x)^{1-2\alpha+\beta}} + \left(2Gc_R+L_{\Phi}\bar{c}\right)\eta^x \left( \frac{e^{(1-2\alpha+\beta)(1-\log v_0^x)}}{e(1-2\alpha+\beta)(v_0^x)^{2\alpha-1}} \cdot \mathbf{1}_{2\alpha \geq 1} + \frac{1}{(1-2\alpha)(v_0^x)^\beta} \cdot \mathbf{1}_{2\alpha < 1} \right)\right.\right.\\
&\left.\left.+ \frac{c_3 c_5 \eta^x}{1-2\alpha+\beta} \right)  \left( \frac{c_4}{\eta^y(v_0^x)^{1-2\alpha+\beta}} + \frac{c_3 (\eta^x)^2}{\eta^y(1-2\alpha+\beta)} \right)^{\frac{\alpha}{1-\beta}}\right]^{ \frac{1}{1-(1-2\alpha+\beta)\left(1+\frac{\alpha}{1-\beta}\right)}}\cdot\mathbf{1}_{2\alpha-\beta<1} \\
&+ 2\left[ \left( \frac{2\Delta \Phi + c_4 c_5}{\eta^x (v_0^x)^{1/4}} + \frac{\left(8Gc_R+4L_{\Phi}\bar{c}\right)\eta^x e^{(1-\log v_0^x)/4}}{e(v_0^x)^{2\alpha-1}} + \frac{4c_3 c_5 \eta^x e^{(1-\log v_0^x)/4}}{e(v_0^x)^{2\alpha-\beta-1}} \right)\right. \\
&\left. \left( \frac{c_4}{\eta^y(v_0^x)^{\frac{(1-\beta)}{4\alpha}}} + \frac{4c_3 \alpha (\eta^x)^2 e^{(1-\beta)(1-\log v_0^x)/(4\alpha)}}{\eta^ye(1-\beta)(v_0^x)^{2\alpha-\beta-1}} \right)^{\frac{\alpha}{1-\beta}}\right]^2\cdot\mathbf{1}_{2\alpha-\beta\ge1}.
    \end{split}
\end{equation}
\end{lemma}
We have now established the key technical components of our convergence analysis.
Lemma~\ref{le.bound sqaure norm y and vTy} controls the accumulated dual gradient norms via the auxiliary sequence $v_T^y$, while Lemmas~\ref{le.bound vTx case 1} and~\ref{le.bound vTx case 2} provide bounds on the primal auxiliary sequence $v_T^x$ across two complementary cases.
Together, these results ensure uniform boundedness of both auxiliary sequences and hence control the accumulated squared norms of the primal and dual gradients, implying the convergence to an $\epsilon$-stationary point.
The main complexity result for the deterministic setting is stated in Theorem \ref{the.deterministic convergence}. We refer the reader to Appendix B for complete proofs.
\begin{theorem}\label{the.deterministic convergence}
    Under Assumptions~\ref{ass.l-smooth}, \ref{ass.g-convex}, \ref{ass.optimal exist and stationary condition hold}, \ref{ass.manifold}, and \ref{ass.retraction}, Algorithm \ref{alg.RAGDA} with deterministic gradients satisfies that for any $0 < \beta < \alpha < 1$, after $T$ iterations,
\[
\min_{t=0.\cdots,T-1}\{\|\grad_x f(x_t, y_t)\| +  \|\grad_y f(x_t, y_t)\|\} \leq \mathcal{O}\left(\frac{1}{T^\frac{1}{2}}\right).
\]
\end{theorem}
\begin{remark}
    The convergence result in Theorem~\ref{the.deterministic convergence} implies that
for any $\epsilon>0$, choosing
$T = \mathcal{O}(\epsilon^{-2})$ guarantees that
there exists an iteration index $t \in \{0,\ldots,T-1\}$ such that
$(x_t,y_t)$ is an $\epsilon$-stationary point as defined in
Definition~\ref{def.stationary}.
\end{remark}
\section{Riemannian stochastic adaptive gradient descent ascent}\label{sec.RSAGDA}
In this section, we extend the RAGDA to the stochastic setting, aiming to solve the following stochastic minimax problem:
\begin{equation}
\min_{x \in \cM_x}\max_{y\in\cM_y} f(x,y) = \mathbb{E}_{\xi \sim \cD}[f(x,y;\xi)],
\end{equation}
where $\cM_x$ and $\cM_y$ are complete Riemannian manifolds, and $\xi$ is a random variable drawn from an unknown distribution $\mathcal{D}$. While the method retains the same overall structure as in the deterministic setting and remains parameter-free, the presence of stochasticity necessitates new proof techniques to establish convergence guarantees.
\subsection{Riemannian stochastic adaptive gradient descent ascent method: RSAGDA}
Compared to the deterministic setting in the previous section, our stochastic method remains structurally similar, with the only essential change being the replacement of exact gradient oracles by stochastic ones. The complete procedure is presented in Algorithm~\ref{alg.RSAGDA}. Although this modification may appear minor, it introduces significant technical challenges for convergence analysis due to the inherent randomness.
\begin{algorithm}[htbp]
\caption{RSAGDA (Riemannian stochastic adaptive gradient descent ascent method)}
\label{alg.RSAGDA}
\KwIn{$(x_0, y_0)$, $v^x_0 > 0$, $v^y_0 > 0$, $\eta^x > 0$, $\eta^y > 0$, $\alpha > 0$, $\beta > 0$ with $\alpha > \beta$}

\For{$t = 0, 1, 2, \dots$}{
    Sample i.i.d. $\xi^x_t$, $\xi^y_t$ from $\mathcal{D}$\;
    Let $g^x_t = \grad_x f(x_t, y_t; \xi^x_t)$, $g^y_t = \grad_y f(x_t, y_t; \xi^y_t)$\;
    $v^x_{t+1} = v^x_t + \|g^x_t\|^2$, $v^y_{t+1} = v^y_t + \|g^y_t\|^2$\;
    $\eta_t = \frac{\eta^x}{\max\{v^x_{t+1}, v^y_{t+1}\}^\alpha}$\;
    $x_{t+1} = \Retr_{x_t}(-\eta_t g^x_t)$\;
    $\gamma_t = \frac{\eta^y}{(v^y_{t+1})^\beta}$\;
    $y_{t+1} = \Retr_{y_t}(\gamma_t g^y_t)$\;
}
\end{algorithm}

For comparison, AdaRHD \citep{shi2025adaptive} also employs adaptive step sizes for Riemannian minimax problems. However, their analysis is restricted to the deterministic setting. In contrast, our method also applies to the stochastic setting and establishes rigorous theoretical convergence guarantees.
\subsection{Convergence analysis}\label{subsec.stochastic convergence}
To address the stochasticity in the problem, we impose some additional yet standard assumptions on the stochastic gradients to facilitate the convergence analysis. In particular, we assume they are unbiased and their norms are uniformly bounded, which in turn implies a uniform bound on the variance.
\begin{assumption}\label{ass.grad bound}
    The stochastic Riemannian gradients are unbiased estimators of the true Riemannian gradients, i.e., $\E_{\xi}[\grad_x f(x,y;\xi)]=\grad_xf(x,y)$, $\E_{\xi}[\grad_y f(x,y;\xi)]=\grad_yf(x,y)$. Moreover, there exists a constant $G > 0$ such that for all $(x, y)$ and $\xi$, the stochastic gradients are uniformly bounded: $\max\{\|\grad_x f(x,y;\xi)\|,\|\grad_y f(x,y;\xi)\|\}\le G$.
\end{assumption}
The unbiasedness assumption is standard in the stochastic optimization literature \citep{ghadimi2013stochastic, bonnabel2013stochastic}. Moreover, the boundedness of stochastic gradients is commonly assumed in many adaptive methods \citep{reddi2019convergence, defossez2020simple, zou2019sufficient, levy2021storm+, li2022tiada}. 

The central idea of our stochastic analysis follows the same principle as the deterministic setting—controlling the growth of the accumulated squared norms of the gradients. However, since the updates are stochastic, we analyze these quantities in expectation, i.e., $\E\big[\sum_{t=0}^{T-1}\|\grad_x f(x_t, y_t)\|^2\big]$ and $\E\big[\sum_{t=0}^{T-1}\|\grad_y f(x_t, y_t)\|^2\big]$. This requires substantially different bounding techniques to establish convergence. Firstly, we provide a key lemma that forms the foundation for realizing this idea. 

\begin{lemma}\label{le.bound conclusion}
Suppose Assumptions~\ref{ass.l-smooth}, \ref{ass.g-convex}, \ref{ass.optimal exist and stationary condition hold}, \ref{ass.manifold}, \ref{ass.retraction}, \ref{ass.grad bound} hold and $0< 2\beta\le \alpha<1$. Then for any constant \(C\), the following bound holds
\begin{align*}
&\mathbb{E} \left[ \sum_{t=0}^{T-1} (f(x_t,y_t^*) - f(x_t,y_t)) \right] \\
\leq &\frac{2{\kappa}^2 (\eta^x)^2}{\mu (\eta^y)^2 \, C^{2\alpha - 2\beta}} \mathbb{E} \left[ \sum_{t=0}^{T-1} \|\grad_x f(x_t,y_t)\|^2 \right] + \frac{(\zeta\bar{c}+\frac{2G}{\mu}c_R)\eta^y}{2(1-\beta)} \mathbb{E} \left[ (v_T^y)^{1-\beta} \right] \\
& + \left( \frac{1}{\mu} + \frac{\eta^y}{(v_0^y)^{\beta}} \right) \frac{4{\kappa} \eta^x G^2}{\eta^y (v_0^y)^{\alpha}} \mathbb{E} \left[ (v_T^y)^{\beta} \right] \\
& + \frac{{\kappa}^2\bar{c} \left( \mu \eta^y C^{\beta} + 2C^{2\beta} \right) (\eta^x)^2}{2\mu (\eta^y)^2} \mathbb{E} \left[ \frac{1 + \log v_T^x - \log v_0^x}{(v_0^x)^{2\alpha - 1}} \cdot \mathbf{1}_{\alpha \geq 0.5} + \frac{(v_T^x)^{1-2\alpha}}{1-2\alpha} \cdot \mathbf{1}_{\alpha < 0.5} \right] \\
& + \left(\zeta\bar{c}{\kappa}^2 + \frac{(2\kappa\sqrt{\bar{c}}(v_0^y)^{\alpha}+\kappa Gc_R\eta^x)^2(\eta^x)^{\beta/\alpha}}{\mu \eta^y (v_0^y)^{2\alpha}}\right)
\frac{(\eta^x)^{2-\beta/\alpha}}{2(1 - \alpha)\eta^y (v_0^y)^{\alpha - 2\beta}} \mathbb{E} \left[ (v_T^x)^{1-\alpha} \right] \\
& + \frac{(v_0^y)^{\beta} G^2}{2\mu^2 \eta^y} + \frac{(2\beta G)^{\frac{1}{1-\beta} + 2} G^2}{4\mu^{\frac{1}{1-\beta} + 3} (\eta^y)^{\frac{1}{1-\beta} + 2} (v_0^y)^{2-2\beta}}.
\end{align*}
\end{lemma}
Lemma~\ref{le.bound conclusion} establishes a key inequality that upper bounds the expected cumulative dual suboptimality
$\E\big[\sum_{t=0}^{T-1}(f(x_t,y_t^*)-f(x_t,y_t))\big]$
in terms of the expected accumulated primal gradient norms and several moments of the auxiliary sequences $v_T^x$ and $v_T^y$. Using the smoothness and concavity of $f$ with respect to $y$, the inequality in Lemma~\ref{le.bound conclusion} can be converted into a bound on the accumulated squared dual gradient norms, leading to Lemma~\ref{le.bound EvTy}.
\begin{lemma}\label{le.bound EvTy}
    Under the same conditions as Lemma \ref{le.bound conclusion}, the following bound holds
\begin{align*}
&\mathbb{E}\left[ \sum_{t=0}^{T-1}\|\grad_y f(x_t,y_t)\|^2 \right] \\
\leq &\frac{4\kappa^3\left(\eta^x\right)^2}{\left(\eta^y\right)^2C^{2\alpha-2\beta}} \mathbb{E} \left[ \sum_{t=0}^{T-1} \|\grad_x f(x_t,y_t)\|^2 \right] + \frac{L_1(\zeta\bar{c}+\frac{2G}{\mu}c_R)\eta^y}{(1-\beta)} \mathbb{E} \left[ (v_T^y)^{1-\beta} \right] \\
& + \left( \frac{1}{\mu} + \frac{\eta^y}{(v_0^y)^{\beta}} \right) \frac{8L_1{\kappa} \eta^x G^2}{\eta^y (v_0^y)^{\alpha}} \mathbb{E} \left[ (v_T^y)^{\beta} \right] \\
& + \frac{{\kappa}^3\bar{c} \left( \mu \eta^y C^{\beta} + 2C^{2\beta} \right) (\eta^x)^2}{ (\eta^y)^2} \mathbb{E} \left[ \frac{1 + \log v_T^x - \log v_0^x}{(v_0^x)^{2\alpha - 1}} \cdot \mathbf{1}_{\alpha \geq 0.5} + \frac{(v_T^x)^{1-2\alpha}}{1-2\alpha} \cdot \mathbf{1}_{\alpha < 0.5} \right] \\
& + \left(\zeta\bar{c}{\kappa}^2 + \frac{(2\kappa\sqrt{\bar{c}}(v_0^y)^{\alpha}+\kappa Gc_R\eta^x)^2(\eta^x)^{\beta/\alpha}}{\mu \eta^y (v_0^y)^{2\alpha}}\right)
\frac{L_1(\eta^x)^{2-\beta/\alpha}}{(1 - \alpha)\eta^y (v_0^y)^{\alpha - 2\beta}} \mathbb{E} \left[ (v_T^x)^{1-\alpha} \right] \\
& + \frac{\kappa(v_0^y)^{\beta} G^2}{\mu \eta^y} + \frac{2L_1(2\beta G)^{\frac{1}{1-\beta} + 2} G^2}{4\mu^{\frac{1}{1-\beta} + 3} (\eta^y)^{\frac{1}{1-\beta} + 2} (v_0^y)^{2-2\beta}}.
\end{align*}
\end{lemma}
We next turn to the primal side and derive a corresponding bound on the accumulated squared primal gradient norms in the following lemma.
\begin{lemma}\label{le.bound EvTx}
    Under the same conditions as Lemma \ref{le.bound conclusion}, the following bound holds
    \begin{align*}
    &\mathbb{E} \left[ \sum_{t=0}^{T-1} \left\| \grad_x f(x_t, y_t) \right\|^2 \right]\\
\leq &
4\mathbb{E} \left[ \frac{\Delta \Phi}{\eta^x} \max \left\{ (v_T^x)^\alpha, (v_T^y)^\alpha \right\} \right]
+ \frac{2(L_{\Phi}\bar{c}+2Gc_R) \eta^x}{1 - \alpha} \mathbb{E} \left[ (v_T^x)^{1 - \alpha} \right]\\
& + \frac{2L_1\kappa (\zeta\bar{c}+\frac{2G}{\mu}c_R)\eta^y}{1 - \beta} \mathbb{E} \left[ (v_T^y)^{1 - \beta} \right]
+ \left( \frac{1}{\mu} + \frac{\eta^y}{(v_0^y)^\beta} \right) \cdot \frac{16L_1\kappa^2\eta^x G^2}{\eta^y (v_0^y)^\alpha} \mathbb{E} \left[ (v_T^y)^\beta \right]\\
&+ \frac{2\kappa^4\bar{c} ( \eta^y C^\beta + 2C^{2\beta}) (\eta^x)^2}{\mu(\eta^y)^2}\mathbb{E} \left[ \frac{1 + \log v_T^x - \log v_0^x}{(v_0^x)^{2\alpha - 1}} \cdot \mathbb{I}_{\alpha \geq 0.5}
+ \frac{(v_T^x)^{1 - 2\alpha}}{1 - 2\alpha} \cdot \mathbb{I}_{\alpha < 0.5} \right]\\
&+ \left(\zeta\bar{c}{\kappa}^2 + \frac{(2\kappa\sqrt{\bar{c}}(v_0^y)^{\alpha}+\kappa Gc_R\eta^x)^2(\eta^x)^{\beta/\alpha}}{\mu \eta^y (v_0^y)^{2\alpha}}\right)
\frac{2L_1\kappa(\eta^x)^{2-\beta/\alpha}}{(1 - \alpha)\eta^y (v_0^y)^{\alpha - 2\beta}} \mathbb{E} \left[ (v_T^x)^{1 - \alpha} \right]
+ \frac{2\kappa^2 (v_0^y)^\beta G^2}{\mu \eta^y}\\
&+ \frac{L_1 \kappa \left( 2 \beta G \right)^{\frac{1}{1 - \beta} + 2} G^2}
{ \mu^{\frac{1}{1 - \beta} + 3} \left( \eta^y \right)^{\frac{1}{1 - \beta} + 2} \left( v_0^y \right)^{2 - 2 \beta}}.
\end{align*}
\end{lemma}
The upper bound provided in Lemma~\ref{le.bound EvTx} is derived by exploiting the descent property of the function~$\Phi$. It consists of the telescoping decrease $\Delta \Phi = \Phi(x_0) - \Phi^*$ on $\Phi$, auxiliary sequences $v_T^x$ and $v_T^y$ that capture the cumulative growth of the primal and dual gradients, and the terms induced by using Lemma~\ref{le.bound conclusion}.

With the above bounds on the accumulated squared primal and dual gradient norms in expectation, we are able to establish the key ingredients required for the stochastic convergence analysis and prove the convergence guarantees of our method in the stochastic setting.
The detailed proofs of the convergence results in this section can be found in Appendix~C.

\begin{theorem}\label{the.stochastic convergence}
Suppose Assumptions~\ref{ass.l-smooth}, \ref{ass.g-convex}, \ref{ass.optimal exist and stationary condition hold}, \ref{ass.manifold}, \ref{ass.retraction}, and \ref{ass.grad bound} hold. Then, for any \( 0 < 2\beta \le \alpha < 1 \), Algorithm \ref{alg.RSAGDA} with stochastic gradient oracles guarantees that after \( T \) iterations,
\[
\min_{t=0,\cdots,T-1}\left\{ \mathbb{E} \left[ \|\grad_x f(x_t, y_t)\|^2 + \|\grad_y f(x_t, y_t)\|^2 \right]\right\} \leq \mathcal{O}\left(T^{\alpha - 1} + T^{-\alpha} + T^{\beta - 1} + T^{-\beta}\right).
\]
In particular, by setting \( \alpha = 2/3 \) and \( \beta = 1/3 \), we obtain the convergence rate
\[
\min_{t=0,\cdots,T-1}\left\{ \mathbb{E} \left[ \|\grad_x f(x_t, y_t)\|^2 + \|\grad_y f(x_t, y_t)\|^2 \right]\right\} \leq \mathcal{O}\left(\frac{1}{T^{1/3}}\right).
\]
\end{theorem}
\begin{remark}
        The convergence result in Theorem~\ref{the.stochastic convergence} implies that
    for any $\epsilon>0$, choosing
    $T = \mathcal{O}(\epsilon^{-6})$ guarantees that
    there exists an iteration index $t \in \{0,\ldots,T-1\}$ such that
    $(x_t,y_t)$ is an $\epsilon$-stationary point as defined in
    Definition~\ref{def.stationary}.
\end{remark}
\subsection{Improved convergence analysis}\label{subsec.improved convergence}
In the previous subsection, we have established the convergence guarantee for Algorithm~\ref{alg.RSAGDA} under mild conditions, including first-order smoothness and bounded stochastic gradients. Under these assumptions, the best achievable convergence rate is $\mathcal{O}(T^{-1/3})$ in terms of the squared norm of the gradient after $T$ iterations. In this subsection, we show that this convergence rate can be further improved to $\mathcal{O}(T^{-1/2})$ by leveraging second-order smoothness  of the objective function. This additional assumption allows for a more precise local approximation of the function, which we incorporate into the convergence analysis to achieve a sharper bound. Importantly, the resulting improvements are obtained by refining the previous arguments, without requiring substantial changes to the overall proof structure. The detailed convergence analysis under this setting is provided in Appendix D.

Specifically, we assume that the Riemannian Hessian and cross derivatives of the objective with respect to \( y \) are Lipschitz continuous, as stated below.
\begin{assumption}\label{ass.second smooth}
    We assume that the Riemannian cross derivatives $\grad_{xy}f(x,y)$ and Hessian $\hess_y f(x,y)$ are $L_2$-Lipschitz continuous. Specifically, for any $(x_1, y_1), (x_2, y_2) \in \cM_x \times \cM_y$, we have
    \begin{align*}
        \|\Gamma_{y_1}^{y_2}\grad_{yx}f(x,y_1)-\grad_{yx}f(x,y_2)\|_{y_2}\le L_2d_{\cM_y}(y_1,y_2),
    \end{align*}
    \begin{align*}
        \|\grad_{yx}f(x_1,y)-\grad_{yx}f(x_2,y)\Gamma_{x_1}^{x_2}\|_{y}\le L_2d_{\cM_x}(x_1,x_2) ,
    \end{align*}
    and 
    \begin{align*}
        \|\Gamma_{y_1}^{y_2}\hess_y f(x,y_1)\Gamma_{y_2}^{y_1} - \hess_y f(x,y_2)\|_{y_2}\le L_2d_{\cM_y}(y_1,y_2).
    \end{align*}
    \begin{align*}
        \|\hess_y f(x_1,y) - \hess_y f(x_2,y)\|_{y}\le L_2d_{\cM_x}(x_1,x_2).
    \end{align*}
\end{assumption}
This assumption is commonly used in Riemannian bilevel optimization \citep{dutta2024riemannian, han2024framework, shi2025adaptive}. With this assumption, the Lipschitz smoothness property of $y^*(x)$ (Lemma D.1 in Appendix D) can be established, which enables a more accurate local approximation and lays the foundation for improved convergence guarantees.

Theorem \ref{the.stochastic convergence} indicates that the convergence rate is dependent on $\alpha$ and $\beta$ with their values constrained by the condition $0 < 2\beta \le \alpha < 1$ in Lemma \ref{le.bound conclusion}.
Now, through a more accurate local approximation in our analysis, this requirement can be relaxed to $0 < \beta \le \alpha < 1$. This relaxation broadens the feasible range of $(\alpha, \beta)$. The following lemma states the result under the relaxed condition.
\begin{lemma}\label{le.imporve bound conclusion}
Suppose the assumptions in Lemma \ref{le.bound conclusion} and Assumption \ref{ass.second smooth} hold, and $0< \beta\le \alpha<1$. Then for any constant \(C\), the following bound holds
\begin{align*}
&\mathbb{E} \left[ \sum_{t=0}^{T-1} (f(x_t,y_t^*) - f(x_t,y_t)) \right] \\
\leq &\frac{4{\kappa}^2 (\eta^x)^2}{\mu (\eta^y)^2 \, C^{2\alpha - 2\beta + \delta}} \mathbb{E} \left[ \sum_{t=0}^{T-1} \|\grad_x f(x_t,y_t)\|^2 \right] + \frac{(\zeta\bar{c}+\frac{2G}{\mu}c_R)\eta^y}{2(1-\beta)} \mathbb{E} \left[ (v_T^y)^{1-\beta} \right] \\
& + \left( \frac{1}{\mu} + \frac{\eta^y}{(v_0^y)^{\beta}} \right) \frac{4{\kappa} \eta^x G^2}{\eta^y (v_0^y)^{\alpha}} \mathbb{E} \left[ (v_T^y)^{\beta} \right] \\
& + \frac{{\kappa}^2\bar{c} \left( \mu \eta^y C^{\beta} + 2C^{2\beta} \right) (\eta^x)^2}{2\mu (\eta^y)^2} \mathbb{E} \left[ \frac{1 + \log v_T^x - \log v_0^x}{(v_0^x)^{2\alpha - 1}} \cdot \mathbf{1}_{\alpha \geq 0.5} + \frac{(v_T^x)^{1-2\alpha}}{1-2\alpha} \cdot \mathbf{1}_{\alpha < 0.5} \right] \\
& + \left( \zeta\bar{c}{\kappa}^2 + \frac{(L_y\bar{c}+2c_R{\kappa})^2 G^2 (\eta^x)^2}{\mu \eta^y (v_0^y)^{2\alpha - \beta}} \right) \frac{(\eta^x)^2}{2(1-\alpha) \eta^y (v_0^y)^{\alpha - \beta}} \mathbb{E} \left[ (v_T^x)^{1-\alpha} \right] \\
& + \frac{(v_0^y)^{\beta} G^2}{2\mu^2 \eta^y} + \frac{(2\beta G)^{\frac{1}{1-\beta} + 2} G^2}{4\mu^{\frac{1}{1-\beta} + 3} (\eta^y)^{\frac{1}{1-\beta} + 2} (v_0^y)^{2-2\beta}}.
\end{align*}
for any $\delta\le \log(2)/\log(TG)$.
\end{lemma}
The above new result also introduces an additional term $\delta$. It allows us to adapt our previous analysis under the relaxed condition $0 < \beta \le \alpha < 1$, while maintaining the convergence rate dictated by the choice of $\alpha$ and $\beta$, namely $\mathcal{O}\left(T^{\alpha - 1} + T^{-\alpha} + T^{\beta - 1} + T^{-\beta}\right)$ as stated in Theorem~\ref{the.stochastic convergence}.
\begin{theorem}\label{the.improved stochastic convergence}
Suppose Assumptions~\ref{ass.optimal exist and stationary condition hold}, \ref{ass.manifold}, \ref{ass.retraction}, \ref{ass.l-smooth}, \ref{ass.g-convex}, \ref{ass.grad bound} and \ref{ass.second smooth} hold. Then, for any \( 0 < \beta \le \alpha < 1 \), Algorithm \ref{alg.RSAGDA} with stochastic gradient oracles guarantees that after \( T \) iterations,
\[
\min_{t=0,\cdots,T-1}\left\{ \mathbb{E} \left[ \|\grad_x f(x_t, y_t)\|^2 + \|\grad_y f(x_t, y_t)\|^2 \right]\right\} \leq \mathcal{O}\left(T^{\alpha - 1} + T^{-\alpha} + T^{\beta - 1} + T^{-\beta}\right).
\]
In particular, by setting \( \alpha = 1/2 \) and \( \beta = 1/2\), we obtain the convergence rate
\[
\min_{t=0,\cdots,T-1}\left\{ \mathbb{E} \left[ \|\grad_x f(x_t, y_t)\|^2 + \|\grad_y f(x_t, y_t)\|^2 \right]\right\} \leq \mathcal{O}\left(\frac{1}{T^{1/2}} \right).
\]
\end{theorem}
It is worth mentioning that the condition
$0 < \beta \le \alpha < 1$ in Theorem~\ref{the.improved stochastic convergence} is similar but slightly weaker than
the one used in Euclidean space min-max optimization \citep{li2022tiada}. In particular, \cite{li2022tiada} requires 
$0 < \beta \le \alpha < 1$ and has established convergence rate of
$\mathcal{O}(T^{-1/4+\tilde{\epsilon}})$ (for a small $\tilde{\epsilon}$).
Consequently, Theorem~\ref{the.improved stochastic convergence} even
sharpens the result of \cite{li2022tiada} in Euclidean space
to the exact $\mathcal{O}(T^{-1/4})$.

    \begin{remark}
        The convergence result in Theorem~\ref{the.improved stochastic convergence} implies that
    for any $\epsilon>0$, choosing
    $T = \mathcal{O}(\epsilon^{-4})$ guarantees that
    there exists an iteration index $t \in \{0,\ldots,T-1\}$ such that
    $(x_t,y_t)$ is an $\epsilon$-stationary point as defined in
    Definition~\ref{def.stationary}.
    \end{remark}
\section{Numerical experiments}\label{sec.experiments}
In this section, we compare our methods on several benchmark problems. The experiments are performed on a 13-inch MacBook Pro (2022) with an Apple M2 processor and 8 GB RAM. The experiments comparing RAGDA are implemented in MATLAB using the Manopt toolbox \citep{boumal2014manopt}. Our method uses both the exponential map and a retraction, whereas all baselines employ the exponential map only. The compared algorithms are as follows:
\begin{itemize}
    \item RHM \citep{han2023riemannian}: the Riemannian Hamiltonian method with fixed stepsize.
    \item RCON \citep{han2023riemannian}: the Riemannian consensus method with fixed stepsize.
    \item RGDA \citep{jordan2022first}: the Riemannian gradient descent ascent method with same stepsize for the min and max variables.
    \item RCEG \citep{jordan2022first, zhang2023sion}: the Riemannian corrected extragradient method.
    \item TSRGDA \citep{huang2023gradient}: the Riemannian gradient descent ascent method with different stepsizes for the min and max variables, i.e., timescale separated. The method name in the original paper is also RGDA. We call it TSGDA here to distinguish it from the method in \cite{jordan2022first}. 
    \item AdaRHD \citep{shi2025adaptive}: the Riemannian adaptive bilevel solver for minimax optimization (Algorithm 4 in \cite{shi2025adaptive}). To enhance its performance, we incorporate separate step size parameters, $\eta^x$ and $\eta^y$, for updating the primal and dual variables, respectively, similar to our proposed method. Additionally, we replace the original inner-loop stopping criterion with a fixed number of iterations.
\end{itemize}
To compare our proposed RSAGDA method, we conduct experiments implemented in PyTorch with Geoopt \citep{kochurov2020geoopt}, using retraction-based updates, against the following baselines
\begin{itemize}
    \item RSGDA \citep{huang2023gradient}: the stochastic Riemannian gradient descent ascent method.
    \item Acc-RSGDA \citep{huang2023gradient}: the stochastic Riemannian gradient descent ascent method with momentum.
\end{itemize}
Note that the methods in \cite{huang2023gradient} were not originally designed or analyzed for the general minimax problems \eqref{eq.deterministic problem} and \eqref{eq.stochastic problem}, where $\cM_y$ is a convex subset of Euclidean space. We replace the projection operation with a retraction or exponential map, as described above, when updating the maximization variable. 

All experiments are terminated when either the desired accuracy is achieved or the maximum number of iterations is reached.
\subsection{Regularized Robust maximum likelihood estimation}\label{subsec.like}
In this subsection, we compare RAGDA with baselines on the robust maximum likelihood estimation problem \eqref{eq.like}. 
To better align with our proposed framework, we incorporate a regularization term $c d^2(S, I)$ into the original formulation, leading to the following modified problem:
\begin{equation}\label{eq.regularize like}
    \min_{x\in \mathbb{S}^{d}}\max_{Y\in\mathcal{P}(d+1)}-\frac{n}{2}\log(\det(Y))-\frac{1}{2}\sum_{i=1}^n([a_i,1]^\top-x)^\top Y^{-1}([a_i,1]^\top-x)+cd^2(Y,I).
\end{equation}
where c is a tunable regularization parameter and if $c
<0$, the problem is nonconvex–strongly-geodesically-concave minimax problem. In our experiments, we set $c = -5$, $d = 30$, and $n = 100$. The data $a_i$ are sampled from the standard Gaussian distribution $N(0,I)$.

We report the gradient norm with respect to both wall-clock time and iteration count in Figure~\ref{fig.like}. 
The hyperparameters for different algorithms are configured as follows: 
RHM uses a step size parameter $\eta = 0.01$;
RCON uses a single step size parameter $\eta = 0.01$ and $\gamma=15$;  
RGDA usesadaptive gradient descent ascent, achieving a step size parameter $\eta = 0.0005$; RCEG use a step size parameter $\eta = 0.01$; TSGDA uses step size $\eta^x = 0.02$ for the primal variable and step size $\eta^y = 0.4$ for the dual variable; AdaRHD uses step size $\eta^x = 1$ for the primal variable and step size $\eta^y = 2.5$ for the dual variable, and the inner loop iteration number is $20$; RAGDA uses step size $\eta^x = 0.5$ for the primal variable and step size $\eta^y = 5$ for the dual variable, $v_0^x=v_0^y=10^{-6}$, and $\alpha=\beta=0.5$. 
The results demonstrate that our method converges more rapidly and attains higher accuracy compared to the baselines.
\begin{figure}[htbp]
    \centering
    \begin{subfigure}[b]{0.45\textwidth}
        \includegraphics[width=\textwidth]{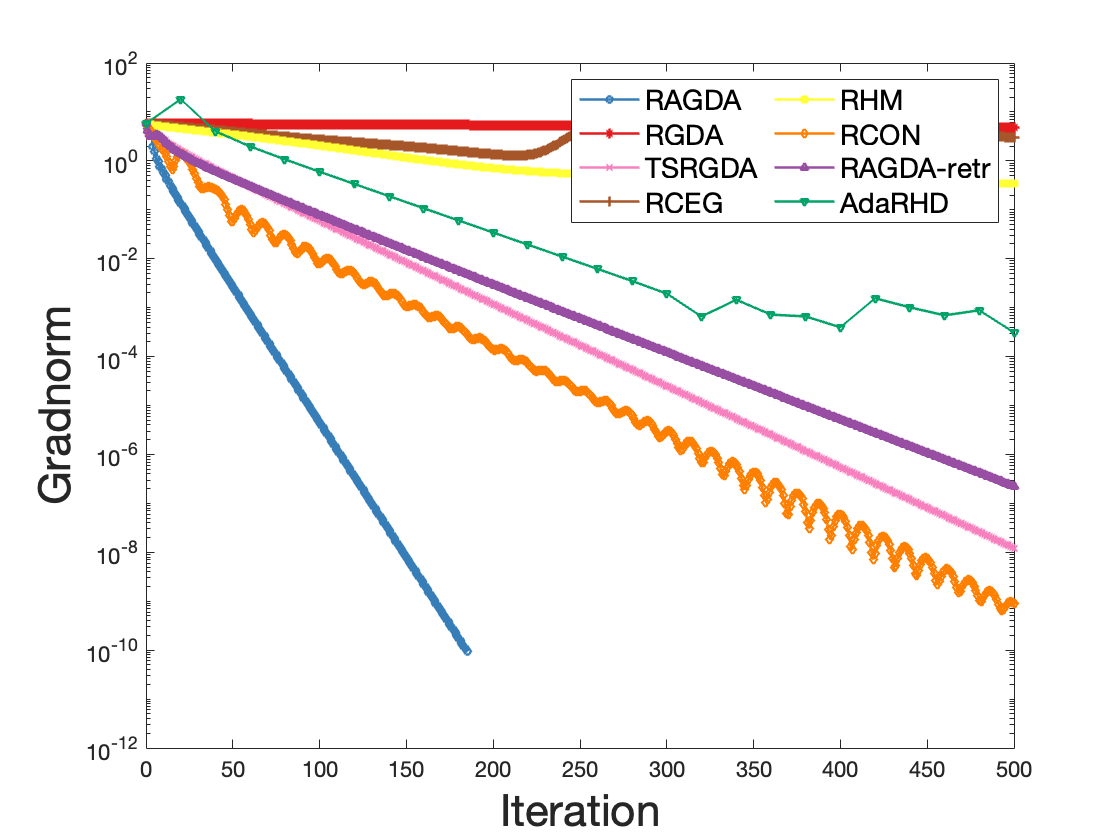}
        \caption{iteration-gradnorm}
        \label{subfig.like-iteration-gradnorm}
    \end{subfigure}
    \hfill
    \begin{subfigure}[b]{0.45\textwidth}
        \includegraphics[width=\textwidth]{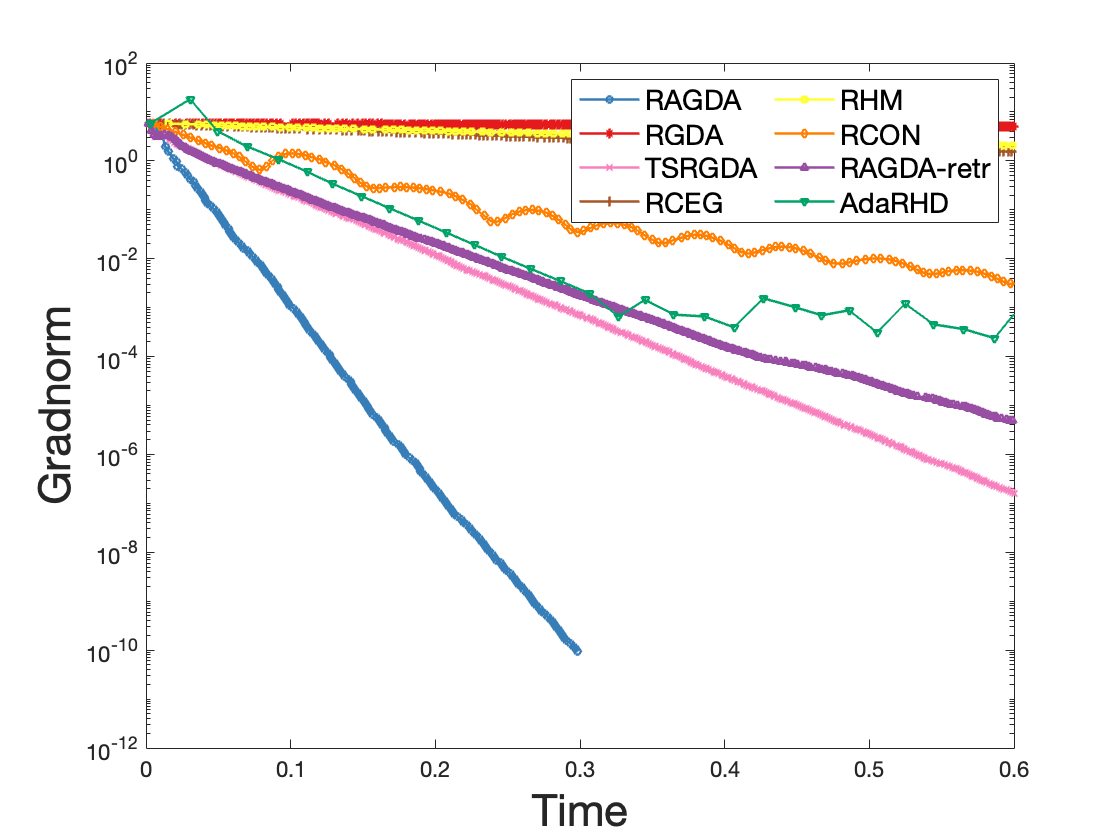}
        \caption{time-gradnorm}
        \label{subfig.like-time-gradnorm}
    \end{subfigure}
    \caption{Experiments on Regularized Robust maximum likelihood estimation \eqref{eq.regularize like}.}
    \label{fig.like}
\end{figure}
\subsection{Robust training of neural networks with orthonormal weights}\label{subsec.robust neural network}
To evaluate the performance of our stochastic method RSAGDA, we conduct experiments on adversarially robust training of deep neural networks \eqref{eq.robust nn} with universal orthonormal perturbation \citep{moosavi2017universal}, which involves solving a Riemannian optimization problem over the Stiefel manifold.

In our experiments, we train a three-layer neural network on the MNIST \citep{lecun2002gradient}, FashionMNIST \citep{xiao2017fashion}, and CIFAR-10 \citep{krizhevsky2009learning} datasets for image classification. The network comprises two hidden layers with orthonormal weight matrices and a fully connected output layer. The input and hidden layer dimensions are 784, 256, and 256, respectively. For MNIST, FashionMNIST, and CIFAR-10, we set $r = 1$, $0.5$, and $0.25$, respectively, and train for 30 epochs with a batch size of 256.
We evaluate three algorithms: RSGDA, Acc-RSGDA, and RSAGDA on the MNIST, FashionMNIST, and CIFAR-10 datasets. RSGDA uses step size $\eta^x = \eta^y = 0.02$ for both primal variable and dual variable across all three datasets.
RSAGDA uses step size $\eta^x = \eta^y = 2.0$, $v_0^x=v_0^y=10^{-6}$, and $\alpha=\beta=0.5$ for both primal variable and dual variable across all three datasets.
Acc-RSGDA uses dataset-specific parameters: for MNIST, $b = 0.2$, $m = 8$, $c_1 = c_2 = 64$, and $\lambda = \gamma = 1$; for FashionMNIST, $b = 0.5$, $m = 8$, $c_1 = c_2 = 64$, and $\lambda = \gamma = 1$; and for CIFAR-10, $b = 0.45$, $m = 16$, $c_1 = c_2 = 64$, and $\lambda = \gamma = 1$.
Training accuracy, test accuracy, and loss across epochs for all compared methods are shown in Figure~\ref{fig.mnist}, Figure~\ref{fig.fashionmnist}, Figure~\ref{fig.cifar10}.


\begin{figure}[htbp]
    \centering
    \begin{subfigure}[b]{0.32\textwidth}
        \includegraphics[width=\textwidth]{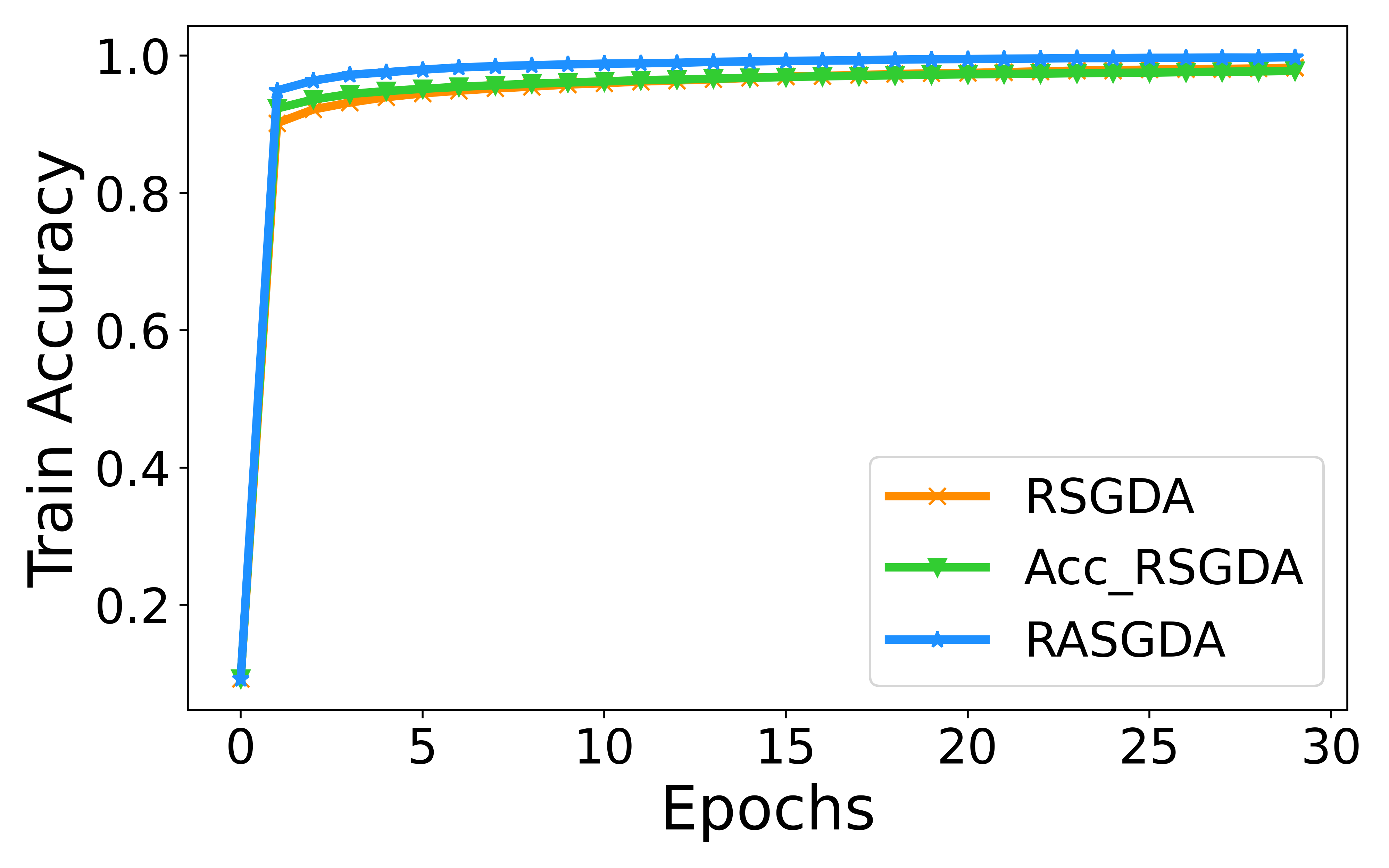}
        \caption{Train accuracy over epochs}
        \label{subfig.mnist train}
    \end{subfigure}
    \hfill
    \begin{subfigure}[b]{0.32\textwidth}
        \includegraphics[width=\textwidth]{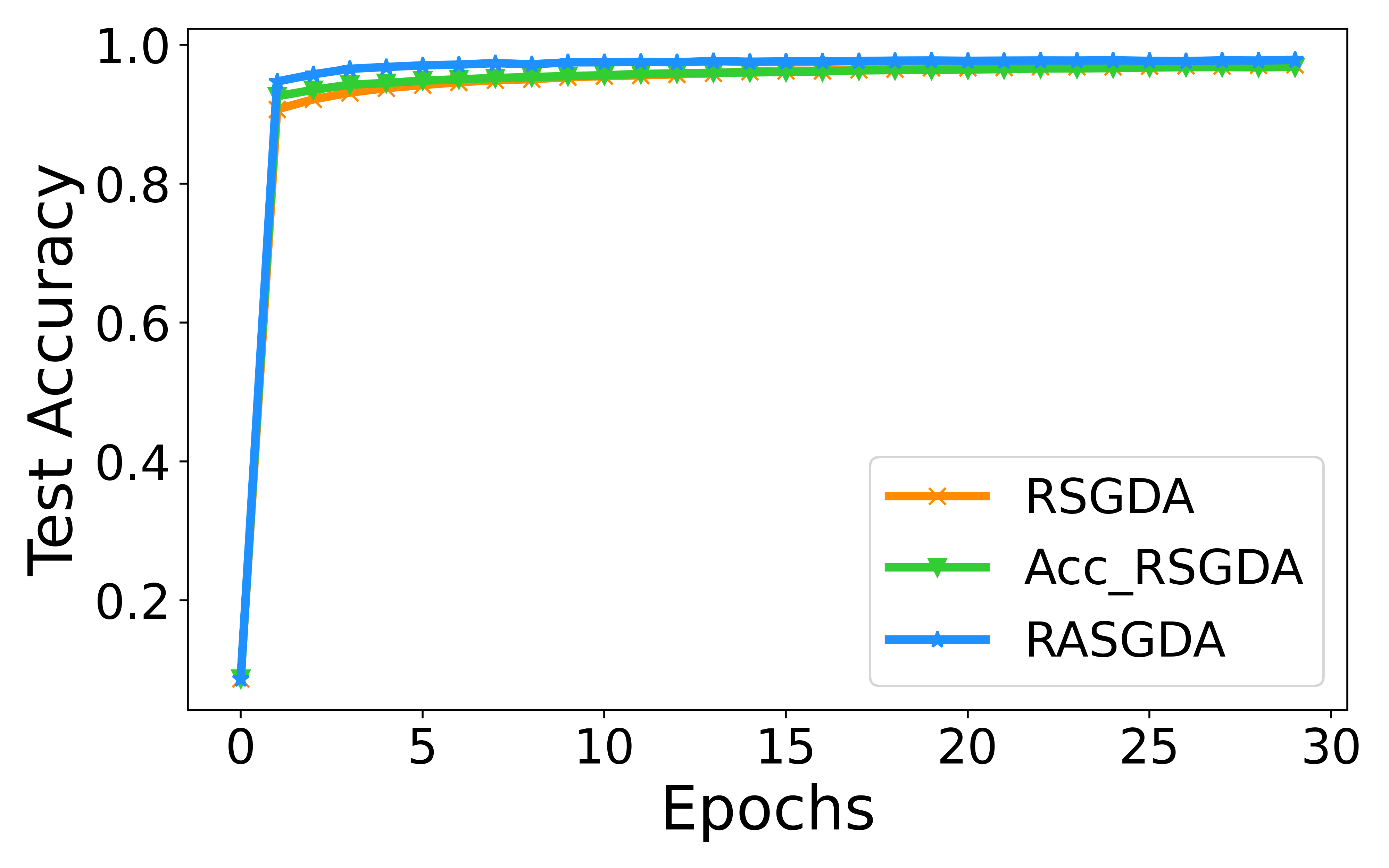}
        \caption{Test accuracy over epochs}
        \label{subfig.mnist test}
    \end{subfigure}
    \hfill
    \begin{subfigure}[b]{0.32\textwidth}
        \includegraphics[width=\textwidth]{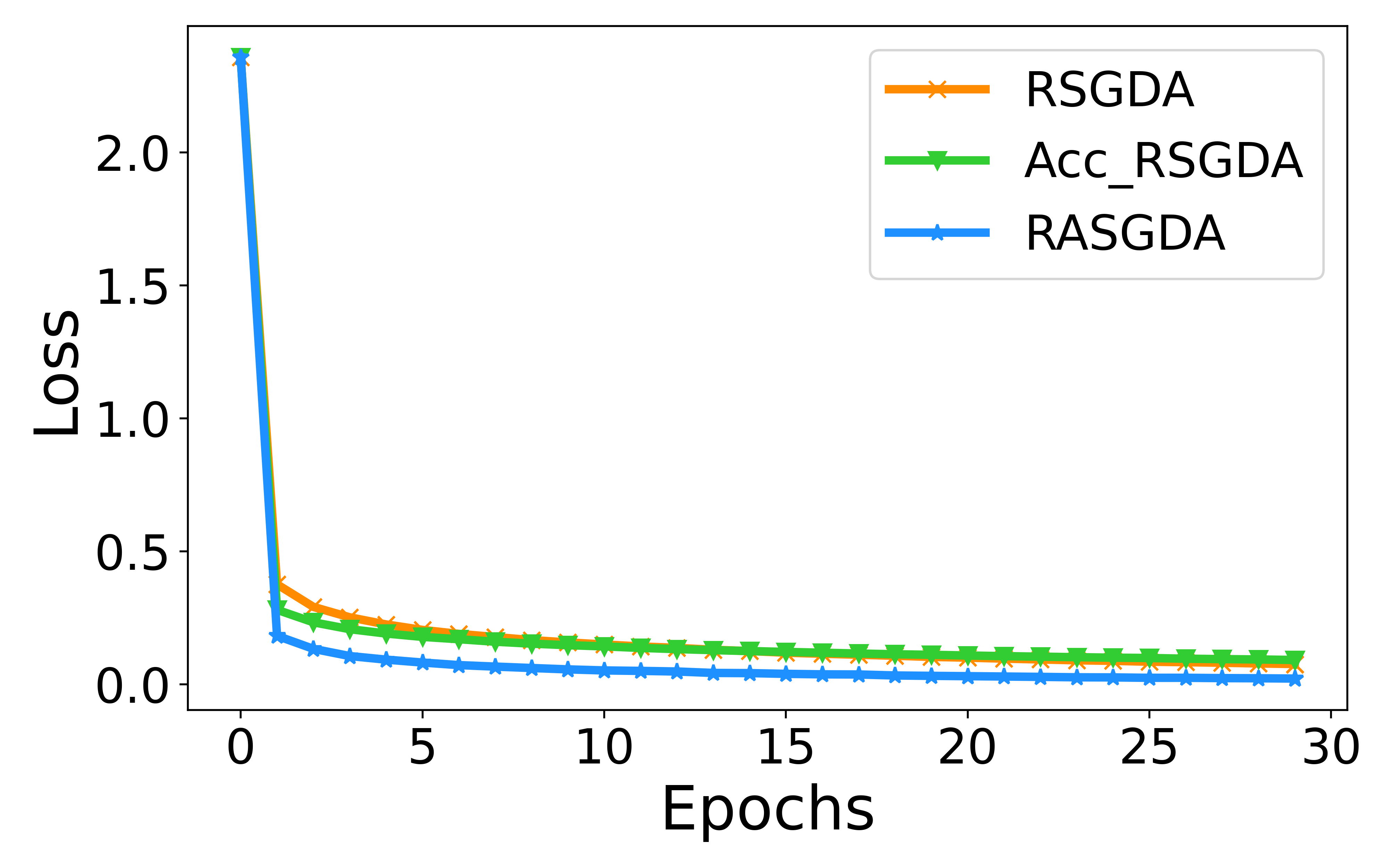}
        \caption{Loss over epochs}
        \label{subfig.mnist loss}
    \end{subfigure}
    \caption{Results of robust training of neural networks with orthonormal weights on the MNIST dataset.}
    \label{fig.mnist}
\end{figure}

\begin{figure}[htbp]
    \centering
    \begin{subfigure}[b]{0.32\textwidth}
        \includegraphics[width=\textwidth]{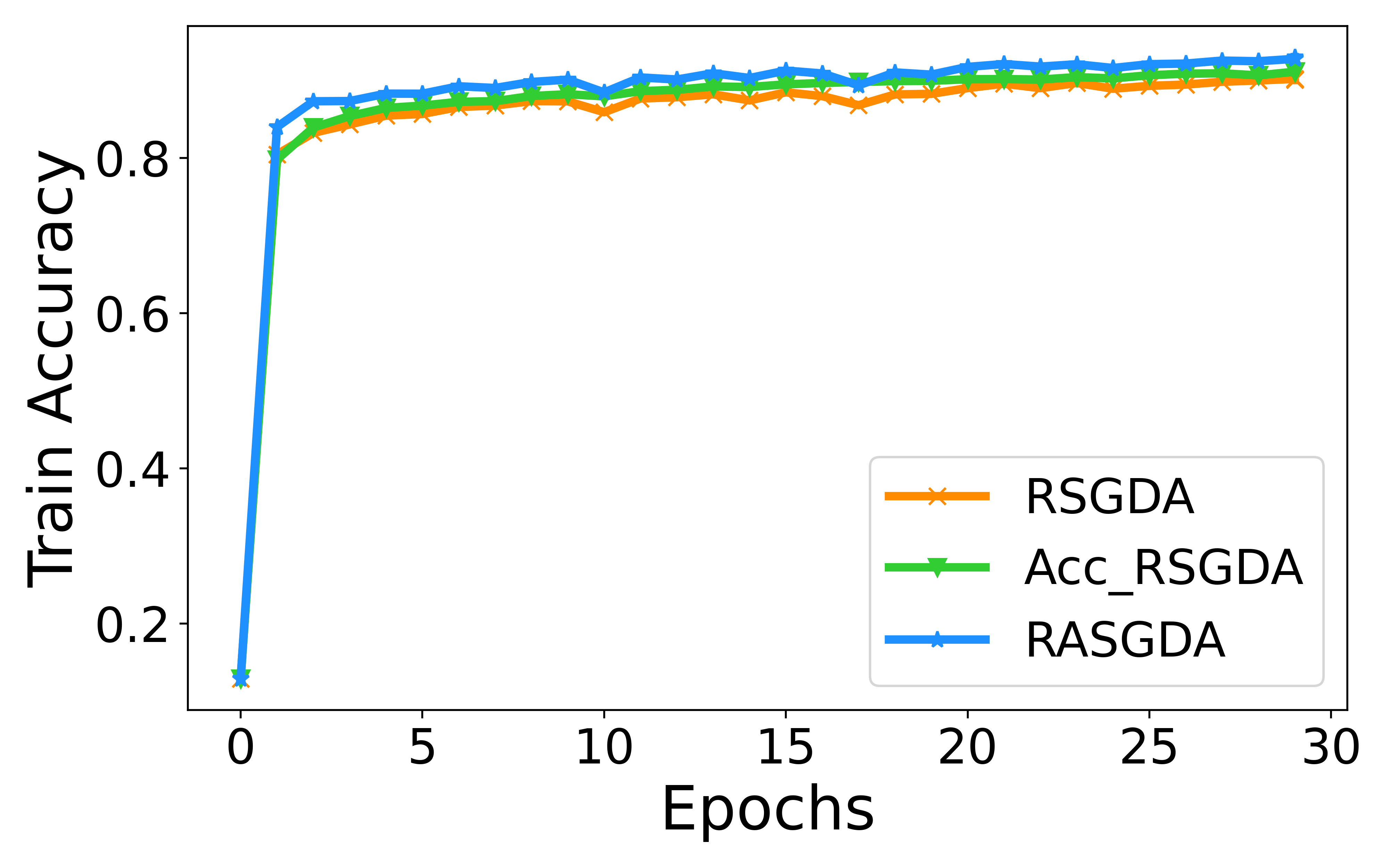}
        \caption{Train accuracy over epochs}
        \label{subfig.fashionmnist train}
    \end{subfigure}
    \hfill
    \begin{subfigure}[b]{0.32\textwidth}
        \includegraphics[width=\textwidth]{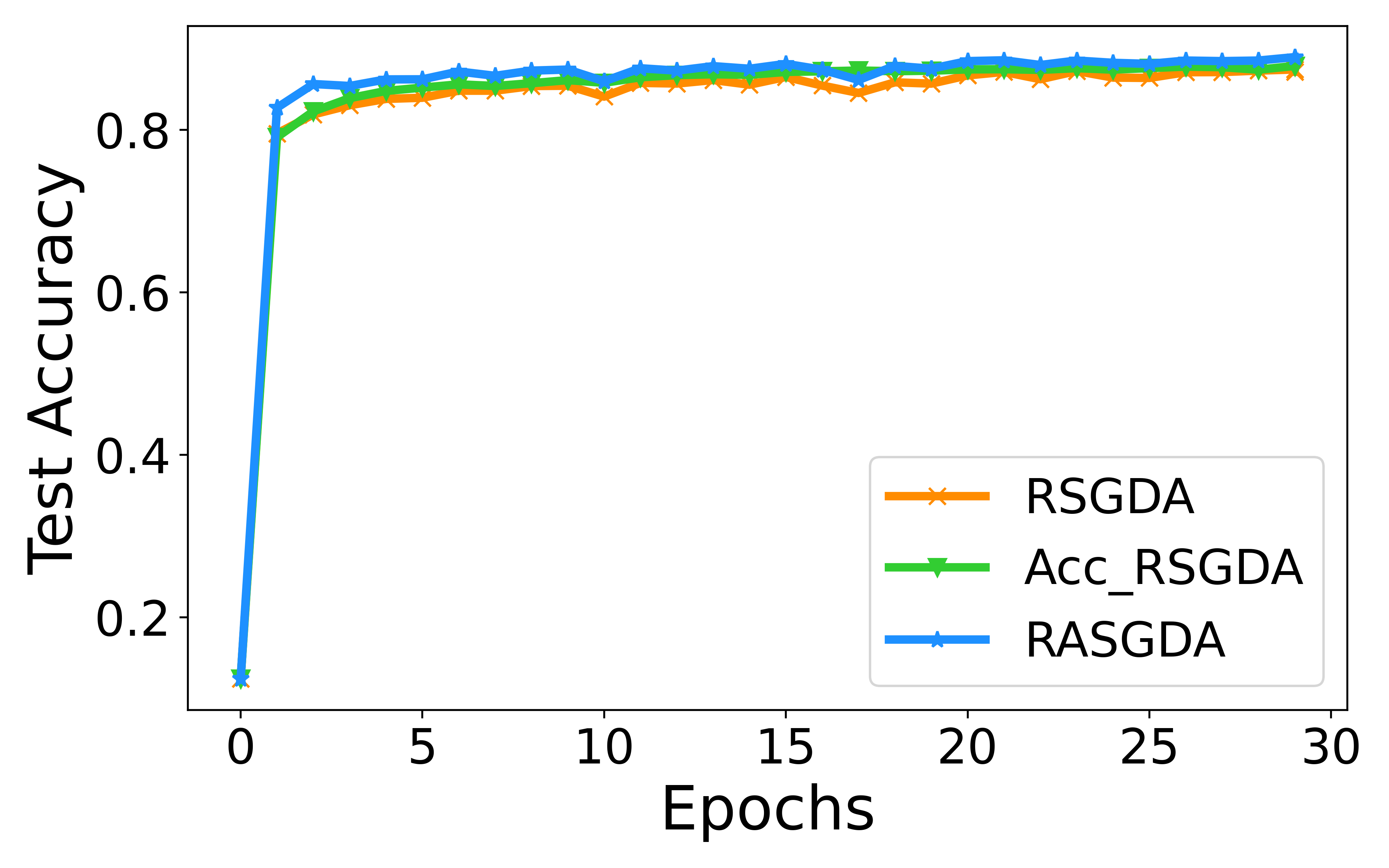}
        \caption{Test accuracy over epochs}
        \label{subfig.test}
    \end{subfigure}
    \hfill
    \begin{subfigure}[b]{0.32\textwidth}
        \includegraphics[width=\textwidth]{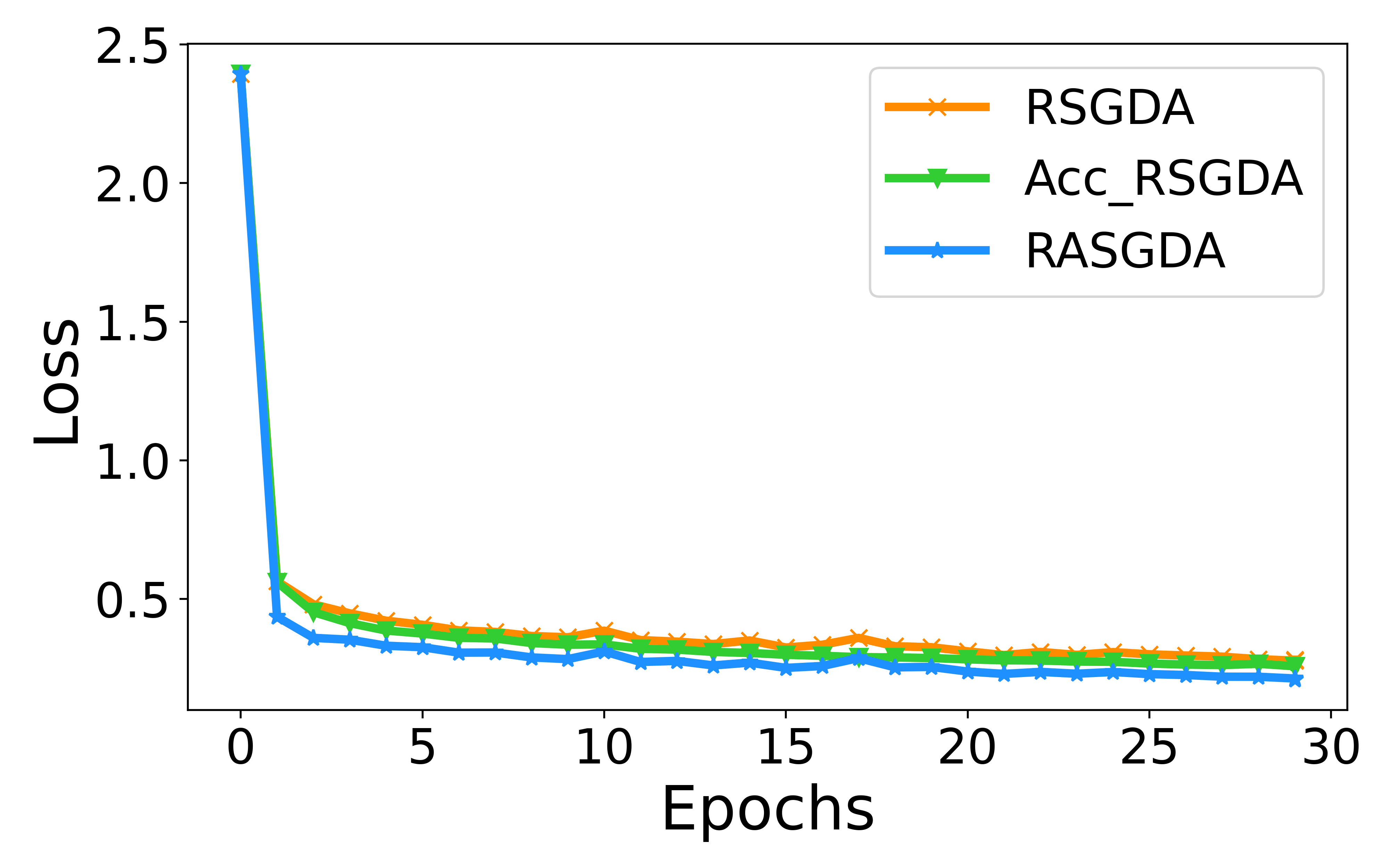}
        \caption{Loss against over epochs}
        \label{subfig.fashionmnist loss}
    \end{subfigure}
    \caption{Results of robust training of neural networks with orthonormal weights on the FashionMNIST dataset.}
    \label{fig.fashionmnist}
\end{figure}

\begin{figure}[htbp]
    \centering
    \begin{subfigure}[b]{0.32\textwidth}
        \includegraphics[width=\textwidth]{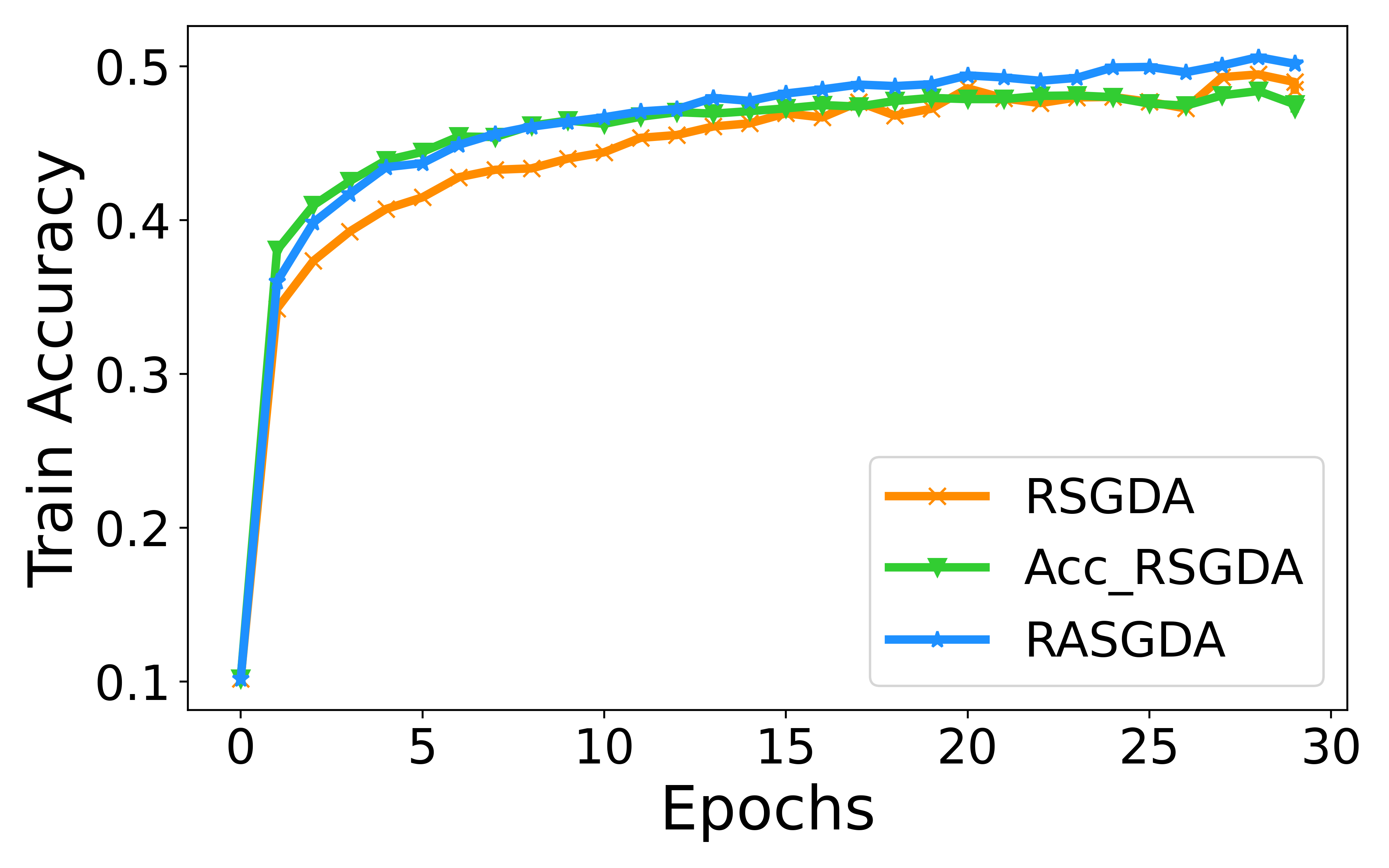}
        \caption{Train accuracy over epochs}
        \label{subfig.cifar10 train}
    \end{subfigure}
    \hfill
    \begin{subfigure}[b]{0.32\textwidth}
        \includegraphics[width=\textwidth]{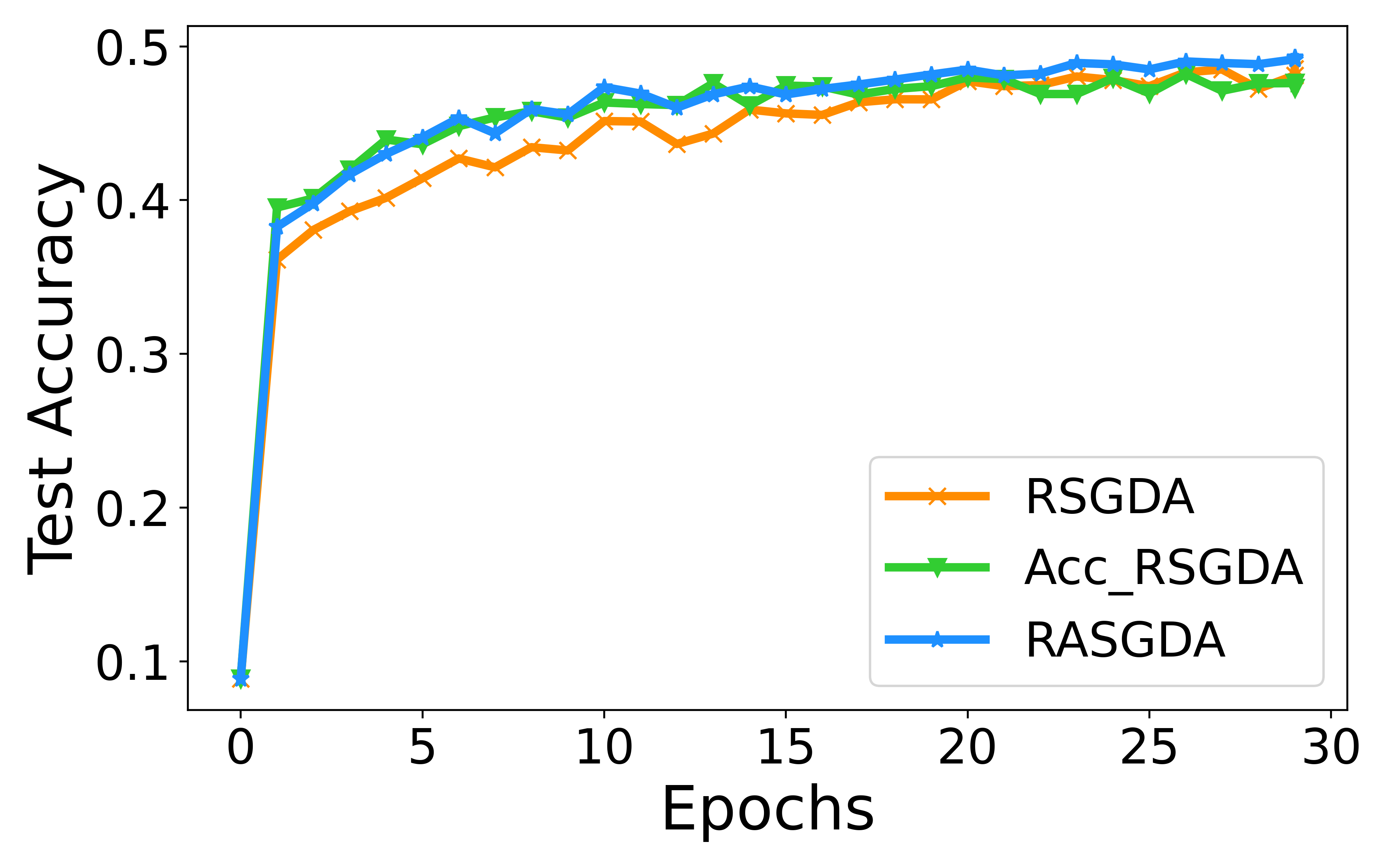}
        \caption{Test accuracy over epochs}
        \label{subfig.cifar10 test}
    \end{subfigure}
    \hfill
    \begin{subfigure}[b]{0.32\textwidth}
        \includegraphics[width=\textwidth]{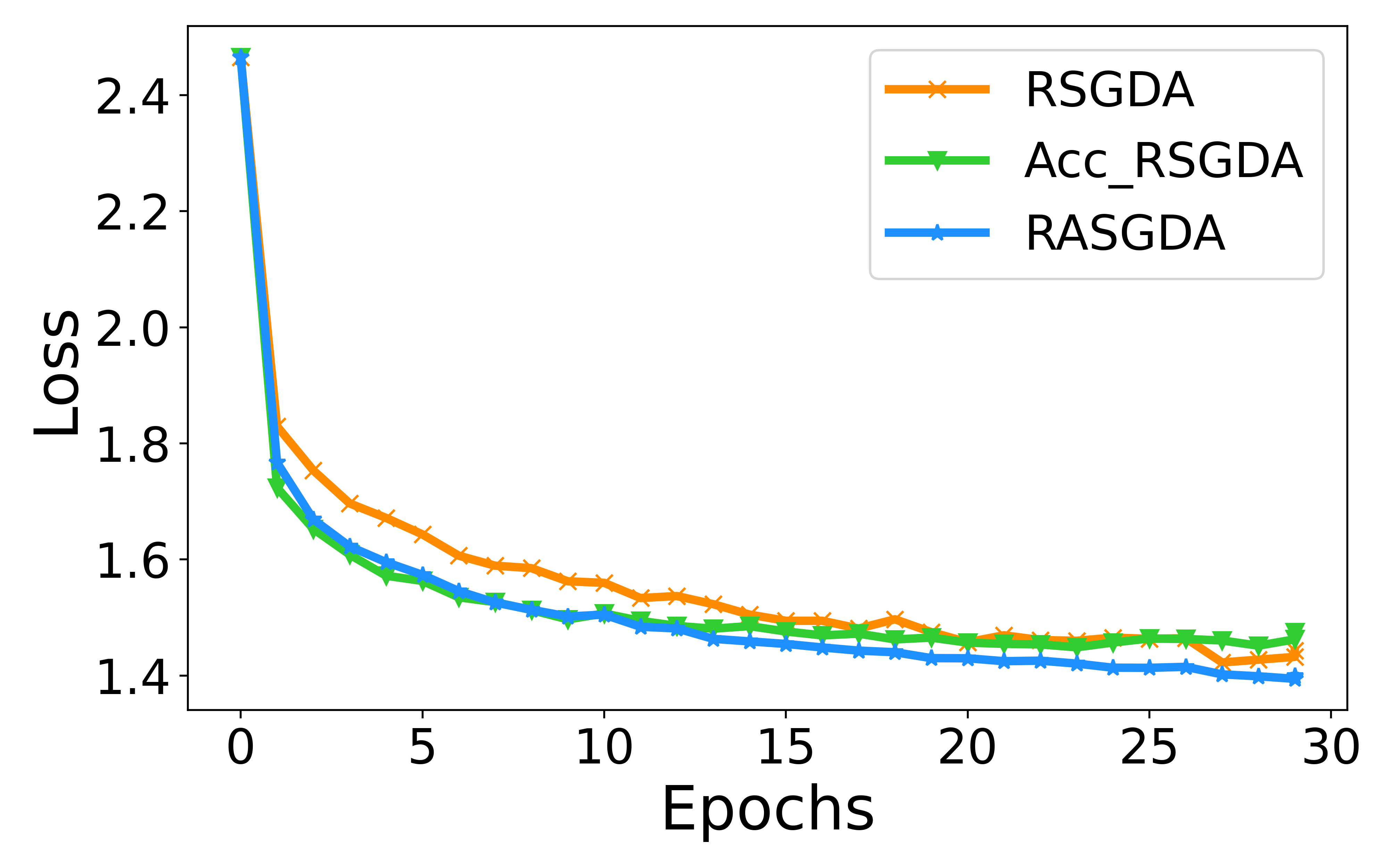}
        \caption{Loss against over epochs}
        \label{subfig.cifar10 loss}
    \end{subfigure}
    \caption{Results of robust training of neural networks with orthonormal weights on the CIFAR10 dataset.}
    \label{fig.cifar10}
\end{figure}

From the results, our method converges faster than the other two baseline methods in terms of the number of gradient calls. Moreover, it consistently achieves higher training and test accuracy throughout the optimization process, indicating better generalization performance. In addition, our method yields a noticeably lower training loss, suggesting a more effective minimization of the objective function. These results highlight the overall superiority of our approach in both convergence speed and final performance.

To evaluate the robustness of the models, we perform adversarial attacks on the test data using Projected Gradient Descent (PGD, 40 steps) attack \citep{kurakin2016adversarial} and Fast Gradient Sign Method (FGSM) attack \citep{goodfellow2014explaining} within norm balls of varying radii. It is worth noting that, although universal perturbations are applied during training, pointwise perturbations are used during testing. The classification accuracies of the models under different attack radii are reported in Tables 1, 2, and 3. From the experimental results, the proposed RSAGDA achieves higher accuracy than both RSGDA and Acc-RSGDA under all attack radii, highlighting the effectiveness of our approach in improving robustness against adversarial perturbations.

\begin{table}[h]
    \centering
    \caption{Classification accuracies ($\%$) on the MNIST dataset under PGD$^{40}$ and FGSM attacks with varying perturbation radii $r$, along with the natural test accuracy.}
    \label{table.minist}
    \resizebox{\textwidth}{!}{
    \begin{tabular}{c|c|ccccc|ccccc}
        \hline
        \multirow{2}{*}{Methods} & \multirow{2}{*}{Natural} 
        & \multicolumn{5}{c|}{PGD$^{40}$} 
        & \multicolumn{5}{c}{FGSM} \\
        \cline{3-12}
        & & $r=0.2$ & $r=0.4$ & $r=0.6$ & $r=0.8$ & $r=1.0$  
        & $r=0.2$ & $r=0.4$ & $r=0.6$ & $r=0.8$ & $r=1.0$ \\
        \hline
        RSGDA     & 97.13\% & 96.18\% & 95.08\% & 93.66\% & 91.60\% & 89.49\% & 96.19\% & 95.08\% & 93.80\% & 91.87\% & 89.97\%\\
        Acc-RSGDA & 96.77\% & 95.82\% & 94.47\% & 93.13\% & 91.45\% & 89.25\% & 95.82\% & 94.49\% & 93.21\% & 91.67\% & 89.66\%\\
        RSAGDA   & 97.66\% & \textbf{97.05\%} & \textbf{96.16\%} & \textbf{94.70\%} & \textbf{93.34\%} & \textbf{90.96\%} & \textbf{97.06\%} & \textbf{96.21\%} & \textbf{94.82\%} & \textbf{93.59\%} & \textbf{91.37\%}\\
        \hline
    \end{tabular}
    }
\end{table}

\begin{table}[h]
    \centering
    \caption{Classification accuracies ($\%$) on the FashionMNIST dataset under PGD$^{40}$ and FGSM attacks with varying perturbation radii $r$, along with the natural test accuracy.}
    \label{table.fashionminist}
    \resizebox{\textwidth}{!}{
    \begin{tabular}{c|c|ccccc|ccccc}
        \hline
        \multirow{2}{*}{Methods} & \multirow{2}{*}{Nat. Img.} 
        & \multicolumn{5}{c|}{PGD$^{40}$} 
        & \multicolumn{5}{c}{FGSM} \\
        \cline{3-12}
        & & $r=0.1$ & $r=0.2$ & $r=0.3$ & $r=0.4$ 
        & $r=0.5$ & $r=0.1$ & $r=0.2$ & $r=0.3$ & $r=0.4$ 
        & $r=0.5$ \\
        \hline
        RSGDA     & 87.16\% & 85.55\% & 83.72\% & 81.55\% & 79.72\% & 77.81\% & 85.56\% & 83.74\% & 81.58\% & 79.85\% & 77.93\%\\
        Acc-RSGDA & 87.71\% & 86.05\% & 84.22\% & 81.97\% & 79.99\% & 77.90\% & 86.06\% & 84.27\% & 82.02\% & 80.21\% & 78.10\%\\
        RSAGDA   & 88.78\% & \textbf{86.97\%} & \textbf{84.97\%} & \textbf{83.04\%} & \textbf{80.68\%} & \textbf{78.29\%} & \textbf{86.97\%} & \textbf{85.01\%} & \textbf{83.11\%} & \textbf{80.79\%} & \textbf{78.58\%}\\
        \hline
    \end{tabular}
    }
\end{table}

\begin{table}[h]
    \centering
    \caption{Classification accuracies ($\%$) on the CIFAR10 dataset under PGD$^{40}$ and FGSM attacks with varying perturbation radii $r$, along with the natural test accuracy.}
    \label{table.cifar10}
    \resizebox{\textwidth}{!}{
    \begin{tabular}{c|c|ccccc|ccccc}
        \hline
        \multirow{2}{*}{Methods} & \multirow{2}{*}{Nat. Img.} 
        & \multicolumn{5}{c|}{PGD$^{40}$} 
        & \multicolumn{5}{c}{FGSM} \\
        \cline{3-12}
        & & $r=0.05$ & $r=0.10$ & $r=0.15$ & $r=0.20$ & $r=0.25$
        & $r=0.05$ & $r=0.10$ & $r=0.15$ & $r=0.20$ & $r=0.25$\\
        \hline
        RSGDA     & 48.50\% & 47.00\% & 45.70\% & 44.38\% & 43.14\% & 41.75\% & 46.97\% & 45.68\% & 44.42\% & 43.19\% & 41.81\%\\
        Acc-RSGDA & 47.31\% & 46.10\% & 44.99\% & 43.89\% & 42.80\% & 41.64\% & 46.08\% & 44.98\% & 43.91\% & 42.86\% & 41.67\%\\
        RSAGDA   & 49.30\% & \textbf{47.95\%} & \textbf{46.64\%} & \textbf{45.31\%} & \textbf{44.17\%} & \textbf{42.86\%} & \textbf{47.94\%} & \textbf{46.63\%} & \textbf{45.34\%} & \textbf{44.21\%} & \textbf{42.91\%}\\
        \hline
    \end{tabular}
    }
\end{table}

\section{Conclusion}
In this work, we proposed two adaptive methods for Riemannian minimax optimization that eliminate the need for problem-specific parameter tuning. For the deterministic setting, we developed Riemannian adaptive gradient descent ascent with an iteration complexity of $\mathcal{O}(\epsilon^{-2})$. For the stochastic setting, we proposed Riemannian stochastic adaptive gradient descent ascent , achieving $\mathcal{O}(\epsilon^{-6})$ complexity, further improved to $\mathcal{O}(\epsilon^{-4})$ under second-order smoothness. Numerical experiments validate the theory and demonstrate the practical effectiveness of our methods. 

Looking ahead, several promising research directions remain. First, it would be valuable to explore variance-reduced \citep{sato2019riemannian, zhang2016riemannian, han2021improved} adaptive methods to further accelerate convergence in stochastic settings. Second, extending our methods to handle non-smooth objectives \citep{hosseini2013nonsmooth, wang2022riemannian} or constraints over product manifolds \citep{becigneul2018riemannian} could broaden their applicability. Finally, integrating these adaptive approaches into large-scale machine learning systems, such as federated learning \citep{wang2025federated, li2022federated} also presents an interesting and promising direction for future research.

\clearpage

\clearpage
\bibliographystyle{abbrvnat}
\bibliography{main}

\newpage
\appendix
\section{Useful lemmas}
\begin{lemma}[\rm Lemma 2 in \cite{yang2022nest}]
\label{le.adaptive bound}
    Let $a_1, \ldots, a_T$ be a sequence of non-negative real numbers, $a_1 > 0$ and $0 < \alpha < 1$. Then we have
\[
\left( \sum_{t=1}^{T} a_t \right)^{1-\alpha}
\le \sum_{t=1}^{T} \frac{a_t}{\left( \sum_{k=1}^{t} a_k \right)^{\alpha}}
\le \frac{1}{1-\alpha} \left( \sum_{t=1}^{T} a_t \right)^{1-\alpha}.
\]

When $\alpha = 1$, we have
\[
\sum_{t=1}^{T} \frac{a_t}{\left( \sum_{k=1}^{t} a_k \right)^{\alpha}} \le 1 + \log\left( \frac{\sum_{t=1}^{T} a_t}{a_1} \right).
\]
\end{lemma}
\begin{lemma}\label{le.lips smooth}
    Under Assumption \ref{ass.l-smooth}, \ref{ass.g-convex} and \ref{ass.optimal exist and stationary condition hold}, there exist constants ${\kappa}$, $L_\phi$ and $L_y$ such that
    \begin{itemize}
        \item [{\rm (1)}] $\dy(y^*(x_1),y^*(x_2))\le {\kappa}\dx(x_1,x_2)$.
        \item [{\rm (2)}]  $\|\operatorname{D}y^*(x)\|\le\kappa$.
        \item [{\rm (3)}] $\|\Gamma_{x_1}^{x_2}\grad \Phi(x_1)-\grad\Phi(x_2)\|\le L_{\phi}d(x_1,x_2)$.
    \end{itemize}
    where ${\kappa}=L_1/\mu$ and $L_{\Phi}=L_1+\kappa L_1$.
\end{lemma}
\begin{proof}
    For any $x_1,x_2\in \cM_x$, the optimality of $y^*(x_1)$ and $y^*(x_2)$ yields
    \[
    \inner{\Exp_{y^*(x_1)}^{-1}(y)}{\grad_y f(x_1, y^*(x_1))}\leq 0, \quad \forall y \in \cM_y,
    \]
    \[
    \inner{\Exp_{y^*(x_2)}^{-1}(y)}{\grad_y f(x_2, y^*(x_2))}\leq 0, \quad \forall y \in \cM_y.
    \]
    By letting \( y = y^*(x_2) \) in the first inequality and \( y = y^*(x_1) \) in second inequality and summing the resulting two inequalities, it follows 
    \begin{align}\label{eq.optimality corollary}
        \inner{\Exp_{y^*(x_1)}^{-1}(y^*(x_2))}{\grad_y f(x_1, y^*(x_1))-\Gamma_{y^*(x_2)}^{y^*(x_1)}\grad_y f(x_2, y^*(x_2))}  \leq 0. 
    \end{align}
Recall that \( f(x_1, \cdot) \) is $\mu$-strongly concave, we have  
\[
    f\left(x_1,y^*(x_1)\right) \leq f(x_1,y^*(x_2)) + \langle \grad_y f(x_1,y^*(x_2)), \Exp_{y^*(x_2)}^{-1}(y^*(x_1)) \rangle - \frac{\mu}{2} d^2_{\cM_y}(y^*(x_1),y^*(x_2))
\]
\[
    f\left(x_1,y^*(x_2)\right) \leq f(x_1,y^*(x_1)) + \langle \grad_y f(x_1,y^*(x_1)), \Exp_{y^*(x_1)}^{-1}(y^*(x_2)) \rangle - \frac{\mu}{2} d^2_{\cM_y}(y^*(x_1),y^*(x_2))
\]
Combining the above two inequalities yeilds
\begin{align}\label{eq.strongly concave corollary}
    \inner{\Exp_{y^*(x_1)}^{-1}(y^*(x_2))}{\Gamma_{y^*(x_2)}^{y^*(x_1)}\grad_y f(x_1, y^*(x_2))-\grad_y f(x_1, y^*(x_1))}+\mu d^2_{\cM_y}(y^*(x_1),y^*(x_2))\le 0.
\end{align}
Then we conclude the desired result by combining \eqref{eq.optimality corollary} and \eqref{eq.strongly concave corollary} with $L_1$-smoothness of $f$, i.e.,  
\begin{align*}
    \mu d^2(y^*(x_2), y^*(x_1)) \leq &\inner{\Exp_{y^*(x_1)}^{-1}(y^*(x_2))}{\Gamma_{y^*(x_2)}^{y^*(x_1)}\grad_y f(x_2, y^*(x_2)) - \Gamma_{y^*(x_2)}^{y^*(x_1)}\grad_y f(x_1, y^*(x_2))}\\
    \le &\|\Exp_{y^*(x_1)}^{-1}(y^*(x_2))\|\|\grad_y f(x_2, y^*(x_2)) - \grad_y f(x_1, y^*(x_2))\|\\
    \leq &L_1 d(y^*(x_2), y^*(x_1))d(x_1, x_2).
\end{align*}
Therefore, we have 
\begin{align}\label{eq.y* Lips}
    d(y^*(x_1), y^*(x_2))\le \kappa d(x_1, x_2).
\end{align}
To prove $\rm (2)$, we have for any $u\in \operatorname{T}_x\cM_x$ 
\begin{align*}
    \|\operatorname{D}y^*(x)(u)\|=\lim_{t\rightarrow 0}\frac{d(y^*(x),y^*(\Exp_x(tu)))}{|t|}\le \lim_{t\rightarrow 0}\frac{\|tu\|}{|t|}=\kappa\|u\|.
\end{align*}
Moreover, from $\grad \Phi(x) = \grad_x f(x, y^*(x))$, it holds
\begin{align*}
    \|\Gamma_{x_1}^{x_2}\grad \Phi(x_1) - \grad \Phi(x_2)\| = &\|\Gamma_{x_1}^{x_2}\grad_x f(x_1, y^*(x_1)) - \grad_x f(x_2, y^*(x_2))\|\\
    \leq &L_1 (d(x_1, x_2) + d(y^*(x_1), y^*(x_2)))\\
    \le &(L_1+L_1 \kappa)d(x_1,x_2).
\end{align*}
where the second inequality is from Assumption \ref{ass.l-smooth} and the last inequality is from \eqref{eq.y* Lips}.
\end{proof}
\begin{lemma}\label{le.l smooth + mu convex}
    If function $g(x):\cM\rightarrow\R$ is $L_1$-smooth and $\mu$ $\mu$-geodesically convex, then we have, for any $x_1,x_2\in\cM$,
    \[
        \inner{\Gamma_{x_2}^{x_1}\grad g(x_2)-\grad g(x_1)}{\Exp^{-1}_{x_1}(x_2)}\ge \frac{\mu L_1}{L_1 + \mu}d^2(x_1,x_2)+\frac{1}{L_1+\mu}\|\Gamma_{x_2}^{x_1}\grad g(x_2)-\grad g(x_1)\|^2.
    \]
\end{lemma}
\begin{proof}
    For fixed $x_2\in\cM$, let $\Phi(x)=g(x)-\frac{\mu}{2}d^2(x,x_2)$. From $\grad \Psi(x)=\grad g(x)+\mu \Exp^{-1}_{x}(x_2)$, it follows 
    \begin{align*}
        &\inner{\grad \Psi(x)-\Gamma_{x_2}^{x}\grad \Psi(x_2)}{-\Exp^{-1}_{x}(x_2)}\\
        =&\inner{\grad g(x)-\Gamma_{x_2}^{x}\grad g(x_2)}{-\Exp^{-1}_{x}(x_2)}+\inner{\mu \Exp^{-1}_{x}(x_2)}{-\Exp^{-1}_{x}(x_2)}\\
        \le & (L_2-\mu)d^2(x,x_2),
    \end{align*}
    where the inequality is from $L_1$-smooth of $g(\cdot)$.
    Moreover, we have 
    $$
        \|\grad \Psi(x)-\Gamma_{x_2}^{x}\grad \Psi(x_2)\|\le (L_1-\mu)d(x,x_2).
    $$
    Let $\psi(x)=\Psi(x)-\inner{\grad \Psi(x_1)}{\Exp^{-1}_{x_1}(x)}$. From $\grad \psi(x)=\grad \Psi(x)-\Gamma_{x_1}^{x} \grad \Psi(x_1)$, it holds
    \begin{align*}
        \|\grad \psi(x)-\Gamma_{x_2}^{x}\grad \psi(x_2)\|=&\|\grad \Psi(x)-\Gamma_{x_1}^{x} \grad \Psi(x_1)-\Gamma_{x_2}^{x}\grad \Psi(x_2)+\Gamma_{x_2}^{x}\Gamma_{x_1}^{x_2}\grad \Psi(x_1)\|\\
        \le &\|\grad \Psi(x)-\Gamma_{x_2}^{x} \grad \Psi(x_2)\| \le (L_1-\mu)d(x,x_2).
    \end{align*}
    Then it holds
    \begin{align*}
        \psi(x)-\psi(x_2)-\inner{\grad \psi(x_2)}{\Exp^{-1}_{x_2}(x)}\le \frac{L_1-\mu}{2}d^2(x,x_2).
    \end{align*}
    From the above inequality, we have 
    \begin{align*}
        \psi(x^*)=&\min_{x\in\cM}\psi(x)\le \min_{x\in\cM}\left\{\psi(x_2)+\inner{\grad \psi(x_2)}{\Exp^{-1}_{x_2}(x)}+\frac{L_1-\mu}{2}d^2(x,x_2)\right\}\\
        &\le \psi(x_2)-\frac{1}{2(L_1-\mu)}\|\grad \psi(x_2)\|^2.
    \end{align*}
    From $x^*=\arg\min_{x\in\cM}\psi(x)=x_1$ and $\psi(x)=\Psi(x)-\inner{\grad \Psi(x_1)}{\Exp^{-1}_{x_1}(x)}$, we can get
    \begin{align*}
        \Psi(x_2)-\Psi(x_1)-\inner{\grad \Psi(x_1)}{\Exp^{-1}_{x_1}(x_2)}\ge \frac{1}{2(L_1-\mu)}\|\grad \psi(x_2)\|^2.
    \end{align*}
    Bringing $\Phi(x)=g(x)-\frac{\mu}{2}d^2(x,x_2)$ to the above inequality yields
    \begin{align*}
        &g(x_2)-g(x_1)+\frac{\mu}{2}d^2(x_1,x_2)-\inner{\grad g(x_1)+\mu \Exp^{-1}_{x_1}(x_2)}{\Exp^{-1}_{x_1}(x_2)}\\
        \ge &\frac{1}{2(L_1 - \mu)}\|\grad g(x_2)-\Gamma_{x_1}^{x_2}(\grad g(x_1)+\mu \Exp^{-1}_{x_1}(x_2))\|^2.
    \end{align*}
    Because we can choose any $x_1,x_2\in\cM$, the above inequality also holds if we exchange the position of $x_1$ and $x_2$. Therefore, we have 
    \begin{align*}
        &g(x_1)-g(x_2)+\frac{\mu}{2}d^2(x_1,x_2)-\inner{\grad g(x_2)+\mu \Exp^{-1}_{x_2}(x_1)}{\Exp^{-1}_{x_2}(x_1)}\\
        \ge &\frac{1}{2(L_1 - \mu)}\|\grad g(x_1)-\Gamma_{x_2}^{x_1}(\grad g(x_2)+\mu \Exp^{-1}_{x_2}(x_1))\|^2.
    \end{align*}
    Summing the above two inequalities and rearranging the new inequalities derives the desired result.
\end{proof}
\begin{lemma}[Trigonometric distance bound 
\citep{petersen2006riemannian, zhang2016riemannian, zhang2016first}
]. \label{le.trigonometric distance bound}
Let $x_a, x_b, x_c \in \mathcal{U} \subseteq \mathcal{M}$ and denote 
$a = d(x_b, x_c)$, $b = d(x_a, x_c)$ and $c = d(x_a, x_b)$ as the geodesic side lengths. Then,
\[
a^2 \leq \zeta b^2 + c^2 - 2 \langle \Exp^{-1}_{x_a}(x_b), \Exp^{-1}_{x_a}(x_c) \rangle_{x_a}
\]
where 
\[
\zeta = 
\begin{cases}
\dfrac{\sqrt{|\tau|} \bar{G}}{\tanh(\sqrt{|\tau|} \bar{G})} & \text{if } \tau < 0 \\
1 & \text{if } \tau \geq 0
\end{cases}
\]
Here, $\bar{G}$ denotes the diameter of $\mathcal{U}$ and $\tau$ denotes the lower bound of the sectional curvature of $\mathcal{U}$.
\end{lemma}
Before proceeding to the proofs of the main theorems, we summarize the notation used the subsequent analysis in Table~\ref{table.notation summary} for ease of reference.
\begin{table}[ht]
\centering
\begin{tabular}{cc}
\hline
\textbf{Notation} & \textbf{Meaning} \\
\hline
\hline
$l$ & Lipschitz constant of $\grad_xf(x,y)$ and $\grad_yf(x,y)$\\
$\mu$ &  geodesic strong convexity constant of $f(x, y)$ in y \\
$\kappa=l/\mu$ & condition number \\
$g^x_t$ & Riemannian gradient $\grad_xf(x_{t},y_{t})$\\
$g^y_t$ & Riemannian gradient $\grad_yf(x_{t},y_{t})$\\
$v_0^x$ &  initial accumulator for accumulated the squared norm of $g^x_t$.\\
$v_t^x$, $t\ge 1$ & $=v_0^x+\sum_{i=0}^{t-1}\|g^x_{i}\|^2$\\
$v_0^y$ &  initial accumulator for accumulated the squared norm of $g^y_t$.\\
$v_t^y$, $t\ge 1$ & $=v_0^y+\sum_{i=0}^{t-1}\|g^y_{i}\|^2$\\
$\eta^x$ & stepsize parameter for $x$-update \\
$\eta^y$ & stepsize parameter for $y$-update\\
$\alpha$ & decay exponent for $x$-update\\
$\beta$ & decay exponent for $y$-update\\
$\eta_t$ & $= \frac{\eta^x}{\max\{v^x_{t+1}, v^y_{t+1}\}^\alpha}$\\
$\gamma_t$ & $= \frac{\eta^y}{(v^y_{t+1})^\beta}$\\
$\zeta$ & curvature constant upper bound\\
$\Phi(x)$ & $= \max_{y\in \cM_y}f(x,y)$\\
$G$ & upper bound of the $\|\grad_x\Phi(x)\|$, $\|\grad_xf(x,y;\xi)\|$ and $\|\grad_yf(x,y;\xi)\|$\\
$\bar{c}$ & distance distortion constant for the retraction
\\
$c_R$ & second-order retraction accuracy constant
\\
$\mathbf{1}_{A}$ & indicator function of event $A$\\
$\Phi^*$ & $=\min_{x \in \cM_x} \Phi(x) > -\infty$\\
$\Delta \Phi$ & $= \Phi(x_0) - \Phi^*$\\
$L_{\Phi}$ & Lipschitz constant of $\grad\Phi(x)$\\
$L_2$ & Lipschitz constant of $\grad_{yx}f(x,y)$ and $\hess_y f(x,y)$\\
\hline
\end{tabular}
\caption{Summary of Notation}
\label{table.notation summary}
\end{table}
\section{Proof of Theorem~\ref{the.deterministic convergence}}\label{app.deterministic convergence}
\begin{lemma}\label{le.bound dist of yt0 and yt0star}
    Suppose Assumption \ref{ass.l-smooth}, \ref{ass.g-convex} and \ref{ass.retraction} hold. Let $t_0$ denote the first iteration such that $(v_{t_0+1}^y)^\beta > c_1$, where
    \[
    c_1 := \max \left\{ \frac{4 \eta^y \mu L_1}{\mu + L_1},\; \eta^y (\mu + L_1)\left(\zeta \bar{c} + \frac{2 G c_R}{\mu} \right) \right\}.
    \]
    Then, the squared geodesic distance between $y_{t_0}$ and its corresponding maximizer $y^*_{t_0}$ satisfies
    \[
        d^2( y_{t_0}, y^*_{t_0})\leq 4 \left( \frac{1}{\mu^2} + \frac{\bar{c}\eta^y}{(v^y_{t_0})^\beta} \right) c_2^{1/\beta} + \frac{2 \kappa^2\bar{c} (v^y_{t_0})^\beta }{(v^y_0)^\beta} \eta_{t_0 - 1}^2\| \grad_x f(x_{t_0 - 1}, y_{t_0 - 1}) \|^2.
    \]
\end{lemma}
\begin{proof}
    \begin{align*}
d^2( y_{t_0}, y^*_{t_0})
&\leq 2 d^2( y_{t_0}, y^*_{t_0-1}) + 2 d^2( y^*_{t_0}, y^*_{t_0-1}) \\
&\le 4d^2(y_{t_0-1},y^*_{t_0-1})+4d^2(y_{t_0-1},y_{t_0}) + 2 d^2( y^*_{t_0}, y^*_{t_0 - 1}) \\
&\leq 4 \left( d^2( y_{t_0 - 1}, y^*_{t_0 - 1}) + \bar{c}\gamma_{t_0 - 1}^2\| \grad_y f(x_{t_0 - 1}, y_{t_0 - 1}) \|^2 \right) + 2 d^2( y^*_{t_0}, y^*_{t_0 - 1}) \\
&\leq 4 \left( \frac{1}{\mu^2} \| \grad_y f(x_{t_0 - 1}, y_{t_0 - 1}) \|^2 + \bar{c}\gamma_{t_0 - 1}^2 \| \grad_y f(x_{t_0 - 1}, y_{t_0 - 1}) \|^2 \right) + 2 d^2( y^*_{t_0}, y^*_{t_0 - 1}) \\
&= 4 \left( \frac{1}{\mu^2} + \bar{c}\gamma_{t_0 - 1}^2 \right) \| \grad_y f(x_{t_0 - 1}, y_{t_0 - 1}) \|^2 + 2d^2( y^*_{t_0}, y^*_{t_0 - 1}) \\
&\leq 4 \left( \frac{1}{\mu^2} + \bar{c}\gamma_{t_0-1}^2 \right) v^y_{t_0} + 2d^2( y^*_{t_0}, y^*_{t_0 - 1})\\
&\leq 4 \left( \frac{1}{\mu^2} + \frac{\bar{c}\eta^y}{(v^y_{t_0})^\beta} \right) c_2^{1/\beta} + 2d^2( y^*_{t_0}, y^*_{t_0 - 1}) \\
&\leq 4 \left( \frac{1}{\mu^2} + \frac{\bar{c}\eta^y}{(v^y_{t_0})^\beta} \right) c_2^{1/\beta} + 2 \kappa^2 d^2( x_{t_0}, x_{t_0 - 1}) \\
&\leq 4 \left( \frac{1}{\mu^2} + \frac{\bar{c}\eta^y}{(v^y_{t_0})^\beta} \right) c_2^{1/\beta} + 2 \kappa^2\bar{c} \eta_{t_0 - 1}^2 \| \grad_x f(x_{t_0 - 1}, y_{t_0 - 1}) \|^2 \\
&\leq 4 \left( \frac{1}{\mu^2} + \frac{\bar{c}\eta^y}{(v^y_{t_0})^\beta} \right) c_2^{1/\beta} + \frac{2 \kappa^2\bar{c} (v^y_{t_0})^\beta }{(v^y_0)^\beta} \eta_{t_0 - 1}^2\| \grad_x f(x_{t_0 - 1}, y_{t_0 - 1}) \|^2.
\end{align*}
The first two inequalities follow from the triangle inequality. The third inequality is from Assumption \ref{ass.retraction}. The fourth inequality follows from the $\mu$-strong concavity. We use $\| \grad_y f(x_{t_0 - 1}, y_{t_0 - 1}) \|^2\le v_{t_0}^y$ in the fifth inequality. The sixth inequality is from the definition of $t_0$. We used Lemma \ref{le.lips smooth}, Assumption \ref{ass.retraction}, and the monotonicity of $v_t^y$ in the last three inequalities, respectively.
\end{proof}
\textbf{Proof of Lemma \ref{le.bound sqaure norm y and vTy}}
\begin{proof}
    For any $ t \geq t_0 $ and $\lambda_t>0$, it follows
\begin{align*}
&d^2(y_{t+1}, y_{t+1}^*)\\
\leq &(1 + \lambda_t) d^2(y_{t+1}, y_t^*) + \left( 1 + \frac{1}{\lambda_t} \right)d^2(y_{t+1}^*, y_t^*) \\
\le &(1 + \lambda_t)\left(d^2(y_t, y_t^*) + \zeta d^2(y_{t},y_{t+1}) - 2  \left\langle \Exp_{y_t}^{-1}(y_{t+1}), \Exp_{y_t}^{-1}(y^*_t) \right\rangle\right)+ \left( 1 + \frac{1}{\lambda_t} \right)d^2(y_{t+1}^*, y_t^*) \\
= &(1 + \lambda_t)\bigg(d^2(y_t, y_t^*)+\zeta d^2(y_{t},y_{t+1}) - 2  \left\langle \Exp_{y_t}^{-1}(y_{t+1})-\Retr_{y_t}^{-1}(y_{t+1}), \Exp_{y_t}^{-1}(y^*_t) \right\rangle\\
    &-2\gamma_t\left\langle \grad_y f(x_t, y_t), \Exp_{y_t}^{-1}(y^*_t) \right\rangle\bigg)+\left( 1 + \frac{1}{\lambda_t} \right)d^2(y_{t+1}^*, y_t^*)\\
    \le &(1 + \lambda_t)\bigg(d^2(y_t, y_t^*)+\zeta d^2(y_{t},y_{t+1}) -2\gamma_t\left\langle \grad_y f(x_t, y_t), \Exp_{y_t}^{-1}(y^*_t) \right\rangle\\
    &+ 2 \frac{G}{\mu}\left\|\Exp_{y_t}^{-1}(y_{t+1})-\Retr_{y_t}^{-1}(y_{t+1})\right\| \bigg)+\left( 1 + \frac{1}{\lambda_t} \right)d^2(y_{t+1}^*, y_t^*)\\
\le &(1 + \lambda_t)\bigg(d^2(y_t, y_t^*)+\zeta\bar{c}\gamma_t^2\left\|\grad_y f(x_t, y_t)\right\|^2 -2\gamma_t\left\langle \grad_y f(x_t, y_t), \Exp_{y_t}^{-1}(y^*_t) \right\rangle\\
    &+ 2 \frac{G}{\mu}c_R\gamma_t^2\left\|\grad_y f(x_t, y_t)\right\|^2 \bigg)+\left( 1 + \frac{1}{\lambda_t} \right)d^2(y_{t+1}^*, y_t^*)\\
= &\underbrace{(1 + \lambda_t) \left( d^2(y_t, y_t^*) + \left(\zeta\bar{c}+\frac{2Gc_R}{\mu}\right)\frac{(\eta^y)^2}{(v_{t+1}^y)^{2\beta}} \|\grad_y f(x_t, y_t)\|^2 - \frac{2 \eta^y}{(v_{t+1}^y)^\beta} \langle \Exp^{-1}_{y_t}(y^*_t), \grad_y f(x_t, y_t) \rangle \right)}_{\rm (A)} \\
&+ \left( 1 + \frac{1}{\lambda_t} \right) d^2(y_{t+1}^*, y_t^*),
\end{align*}
Here, the first inequality follows from triangle inequality and Cauchy's inequality. The second inequality uses Lemma \ref{le.trigonometric distance bound}. For the third inequality, we apply the strong convexity condition: $\|\Exp_{y_t}^{-1}(y_t^*)\|\le \frac{1}{\mu}\|\grad_yf(x_t,y_t)\|\le \frac{G}{\mu}$. And we use the Assumption \ref{ass.retraction} in the last inequality.

From Lemma~\ref{le.l smooth + mu convex}, we bound term (A) as
\begin{align*}
\text{Term (A)} &\leq (1 + \lambda_t) \left( \left( 1 - \frac{2 \eta^y \mu L_1}{(\mu + L_1)(v_{t+1}^y)^\beta} \right) d^2(y_t, y_t^*) \right. \\
&\quad \left. + \left( \left(\zeta\bar{c}+\frac{2Gc_R}{\mu}\right)\frac{(\eta^y)^2}{(v_{t+1}^y)^{2\beta}} -\frac{2 \eta^y}{(\mu + L_1)(v_{t+1}^y)^\beta} \right) \|\grad_y f(x_t, y_t)\|^2 \right).
\end{align*}
Let $\lambda_t = \frac{\eta^y \mu L_1}{(\mu + L_1)(v_{t+1}^y)^\beta - 2\eta^y \mu L_1}$, which is positive for $t \ge t_0$. Then,
\begin{align*}
\text{Term (A)} \leq &\left( 1 - \frac{\eta^y \mu L_1}{(\mu + L_1)(v_{t+1}^y)^\beta} \right) d^2(y_t, y_t^*) \\
& + (1 + \lambda_t) \left( \left(\zeta\bar{c}+\frac{2Gc_R}{\mu}\right)\frac{(\eta^y)^2}{(v_{t+1}^y)^{2\beta}} - \frac{2 \eta^y}{(\mu + L_1)(v_{t+1}^y)^\beta} \right) \|\grad_y f(x_t, y_t)\|^2 \\
\leq &d^2(y_t, y_t^*) + \underbrace{(1 + \lambda_t) \left( \left(\zeta\bar{c}+\frac{2Gc_R}{\mu}\right)\frac{(\eta^y)^2}{(v_{t+1}^y)^{2\beta}} - \frac{2 \eta^y}{(\mu + L_1)(v_{t+1}^y)^\beta} \right)}_{\rm (B)} \|\grad_y f(x_t, y_t)\|^2.
\end{align*}
We now use the fact that \(1+\lambda_t \geq 1\) and \(\left(v_{t+1}^y\right)^\beta \geq \eta^y(\mu + L_1)\left(\zeta\bar{c}+\frac{2Gc_R}{\mu}\right)\), which together imply that term (B) is bounded above by \(\leq -\frac{\eta^y}{(\mu + L_1)(v_{t+1}^y)^\beta}\). By combining the above bounds, we obtain
\begin{align*}
&d^2(y_{t+1}, y_{t+1}^*)\\ 
\leq &d^2(y_t, y_t^*) - \frac{\eta^y}{\left(\mu + L_1\right)\left(v_{t+1}^y\right)^\beta}\|\grad_y f(x_t,y_t)\|^2 + \left(1 + \frac{1}{\lambda_t}\right)d^2(y_{t+1}^*, y_t^*) \\
\leq &d^2(y_t, y_t^*) - \frac{\eta^y}{\left(\mu + L_1\right)\left(v_{t+1}^y\right)^\beta}\|\grad_y f(x_t,y_t)\|^2 + \frac{\left(\mu + L_1\right)\left(v_{t+1}^y\right)^\beta}{\eta^y\mu L_1}d^2(y_{t+1}^*, y_t^*) \\
\leq &d^2(y_t, y_t^*) - \frac{\eta^y}{\left(\mu + L_1\right)\left(v_{t+1}^y\right)^\beta}\|\grad_y f(x_t,y_t)\|^2  + \frac{\left(\mu + L_1\right)\kappa^2\left(v_{t+1}^y\right)^\beta}{\eta^y\mu L_1}d^2(x_{t+1}, x_t) \\
\le &d^2(y_t, y_t^*) - \frac{\eta^y}{\left(\mu + L_1\right)\left(v_{t+1}^y\right)^\beta}\|\grad_y f(x_t,y_t)\|^2 + \frac{\left(\mu + L_1\right)\kappa^2\bar{c}\left(v_{t+1}^y\right)^\beta}{\eta^y\mu L_1}\eta_t^2 \|\grad_x f(x_t,y_t)\|^2.
\end{align*}
Here, we use the definition of $\lambda_t$ in the second inequality and the last two inequalities invoke Lemma \ref{le.lips smooth} and Assumption \ref{ass.retraction}.
Summing the above inequality from $t = t_0$ to $T - 1$ yields
\begin{equation}
\begin{split}
    &\sum_{t=t_0}^{T-1} \frac{\eta^y}{(\mu + L_1)(v^y_{t+1})^\beta} \| \grad_y f(x_t, y_t) \|^2\\
\leq &d^2(y_{t_0}, y^*_{t_0})
+ \sum_{t=t_0}^{T-1} \frac{(\mu + L_1) \kappa^2 \bar{c}(v^y_{t+1})^\beta \eta_t^2}{\eta^y \mu L_1} \| \grad_x f(x_t, y_t) \|^2.
\end{split}\nonumber
\end{equation}
Applying Lemma \ref{le.bound dist of yt0 and yt0star} gives
\begin{align*}
\sum_{t=t_0}^{T-1} \frac{\eta^y}{(v_{t+1}^y)^\beta} \| \grad_y f(x_t, y_t) \|^2 
\leq &\underbrace{4(\mu + L_1) \left( \frac{1}{\mu^2} + \frac{\bar{c}\eta^y}{(v_{t_0}^y)^\beta} \right) c_1^{1/\beta}}_{c_2} \\
& + \underbrace{(\mu + L_1) \left( \frac{2\kappa^2\bar{c}}{(v_0^y)^\beta} + \frac{(\mu + L_1)\kappa^2\bar{c}}{\eta^y \mu L_1} \right)}_{c_3} \sum_{t=t_0-1}^{T-1} (v_{t+1}^y)^\beta \eta_t^2 \| \grad_x f(x_t, y_t) \|^2.
\end{align*}
By adding terms from 0 to \( t_0 - 1 \) and \(\frac{\eta^y v_0^y}{(v_0^y)^\beta}\) to both sides, we have
\begin{align*}
&\frac{\eta^y v_0^y}{(v_0^y)^\beta} + \sum_{t=0}^{T-1} \frac{\eta^y}{(v_{t+1}^y)^\beta} \| \grad_y f(x_t, y_t) \|^2 \\
\leq &c_2 + c_3 \sum_{t=0}^{T-1} (v_{t+1}^y)^\beta \eta_t^2 \| \grad_x f(x_t, y_t) \|^2 + \frac{\eta^y v_0^y}{(v_0^y)^\beta} + \sum_{t=0}^{t_0-1} \frac{\eta^y}{(v_{t+1}^y)^\beta} \| \grad_y f(x_t, y_t) \|^2 \\
\leq &c_2 + c_3 \sum_{t=0}^{T-1} (v_{t+1}^y)^\beta \eta_t^2 \| \grad_x f(x_t, y_t) \|^2 + \frac{\eta^y}{1-\beta} v_{t_0}^{1-\beta} \\
\leq &c_2 + c_3 \sum_{t=0}^{T-1} (v_{t+1}^y)^\beta \eta_t^2 \| \grad_x f(x_t, y_t) \|^2 + \frac{\eta^y c_1^{\frac{1-\beta}{\beta}}}{1-\beta} \\
= &\underbrace{c_2 + \frac{\eta^y c_1^{\frac{1-\beta}{\beta}}}{1-\beta}}_{c_4} + c_3 (\eta^x)^2 \sum_{t=0}^{T-1} \frac{(v_{t+1}^y)^\beta}{\max \{v_{t+1}^x, v_{t+1}^y\}^{2\alpha}} \| \grad_x f(x_t, y_t) \|^2 \\
\le &c_4 + c_3 (\eta^x)^2 \sum_{t=0}^{T-1} \frac{1}{(v_{t+1}^x)^{2\alpha - \beta}} \| \grad_x f(x_t, y_t) \|^2 \\
\leq &c_4 + c_3 (\eta^x)^2 \left( \frac{1 + \log v_T^x - \log v_0^x}{(v_0^x)^{2\alpha - \beta - 1}} \cdot \mathbf{1}_{2\alpha - \beta < 1} + \frac{(v_T^x)^{1-2\alpha + \beta}}{1-2\alpha + \beta} \cdot \mathbf{1}_{2\alpha - \beta < 1} \right).
\end{align*}
Here, the second inequality is from Lemma \ref{le.adaptive bound}. The third inequality is due to the definition of $t_0$. And We use Lemma \ref{le.adaptive bound} in the last inequality. 
Finally, from Lemma \ref{le.adaptive bound}, we also know that $\eta^y (v_T^y)^{1 - \beta} \le \text{LHS}$, which completes the proof.
\end{proof}
\begin{lemma}\label{le.bound square norm x}
Under the same assumptions as Lemma~\ref{le.bound sqaure norm y and vTy}, the following inequality holds:
\begin{align*}
        &\sum_{t=0}^{T-1}\eta_t\|\grad_xf(x_t,y_t)\|^2\\
        \le &2 \Delta \Phi + \left(2Gc_R+L_{\Phi}\bar{c}\right) (\eta^{x})^2 \left( \frac{1 + \log v_T^x - \log v_0^x}{(v_0^x)^{2\alpha - 1}} \cdot \mathbf{1}_{2\alpha\ge1} + \frac{(v_T^x)^{1 - 2\alpha}}{1 - 2\alpha}\cdot\mathbf{1}_{2\alpha<1} \right)\\
        &+ c_5c_4+c_5 c_3 (\eta^x)^2 \left( \frac{1 + \log v_T^x - \log v_0^x}{(v_0^x)^{2\alpha - \beta - 1}} \cdot \mathbf{1}_{2\alpha - \beta < 1} + \frac{(v_T^x)^{1-2\alpha + \beta}}{1-2\alpha + \beta} \cdot \mathbf{1}_{2\alpha - \beta < 1} \right)
    \end{align*}
    where $c_5 = \frac{\eta^x \kappa^2}{\eta^y (v_{t_0}^y)^{\alpha - \beta}}$.
\end{lemma}
\begin{proof}
    From the smoothness of $\Phi(\cdot)$ established in Lemma \ref{le.lips smooth}, we have 
    \begin{align*}
\Phi(x_{t+1})
\leq &\Phi(x_t) + \langle \grad \Phi(x_{t}), \Exp_{x_{t}}^{-1}(x_{t+1}) \rangle + \frac{L_{\Phi}}{2}d^2(x_t,x_{t+1}) \\
=&\Phi(x_t) -\eta_t \langle \grad \Phi(x_{t}), \grad_x f(x_t, y_t) \rangle + \langle \grad \Phi(x_{t}), \Exp_{x_{t}}^{-1}(x_{t+1})-\eta_t\grad_x f(x_t, y_t)  \rangle + \frac{L_{\Phi}}{2}d^2(x_t,x_{t+1})\\
\le &\Phi(x_t) -\eta_t \langle \grad \Phi(x_{t}), \grad_x f(x_t, y_t) \rangle + G\|\Exp_{x_{t}}^{-1}(x_{t+1})-\eta_t\grad_x f(x_t, y_t)\| + \frac{L_{\Phi}}{2}d^2(x_t,x_{t+1})\\
\le &\Phi(x_t) -\eta_t \langle \grad \Phi(x_{t}), \grad_x f(x_t, y_t) \rangle  + \left(Gc_R+\frac{L_{\Phi}\bar{c}}{2}\right)\eta_t^2\|\grad_x f(x_t, y_t)\|^2\\
= &\Phi(x_t) - \eta_t \|\grad_x f(x_t, y_t)\|^2 + \eta_t \langle \grad_x f(x_t, y_t) - \grad \Phi(x_t), \grad_x f(x_t, y_t) \rangle \\
&+ \left(Gc_R+\frac{L_{\Phi}\bar{c}}{2}\right) \eta_t^2 \|\grad_x f(x_t, y_t)\|^2 \\
\leq &\Phi(x_t) - \eta_t \|\grad_x f(x_t, y_t)\|^2 + \frac{\eta_t}{2} \|\grad_x f(x_t, y_t)\|^2 + \frac{\eta_t}{2} \|\grad_x f(x_t, y_t) - \grad \Phi(x_t)\|^2 \\
&+ \left(Gc_R+\frac{L_{\Phi}\bar{c}}{2}\right) \eta_t^2 \|\grad_x f(x_t, y_t)\|^2 \\
= &\Phi(x_t) - \frac{\eta_t}{2} \|\grad_x f(x_t, y_t)\|^2 + \left(Gc_R+\frac{L_{\Phi}\bar{c}}{2}\right) \eta_t^2 \|\grad_x f(x_t, y_t)\|^2 + \frac{\eta_t}{2} \|\grad_x f(x_t, y_t) - \grad \Phi(x_t)\|^2 \\
=& \Phi(x_t) - \frac{\eta_t}{2} \|\grad_x f(x_t, y_t)\|^2 + \left(Gc_R+\frac{L_{\Phi}\bar{c}}{2}\right) \eta_t^2 \|\grad_x f(x_t, y_t)\|^2\\ 
&+ \frac{\eta^x}{2 \max\{ (v_{t+1}^x)^\alpha, (v_{t+1}^y)^\alpha \}} \|\grad_x f(x_t, y_t) - \grad \Phi(x_t)\|^2 \\
\leq &\Phi(x_t) - \frac{\eta_t}{2} \|\grad_x f(x_t, y_t)\|^2 + \left(Gc_R+\frac{L_{\Phi}\bar{c}}{2}\right) \eta_t^2 \|\grad_x f(x_t, y_t)\|^2 \\
&+ \frac{\eta^x}{2 (v_{t_0}^y)^{\alpha - \beta}(v_{t+1}^y)^{\beta}} \|\grad_x f(x_t, y_t) - \grad \Phi(x_t)\|^2\\
\le &\Phi(x_t) - \frac{\eta_t}{2} \|\grad_x f(x_t, y_t)\|^2 + \left(Gc_R+\frac{L_{\Phi}\bar{c}}{2}\right) \eta_t^2 \|\grad_x f(x_t, y_t)\|^2 \\
&+ \frac{\eta^x\kappa^2}{2 (v_{t_0}^y)^{\alpha - \beta}(v_{t+1}^y)^{\beta}} \|\grad_y f(x_t, y_t)\|^2\\
= &\Phi(x_t) - \frac{\eta_t}{2} \|\grad_x f(x_t, y_t)\|^2 + \left(Gc_R+\frac{L_{\Phi}\bar{c}}{2}\right) \eta_t^2 \|\grad_x f(x_t, y_t)\|^2 \\
&+ \frac{\eta^x\kappa^2}{2 \eta^y(v_{t_0}^y)^{\alpha - \beta}(v_{t+1}^y)^{\beta}}\gamma_t \|\grad_y f(x_t, y_t)\|^2.
\end{align*}
Here, the second inequality uses the boundedness of $\|\grad \Phi(x_t)\| \le G$. The third inequality follows from Assumption~\ref{ass.retraction}, while the fourth inequality is a consequence of Cauchy–Schwarz and Young's inequality. The penultimate inequality additionally uses the fact that $0 \le t_0 \le t+1$ and the monotonicity of $v_t$. In the final inequality, we apply the $L_1$-smoothness and strong concavity of $f(x, \cdot)$, which together imply
\(
\|\grad_x f(x_t, y_t) - \grad \Phi(x_t)\| \leq L_1 \|y_t - y_t^*\| \leq 2\kappa \|\grad_y f(x_t, y_t)\|.
\)

By telescoping and rearranging the terms, we have
\begin{equation}\label{eq.sum gradx bound}
    \begin{split}
&\sum_{t=0}^{T-1} \eta_t \|\grad_x f(x_t, y_t)\|^2\\
\le &2\underbrace{(\Phi(x_0) - \Phi^*)}_{\Delta \Phi} + \left(2Gc_R+L_{\Phi}\bar{c}\right) \sum_{t=0}^{T-1} \eta_t^2 \|\grad_x f(x_t, y_t)\|^2 +\underbrace{\frac{\eta^x \kappa^2}{\eta^y (v_{t_0}^y)^{\alpha - \beta}}}_{c_5}\sum_{t=0}^{T-1} \gamma_t\|\grad_y f(x_t, y_t)\|^2 \\
= &2 \Delta \Phi + \sum_{t=0}^{T-1} \frac{\left(2Gc_R+L_{\Phi}\bar{c}\right) (\eta^x)^2}{\max \left\{ (v_{t+1}^x)^2, (v_{t+1}^y)^2 \right\}^{\alpha}} \|\grad_x f(x_t, y_t)\|^2 + c_5 \sum_{t=0}^{T-1} \gamma_t \|\grad_y f(x_t, y_t)\|^2 \\
\leq &2 \Delta \Phi + \sum_{t=0}^{T-1} \frac{\left(2Gc_R+L_{\Phi}\bar{c}\right) (\eta^x)^2}{(v_{t+1}^x)^{2\alpha}} \|\grad_x f(x_t, y_t)\|^2 + c_5 \sum_{t=0}^{T-1} \gamma_t \|\grad_y f(x_t, y_t)\|^2 \\
\leq &2 \Delta \Phi + \left(2Gc_R+L_{\Phi}\bar{c}\right) (\eta^{x})^2 \left( \frac{1 + \log v_T^x - \log v_0^x}{(v_0^x)^{2\alpha - 1}} \cdot \mathbf{1}_{2\alpha\ge1} + \frac{(v_T^x)^{1 - 2\alpha}}{1 - 2\alpha}\cdot\mathbf{1}_{2\alpha<1} \right) + c_5 \sum_{t=0}^{T-1} \gamma_t \|\grad_y f(x_t, y_t)\|^2,
    \end{split}
\end{equation}
where we use Lemma \ref{le.adaptive bound} in the last inequality. Finally, applying Lemma \ref{le.bound sqaure norm y and vTy} to the above inequality completes the proof.
\end{proof}
\textbf{Proof of Lemma \ref{le.bound vTx case 1}}
\begin{proof}
    From Lemma \ref{le.bound square norm x} and $v_T^y\le v_T^x$, we have 
    \begin{equation}\label{eq.first case}
        \begin{split}
            &\sum_{t=0}^{T-1} \| \grad_x f(x_t, y_t) \|^2 \\
\leq &\frac{2 \Delta \Phi (v_T^x)^\alpha}{\eta^x} + \left(2Gc_R+L_{\Phi}\bar{c}\right)\eta^x \left( \frac{(v_T^x)^\alpha (1 + \log v_T^x - \log v_0^x)}{(v_0^x)^{2\alpha-1}} \cdot \mathbf{1}_{2\alpha \geq 1} + \frac{(v_T^x)^{1-\alpha}}{1-2\alpha} \cdot \mathbf{1}_{2\alpha < 1} \right) \\
&+ \frac{c_4 c_5 (v_T^x)^\alpha}{\eta^x} + c_3 c_5 \eta^x \left( \frac{(v_T^x)^\alpha (1 + \log v_T^x - \log v_0^x)}{(v_0^x)^{2\alpha-\beta-1}} \cdot \mathbf{1}_{2\alpha-\beta \geq 1} + \frac{(v_T^x)^{1-\alpha+\beta}}{1-2\alpha+\beta} \cdot \mathbf{1}_{2\alpha-\beta < 1} \right) \\
= &\frac{2 \Delta \Phi (v_T^x)^\alpha}{\eta^x} + \left(2Gc_R+L_{\Phi}\bar{c}\right)\eta^x \left( \frac{(v_T^x)^\alpha (v_T^x)^{\frac{1-\alpha}{2}} (v_T^x)^{\frac{\alpha-1}{2}} (1 + \log v_T^x - \log v_0^x)}{(v_0^x)^{2\alpha-1}} \cdot \mathbf{1}_{2\alpha \geq 1} + \frac{(v_T^x)^{1-\alpha}}{1-2\alpha} \cdot \mathbf{1}_{2\alpha < 1} \right) \\
&+ \frac{c_4 c_5 (v_T^x)^\alpha}{\eta^x} + c_3 c_5 \eta^x \left( \frac{(v_T^x)^\alpha (v_T^x)^{\frac{1-\alpha}{2}} (v_T^x)^{\frac{\alpha-1}{2}} (1 + \log v_T^x - \log v_0^x)}{(v_0^x)^{2\alpha-\beta-1}} \cdot \mathbf{1}_{2\alpha-\beta \geq 1} + \frac{(v_T^x)^{1-\alpha+\beta}}{1-2\alpha+\beta} \cdot \mathbf{1}_{2\alpha-\beta < 1} \right) \\
\leq &\frac{2 \Delta \Phi (v_T^x)^\alpha}{\eta^x} + \left(2Gc_R+L_{\Phi}\bar{c}\right)\eta^x \left( \frac{2e^{(1-\alpha)(1-\log v_0^x)/2} (v_T^x)^{\frac{1+\alpha}{2}}}{e(1-\alpha)(v_0^x)^{2\alpha-1}} \cdot \mathbf{1}_{2\alpha \geq 1} + \frac{(v_T^x)^{1-\alpha}}{1-2\alpha} \cdot \mathbf{1}_{2\alpha < 1} \right) \\
&+ \frac{c_4 c_5 (v_T^x)^\alpha}{\eta^x} + c_3 c_5 \eta^x \left( \frac{2e^{(1-\alpha)(1-\log v_0^x)/2} (v_T^x)^{\frac{1+\alpha}{2}}}{e(1-\alpha)(v_0^x)^{2\alpha-\beta-1}} \cdot \mathbf{1}_{2\alpha-\beta \geq 1} + \frac{(v_T^x)^{1-\alpha+\beta}}{1-2\alpha+\beta} \cdot \mathbf{1}_{2\alpha-\beta < 1} \right),
        \end{split}
    \end{equation}
where we used \( x^{-m}(c + \log x) \leq \frac{e^{cm}}{em} \) for \( x > 0 \), \( m > 0 \) and \( c \in \mathbb{R} \) in the last inequality. Finally, applying the inequality:  
if \( x \le \sum_{i=1}^n b_i x^{\alpha_i} \) for \( x > 0 \), \( 0 < \alpha_i < 1 \), and \( b_i > 0 \),  
then \( x \le n \sum_{i=1}^n b_i^{1/(1 - \alpha_i)} \),  
to the inequality \eqref{eq.first case} (whose LHS equals \( v_T^x - v_0^x \)), we arrive at the bound \eqref{eq.vTx bound case 1}.
\end{proof}
\textbf{Proof of Lemma \ref{le.bound vTx case 2}}
\begin{proof}
    Since $v_T^x < v_T^y$, we derive an upper bound for $v_T^x$ by leveraging the upper bound of $v_T^y$ given in \eqref{eq.vTy bound}. The proof proceeds by considering two separate cases:
\begin{itemize}
    \item [(1)] $2\alpha < 1 + \beta$: In this case, we have
    \begin{align*}
&\sum_{t=0}^{T-1} \|\grad_x f(x_t, y_t)\|^2 \\
\leq &\left( \frac{2\Delta \Phi + c_4 c_5}{\eta^x} + \left(2Gc_R+L_{\Phi}\bar{c}\right)\eta^x \left( \frac{1 + \log v_T^x - \log v_0^x}{(v_0^x)^{2\alpha-1}} \cdot \mathbf{1}_{2\alpha \geq 1} + \frac{(v_T^x)^{1-2\alpha}}{1-2\alpha} \cdot \mathbf{1}_{2\alpha < 1} \right)\right.\\
&\left.+ \frac{c_3 c_5 \eta^x (v_T^x)^{1-2\alpha+\beta}}{1-2\alpha+\beta} \right) \left( \frac{c_4}{\eta^y} + \frac{c_3 (\eta^x)^2 (v_T^x)^{1-2\alpha+\beta}}{\eta^y(1-2\alpha+\beta)} \right)^{\frac{\alpha}{1-\beta}} \\
\leq &\left( \frac{2\Delta \Phi + c_4 c_5}{\eta^x (v_0^x)^{1-2\alpha+\beta}} + \left(2Gc_R+L_{\Phi}\bar{c}\right)\eta^x \left( \frac{1 + \log v_T^x - \log v_0^x}{(v_0^x)^{2\alpha-1} (v_T^x)^{1-2\alpha+\beta}} \cdot \mathbf{1}_{2\alpha \geq 1} + \frac{1}{(1-2\alpha)(v_0^x)^\beta} \cdot \mathbf{1}_{2\alpha < 1} \right)\right.\\
&\left.+ \frac{c_3 c_5 \eta^x}{1-2\alpha+\beta} \right)
\left( \frac{c_4}{\eta^y(v_0^x)^{1-2\alpha+\beta}} + \frac{c_3 (\eta^x)^2}{\eta^y(1-2\alpha+\beta)} \right)^{\frac{\alpha}{1-\beta}} \cdot (v_T^x)^{1-2\alpha+\beta+ \frac{(1-2\alpha+\beta)\alpha}{1-\beta}} \\
\leq &\left( \frac{2\Delta \Phi + c_4 c_5}{\eta^x (v_0^x)^{1-2\alpha+\beta}} + \left(2Gc_R+L_{\Phi}\bar{c}\right)\eta^x \left( \frac{e^{(1-2\alpha+\beta)(1-\log v_0^x)}}{e(1-2\alpha+\beta)(v_0^x)^{2\alpha-1}} \cdot \mathbf{1}_{2\alpha \geq 1} + \frac{1}{(1-2\alpha)(v_0^x)^\beta} \cdot \mathbf{1}_{2\alpha < 1} \right)\right.\\
&\left.+ \frac{c_3 c_5 \eta^x}{1-2\alpha+\beta} \right)  \left( \frac{c_4}{\eta^y(v_0^x)^{1-2\alpha+\beta}} + \frac{c_3 (\eta^x)^2}{\eta^y(1-2\alpha+\beta)} \right)^{\frac{\alpha}{1-\beta}} \cdot (v_T^x)^{1-2\alpha+\beta+ \frac{(1-2\alpha+\beta)\alpha}{1-\beta}},
\end{align*}
where we use \eqref{eq.vTy bound} in the first inequality. 
Noting that $\alpha > \beta$, we have
\[
1 - 2\alpha + \beta + \frac{(1-2\alpha+\beta)\alpha}{1-\beta} \leq \frac{(1-\alpha)\alpha}{1-\beta} + 1 - \alpha = 1 + \frac{\alpha(\beta-\alpha)}{1-\beta} < 1.
\]
Applying the same reasoning as in \eqref{eq.vTx bound case 1}, we obtain
\begin{equation}\label{eq.vTx bound case 2-1}
    \begin{split}
v_T^x 
\leq &2 \left[\left( \frac{2\Delta \Phi + c_4 c_5}{\eta^x (v_0^x)^{1-2\alpha+\beta}} + \left(2Gc_R+L_{\Phi}\bar{c}\right)\eta^x \left( \frac{e^{(1-2\alpha+\beta)(1-\log v_0^x)}}{e(1-2\alpha+\beta)(v_0^x)^{2\alpha-1}} \cdot \mathbf{1}_{2\alpha \geq 1} + \frac{1}{(1-2\alpha)(v_0^x)^\beta} \cdot \mathbf{1}_{2\alpha < 1} \right)\right.\right.\\
&\left.\left.+ \frac{c_3 c_5 \eta^x}{1-2\alpha+\beta} \right)  \left( \frac{c_4}{\eta^y(v_0^x)^{1-2\alpha+\beta}} + \frac{c_3 (\eta^x)^2}{\eta^y(1-2\alpha+\beta)} \right)^{\frac{\alpha}{1-\beta}}\right]^{ \frac{1}{1-(1-2\alpha+\beta)\left(1+\frac{\alpha}{1-\beta}\right)}} + 2v_0^x ,
    \end{split}
\end{equation}
which gives us a constant RHS.
\item [(2)] $2\alpha\ge 1+\beta$: In this case, it holds that
\begin{align*}
&\sum_{t=0}^{T-1} \|\grad_x f(x_t, y_t)\|^2 \\
\leq &\left( \frac{2\Delta \Phi + c_4 c_5}{\eta^x} + \frac{\left(2Gc_R+L_{\Phi}\bar{c}\right)\eta^x (1 + \log v_T^x - \log v_0^x)}{(v_0^x)^{2\alpha-1}} + \frac{c_3 c_5 \eta^x (1 + \log v_T^x - \log v_0^x)}{(v_0^x)^{2\alpha-\beta-1}} \right) \\
& \left( \frac{c_4}{\eta^y} + \frac{c_3 (\eta^x)^2 (1 + \log v_T^x - \log v_0^x)}{\eta^y(v_0^x)^{2\alpha-\beta-1}} \right)^{\frac{\alpha}{1-\beta}} \\
\leq &\left( \frac{2\Delta \Phi + c_4 c_5}{\eta^x (v_0^x)^{1/4}} + \frac{\left(2Gc_R+L_{\Phi}\bar{c}\right)\eta^x (1 + \log v_T^x - \log v_0^x)}{(v_0^x)^{2\alpha-1} (v_T^x)^{1/4}} + \frac{c_3 c_5 \eta^x (1 + \log v_T^x - \log v_0^x)}{(v_0^x)^{2\alpha-\beta-1} (v_T^x)^{1/4}} \right) \\
& \left( \frac{c_4}{\eta^y(v_0^x)^{\frac{(1-\beta)}{4\alpha}}} + \frac{c_3 (\eta^x)^2 (1 + \log v_T^x - \log v_0^x)}{\eta^y(v_0^x)^{2\alpha-\beta-1} (v_T^x)^{\frac{(1-\beta)}{4\alpha}}} \right)^{\frac{\alpha}{1-\beta}} \cdot (v_T^x)^{1/2} \\
\leq &\left( \frac{2\Delta \Phi + c_4 c_5}{\eta^x (v_0^x)^{1/4}} + \frac{\left(8Gc_R+4L_{\Phi}\bar{c}\right)\eta^x e^{(1-\log v_0^x)/4}}{e(v_0^x)^{2\alpha-1}} + \frac{4c_3 c_5 \eta^x e^{(1-\log v_0^x)/4}}{e(v_0^x)^{2\alpha-\beta-1}} \right) \\
& \left( \frac{c_4}{\eta^y(v_0^x)^{\frac{(1-\beta)}{4\alpha}}} + \frac{4c_3 \alpha (\eta^x)^2 e^{(1-\beta)(1-\log v_0^x)/(4\alpha)}}{\eta^ye(1-\beta)(v_0^x)^{2\alpha-\beta-1}} \right)^{\frac{\alpha}{1-\beta}} \cdot (v_T^x)^{1/2},
\end{align*}
where we use \eqref{eq.vTy bound} in the first inequality. Using the same argument as in \eqref{eq.vTx bound case 1}, we conclude 
\begin{equation}\label{eq.vTx bound case 2-2'}
    \begin{split}
v_T^x
\leq &2 \left[ \left( \frac{2\Delta \Phi + c_1 c_5}{\eta^x (v_0^x)^{1/4}} + \frac{\left(8Gc_R+4L_{\Phi}\bar{c}\right)\eta^x e^{(1-\log v_0^x)/4}}{e(v_0^x)^{2\alpha-1}} + \frac{4c_1 c_4 \eta^x e^{(1-\log v_0^x)/4}}{e(v_0^x)^{2\alpha-\beta-1}} \right)\right. \\
&\left. \left( \frac{c_5}{\eta^y(v_0^x)^{\frac{(1-\beta)}{4\alpha}}} + \frac{4c_4 \alpha (\eta^x)^2 e^{(1-\beta)(1-\log v_0^x)/(4\alpha)}}{\eta^ye(1-\beta)(v_0^x)^{2\alpha-\beta-1}} \right)^{\frac{\alpha}{1-\beta}}\right]^2 + 2v_0^x.
    \end{split}
\end{equation}
\end{itemize}
Hence, combining the two cases, the desired bound follows, which completes the proof of the lemma.
\end{proof}
\begin{theorem}[Restatement of Theorem~\ref{the.deterministic convergence}]\label{the.restate of deterministic}
    Under Assumptions~\ref{ass.l-smooth}, \ref{ass.g-convex}, \ref{ass.optimal exist and stationary condition hold}, \ref{ass.manifold}, and \ref{ass.retraction}, suppose Algorithm~\ref{alg.RAGDA} is run with deterministic gradient oracles. Then for any \( 0 < \beta < \alpha < 1 \), after \( T \) iterations, it holds that
\begin{equation}\label{eq.convergence x}
\sum_{t=0}^{T-1} \|\grad_x f(x_t, y_t)\|^2 \leq \max\{7C_1, 2C_2\},
\end{equation}
\begin{equation}\label{eq.convergence y}
\sum_{t=0}^{T-1} \|\grad_y f(x_t, y_t)\|^2 \leq C_4
\end{equation}
where the constants \( C_1 \), \( C_2 \), \( C_3 \), and \( C_4 \) are defined as follows:
\begin{equation}
    \begin{split}
        C_1 = &v_0^x + \left( \frac{2 \Delta \Phi}{\eta^x} \right)^{\frac{1}{1-\alpha}} + \left( \frac{\left(4Gc_R+2L_{\Phi}\bar{c}\right)\eta^x e^{(1-\alpha)(1-\log v_0^x)/2}}{e(1-\alpha)(v_0^x)^{2\alpha-1}} \right)^{\frac{2}{1-\alpha}} \cdot \mathbf{1}_{2\alpha \geq 1}\\
   &+ \left( \frac{\left(2Gc_R+L_{\Phi}\bar{c}\right)\eta^x}{1-2\alpha} \right)^{\frac{1}{\alpha}} \cdot \mathbf{1}_{2\alpha < 1}+ \left( \frac{c_4 c_5}{\eta^x} \right)^{\frac{1}{1-\alpha}} \\
&+ \left( \frac{2c_3 c_5 \eta^x e^{(1-\alpha)(1-\log v_0^x)/2}}{e(1-\alpha)\left(v_0^x\right)^{2\alpha-\beta-1}} \right)^{\frac{2}{1-\alpha}} \cdot \mathbf{1}_{2\alpha-\beta \ge 1} 
+\left( \frac{c_3 c_5 \eta^x}{1-2\alpha+\beta} \right)^{\frac{1}{\alpha-\beta}} \cdot \mathbf{1}_{2\alpha-\beta < 1},
    \end{split}
\end{equation}
\begin{equation}
    \begin{split}
        C_2 = &v_0^x + \left[\left( \frac{2\Delta \Phi + c_1 c_5}{\eta^x (v_0^x)^{1-2\alpha+\beta}} + \left(2Gc_R+L_{\Phi}\bar{c}\right)\eta^x \left( \frac{e^{(1-2\alpha+\beta)(1-\log v_0^x)}}{e(1-2\alpha+\beta)(v_0^x)^{2\alpha-1}} \cdot \mathbf{1}_{2\alpha \geq 1} + \frac{1}{(1-2\alpha)(v_0^x)^\beta} \cdot \mathbf{1}_{2\alpha < 1} \right)\right.\right.\\
&\left.\left.+ \frac{c_1 c_4 \eta^x}{1-2\alpha+\beta} \right)  \left( \frac{c_5}{\eta^y(v_0^x)^{1-2\alpha+\beta}} + \frac{c_4 (\eta^x)^2}{\eta^y(1-2\alpha+\beta)} \right)^{\frac{\alpha}{1-\beta}}\right]^{ \frac{1}{1-(1-2\alpha+\beta)\left(1+\frac{\alpha}{1-\beta}\right)}}\cdot\mathbf{1}_{2\alpha-\beta<1} \\
&+ \left[ \left( \frac{2\Delta \Phi + c_1 c_5}{\eta^x (v_0^x)^{1/4}} + \frac{\left(8Gc_R+4L_{\Phi}\bar{c}\right)\eta^x e^{(1-\log v_0^x)/4}}{e(v_0^x)^{2\alpha-1}} + \frac{4c_1 c_4 \eta^x e^{(1-\log v_0^x)/4}}{e(v_0^x)^{2\alpha-\beta-1}} \right)\right. \\
&\left. \left( \frac{c_5}{\eta^y(v_0^x)^{\frac{(1-\beta)}{4\alpha}}} + \frac{4c_4 \alpha (\eta^x)^2 e^{(1-\beta)(1-\log v_0^x)/(4\alpha)}}{\eta^ye(1-\beta)(v_0^x)^{2\alpha-\beta-1}} \right)^{\frac{\alpha}{1-\beta}}\right]^2\cdot\mathbf{1}_{2\alpha-\beta\ge1},
    \end{split}
\end{equation}
\begin{equation}
    C_3 = \max\{7C_1,2C_2\},
\end{equation}
\begin{equation}
    C_4 = \left(\frac{c_5}{\eta^y} + \frac{c_4}{\eta^y} (\eta^x)^2 \left( \frac{1 + \log C_3 - \log v_0^x}{(v_0^x)^{2\alpha - \beta - 1}} \cdot \mathbf{1}_{2\alpha - \beta < 1} + \frac{(C_3)^{1-2\alpha + \beta}}{1-2\alpha + \beta} \cdot \mathbf{1}_{2\alpha - \beta < 1} \right)\right)^{\frac{1}{1-\beta}},
\end{equation}
with the auxiliary constants defined as:
\[
\Delta \Phi = \Phi(x_0) - \Phi^*, \quad
c_1 = \max \left\{ \frac{4 \eta^y\mu L_1}{\mu + L_1}, \eta^y (\mu + L_1)/\left(\zeta\bar{c}+\frac{2Gc_R}{\mu}\right) \right\},
\]
\[
c_2 = 4(\mu + L_1) \left( \frac{1}{\mu^2} + \frac{\bar{c}\eta^y}{(v_{t_0}^y)^\beta} \right) c_2^{1/\beta}, \quad
c_3 = (\mu + L_1) \left( \frac{2\kappa^2\bar{c}}{(v_0^y)^\beta} + \frac{(\mu + L_1)\kappa^2\bar{c}}{\eta^y \mu L_1} \right),
\]
\[
c_4 = c_2 + \frac{\eta^y c_1^{\frac{1-\beta}{\beta}}}{1-\beta}, \quad
c_5 = \frac{\eta^x \kappa^2}{\eta^y (v_{t_0}^y)^{\alpha - \beta}}.
\]
\end{theorem}
\begin{proof}
    Combing Lemma \ref{le.bound sqaure norm y and vTy}, Lemma \ref{le.bound vTx case 1} and Lemma \ref{le.bound vTx case 2}, we obtain the desired bounds on \( v_T^x \) and \( v_T^y \). Since by definition
    \[
    \sum_{t=0}^{T-1} \|\grad_x f(x_t, y_t)\|^2 \leq v_T^x \quad \text{and} \quad \sum_{t=0}^{T-1} \|\grad_y f(x_t, y_t)\|^2 \leq v_T^y,
    \]
    the result follows immediately.   
\end{proof}
\section{Proof of Theorem \ref{the.stochastic convergence}}\label{app.stochastic convergence}
\begin{lemma}\label{le.stochastic auxiliary lemma}
Suppose Assumptions \ref{ass.g-convex}, \ref{ass.manifold}, \ref{ass.retraction}, and \ref{ass.grad bound} hold. For any integers $t_1 < t_2$, suppose that for all $t \in [t_1, t_2 - 1]$, and any constants $\lambda_t > 0$ and $S_t \ge 0$, the following inequality holds:
\[d^2(y_{t+1}, y_{t+1}^*) \leq (1 + \lambda_t) d^2(y_{t+1}, y_t^*) + S_t,\]

the cumulative expected dual suboptimality over $[t_1, t_2)$ satisfies:
\begin{align*}
    \mathbb{E} \left[ \sum_{t=t_1}^{t_2-1} (f(x_t, y_t^*) - f(x_t, y_t)) \right] &\leq \mathbb{E} \left[ \sum_{t=t_1+1}^{t_2-1} \left( \frac{1 - \gamma_t \mu}{2\gamma_t} d^2(y_t , y_t^*) - \frac{1}{2\gamma_t (1 + \lambda_t)} d^2(y_{t+1}, y_{t+1}^*) \right) \right]\\
    &+ \mathbb{E} \left[ \sum_{t=t_1}^{t_2-1} \frac{(\zeta\bar{c}+\frac{2G}{\mu}c_R)\gamma_t}{2} \left\| \grad_y \tilde{f}(x_t, y_t) \right\|^2 \right] + \mathbb{E} \left[ \sum_{t=t_1}^{t_2-1} \frac{S_t}{2\gamma_t (1 + \lambda_t)} \right].
\end{align*}
\end{lemma}
\begin{proof}
By the assumption of the lemma, we have
\begin{align*}
    &d^2(y_{t+1}, y_{t+1}^*)\\
    \leq &(1 + \lambda_t) d^2(y_{t+1}, y_t^*) + S_t\\
    \le &(1 + \lambda_t) \left(d^2(y_t, y_t^*) + \zeta d^2(y_{t},y_{t+1}) - 2  \left\langle \Exp_{y_t}^{-1}(y_{t+1}), \Exp_{y_t}^{-1}(y^*_t) \right\rangle \right) + S_t\\
    \le &(1 + \lambda_t)\bigg(d^2(y_t, y_t^*)+\zeta\bar{c}\gamma_t^2\left\|\grad_y \tilde{f}(x_t, y_t)\right\|^2 - 2  \left\langle \Exp_{y_t}^{-1}(y_{t+1})-\Retr_{y_t}^{-1}(y_{t+1}), \Exp_{y_t}^{-1}(y^*_t) \right\rangle\\
    &+2\gamma_t\left\langle \grad_y \tilde{f}(x_t, y_t), \Exp_{y_t}^{-1}(y^*_t) \right\rangle\bigg)+S_t\\
    \le &(1 + \lambda_t)\bigg(d^2(y_t, y_t^*)+\zeta\bar{c}\gamma_t^2\left\|\grad_y \tilde{f}(x_t, y_t)\right\|^2 +2\gamma_t\left\langle \grad_y \tilde{f}(x_t, y_t), \Exp_{y_t}^{-1}(y^*_t) \right\rangle\\
    &+ 2 \frac{G}{\mu}\left\|\Exp_{y_t}^{-1}(y_{t+1})-\Retr_{y_t}^{-1}(y_{t+1})\right\| \bigg)+S_t\\
    \le& (1 + \lambda_t) \bigg(d^2(y_t, y_t^*) + (\zeta\bar{c}+\frac{2G}{\mu}c_R)\gamma_t^2 \left\| \grad_y \tilde{f}(x_t, y_t) \right\|^2 + 2 \gamma_t \left\langle \grad_y \tilde{f}(x_t, y_t), \Exp_{y_t}^{-1}(y^*_t) \right\rangle\\
    &+\gamma_t \mu d^2(y_t, y_t^*) - \gamma_t \mu d^2(y_t, y_t^*) \bigg) + S_t.
\end{align*}
Here, the second inequality follows from Lemma \ref{le.trigonometric distance bound}. The third and fifth inequalities use Assumption \ref{ass.retraction}. The fourth inequality follows from the strong concavity of \( f \) in \( y \).

Dividing both sides by \( \gamma_t(1 + \lambda_t) \) and rearranging terms yields
\begin{align*}
    &2 \left\langle \grad_y \tilde{f}(x_t, y_t), \Exp_{y_t}^{-1}(y^*_t)\right\rangle - \mu d^2(y_t, y_t^*)\\
    \le &\frac{1 - \gamma_t \mu}{\gamma_t} d^2(y_t, y_t^*) - \frac{1}{\gamma_t (1 + \lambda_t)} d^2(y_{t+1}, y_{t+1}^*) + (\zeta\bar{c}+\frac{2G}{\mu}c_R)\gamma_t \left\| \grad_y \tilde{f}(x_t, y_t) \right\|^2 + \frac{S_t}{\gamma_t (1 + \lambda_t)}.
\end{align*}

By summing both sides from \( t = t_1 \) to \( t_2 - 1 \), we obtain
\begin{align*}
    &\sum_{t=t_1}^{t_2-1} \left( \left\langle \grad_y \tilde{f}(x_t, y_t), \Exp_{y_t}^{-1}(y^*_t) \right\rangle - \frac{\mu}{2} d^2(y_t, y_t^*) \right)\\
    \leq &\sum_{t=t_1}^{t_2-1} \left( \frac{1 - \gamma_t \mu}{2 \gamma_t} d^2(y_t, y_t^*) - \frac{1}{2 \gamma_t (1 + \lambda_t)} d^2(y_{t+1}, y_{t+1}^*) \right) + \sum_{t=t_1}^{t_2-1} \frac{(\zeta\bar{c}+\frac{2G}{\mu}c_R)\gamma_t}{2} \left\| \grad_y \tilde{f}(x_t, y_t) \right\|^2\\
    &+ \sum_{t=t_1}^{t_2-1} \frac{S_t}{2 \gamma_t (1 + \lambda_t)}.
\end{align*}
By taking expectation on both sides, we get
\begin{align*}
    \mathbb{E} [{\rm LHS}] \geq &\mathbb{E} \left[ \sum_{t=t_1}^{t_2-1} \mathbb{E}_{\xi_t^y} \left[ \left( \left\langle \grad_y \tilde{f}(x_t, y_t), \Exp_{y_t}^{-1}(y^*_t) \right\rangle - \frac{\mu}{2} d^2(y_t, y_t^*) \right) \right] \right]\\
    = &\mathbb{E} \left[ \sum_{t=t_1}^{t_2-1} \left( \left\langle \grad_y f(x_t, y_t), \Exp_{y_t}^{-1}(y^*_t) \right\rangle - \frac{\mu}{2}d^2(y_t, y_t^*) \right) \right]\\
    \ge & \mathbb{E} \left[ \sum_{t=t_1}^{t_2-1} \left( f(x_t, y_t^*) - f(x_t, y_t) \right) \right]
\end{align*}
where the last inequality uses the strong concavity of \( f \) in \( y \). Combining the above bounds completes the proof.
\end{proof}
\begin{lemma}\label{le.v<C}
    Suppose Assumptions~\ref{ass.l-smooth}, \ref{ass.g-convex}, \ref{ass.optimal exist and stationary condition hold}, \ref{ass.manifold}, \ref{ass.retraction} and \ref{ass.grad bound} hold. For any constant $C>0$, if $t_3$ is the first iteration such that \( v_{t}^y > C \), it follows
\begin{align*}
&\mathbb{E}\left[ \sum_{t=0}^{t_3-1} (f(x_t, y_t^*) - f(x_t, y_t)) \right] \\
\leq &\mathbb{E} \left[ \sum_{t=0}^{t_3-1} \left( \frac{1 - \gamma_t \mu}{2 \gamma_t} d^2(y_t, y_t^*) - \frac{1}{\gamma_t (2 + \mu \gamma_t)} d^2(y_{t+1},y_{t+1}^*) \right) \right] + \mathbb{E} \left[ \sum_{t=0}^{t_3-1} \frac{(\zeta\bar{c}+\frac{2G}{\mu}c_R)\gamma_t}{2} \| \grad_y \tilde{f}(x_t, y_t) \|^2 \right] \\
& + \frac{\kappa^2\bar{c} (\mu \eta^y C^\beta + 2C^{2\beta}) (\eta^x)^2}{2\mu (\eta^y)^2} \mathbb{E} \left[ \frac{1 + \log v_{t_3}^x - \log v_{0}^x}{(v_{0}^x)^{2\alpha - 1}} \cdot 1_{\alpha > 0.5} + \frac{(v_{t_3}^x)^{1-2\alpha}}{1-2\alpha} \cdot 1_{\alpha < 0.5} \right].
\end{align*}
\end{lemma}
\begin{proof}
    By Young's inequality, we have
\[
d^2(y_{t+1}, y_{t+1}^*) \leq (1 + \lambda_t)d^2(y_{t+1}, y_t^*) + \left( 1 + \frac{1}{\lambda_t} \right) d^2(y_{t+1}^*, y_t^*).
\]

By choosing \(\lambda_t = \frac{\mu \gamma_t}{2}\) and applying Lemma \ref{le.stochastic auxiliary lemma}, we have
\begin{align*}
&\mathbb{E} \left[ \sum_{t=0}^{t_3-1} (f(x_t, y_t^*) - f(x_t, y_t)) \right] \\
\leq &\mathbb{E} \left[ \sum_{t=0}^{t_3-1} \left( \frac{1 - \gamma_t \mu}{2 \gamma_t} d^2(y_t , y_t^*) - \frac{1}{\gamma_t (2 + \mu \gamma_t)} d^2(y_{t+1}, y_{t+1}^*) \right) \right] \\
& + \mathbb{E} \left[ \sum_{t=0}^{t_3-1} \frac{(\zeta\bar{c}+\frac{2G}{\mu}c_R)\gamma_t}{2} \| \grad_y \tilde{f}(x_t, y_t) \|^2 \right] + \mathbb{E} \left[ \sum_{t=0}^{t_3-1} \frac{(1 + \frac{2}{\mu \gamma_t})}{\gamma_t (2 + \mu \gamma_t)} d^2(y_{t+1}^*, y_t^*) \right].
\end{align*}
We now proceed to bound the last term. Observe that
\begin{align*}
&\mathbb{E} \left[ \sum_{t=0}^{t_3-1} \frac{(1 + \frac{2}{\mu \gamma_t})}{\gamma_t (2 + \mu \gamma_t)} d^2(y_{t+1}^*, y_t^*) \right] \\
\leq &\mathbb{E} \left[ \sum_{t=0}^{t_3-1} \frac{(1 + \frac{2}{\mu \gamma_t})}{2 \gamma_t} d^2(y_{t+1}^*, y_t^*) \right] \\
= &\mathbb{E} \left[ \sum_{t=0}^{t_3-1} \frac{\mu \eta^y (v_{t+1}^y)^{\beta} + 2 (v_{t+1}^y)^{2\beta}}{2\mu (\eta^y)^2} d^2(y_{t+1}^*, y_t^*) \right] \\
\leq &\frac{\mu \eta^y C^\beta + 2C^{2\beta}}{2\mu (\eta^y)^2} \mathbb{E} \left[ \sum_{t=0}^{t_3-1}d^2(y_{t+1}^*, y_t^*) \right],
\end{align*}
where the last inequality uses the bound \( v_{t+1}^y \leq C \) for all \( t = 0, \dots, t_3 - 1 \). By Lemma \ref{le.lips smooth} and Assumption \ref{ass.retraction}, we have
\begin{align*}
    &\sum_{t=0}^{t_3-1} d^2(y_{t+1}^*, y_t^*) \\
    \leq &\kappa^2 \bar{c}\sum_{t=0}^{t_3-1} \eta_t^2\|\grad_x\tilde{f}(x_t,y_t)\|^2\\
   = &\kappa^2\bar{c} (\eta^x)^2 \sum_{t=0}^{t_3-1} \frac{1}{\max \{ v_{t+1}^x, v_{t+1}^y \}^{2 \alpha}} \left\| \grad_x \tilde{f}(x_t, y_t) \right\|^2 \\
   \leq &\kappa^2\bar{c} (\eta^x)^2 \sum_{t=0}^{t_3-1} \frac{1}{(v_{t+1}^x)^{2 \alpha}} \left\| \grad_x \tilde{f}(x_t, y_t) \right\|^2 \\
   \leq &\kappa^2\bar{c} (\eta^x)^2 \left( \frac{v_0^x}{(v_0^x)^{2 \alpha}} + \sum_{t=0}^{t_3-1} \frac{1}{(v_{t+1}^x)^{2 \alpha}} \left\| \grad_x \tilde{f}(x_t, y_t) \right\|^2 \right) \\
   \leq &\kappa^2\bar{c} (\eta^x)^2 \left( \frac{1 + \log v_{t_3}^x - \log v_0^x}{(v_0^x)^{2\alpha-1}} \cdot 1_{\alpha \geq 0.5} + \frac{(v_{t_3}^x)^{1-2\alpha}}{1-2\alpha} \cdot 1_{\alpha < 0.5} \right)
\end{align*}
where the final bound follows from Lemma \ref{le.adaptive bound}.
Substituting this estimate back into the previous inequality completes the proof.
\end{proof}
\begin{lemma}\label{le.v>C}
    Suppose the assumptions in Lemma \ref{le.v<C} hold. Let \( t_3 \) be defined as in Lemma \ref{le.v<C} and $0< 2\beta\le \alpha<1$, then we have
\begin{align*}
&\mathbb{E} \left[ \sum_{t=t_3}^{T-1} (f(x_t, y_t^*) - f(x_t, y_t)) \right] \\
\leq &\mathbb{E} \left[ \sum_{t=t_3}^{T-1} \frac{1 - \gamma_t \mu}{2\gamma_t} d^2( y_t , y_t^* ) - \frac{1}{\gamma_t (2 + \mu \gamma_t)} d^2(y_{t+1}, y_{t+1}^*) \right] + \mathbb{E} \left[ \sum_{t=t_3}^{T-1} \frac{(\zeta\bar{c}+\frac{2G}{\mu}c_R)\gamma_t}{2} \| \grad_y \tilde{f}(x_t, y_t) \|^2 \right] \\
& + \left(\zeta\bar{c}{\kappa}^2 + \frac{(2\kappa\sqrt{\bar{c}}(v_0^y)^{\alpha}+\kappa Gc_R\eta^x)^2(\eta^x)^{\beta/\alpha}}{\mu \eta^y (v_0^y)^{2\alpha}}\right)
\frac{(\eta^x)^{2-\beta/\alpha}}{2(1 - \alpha)\eta^y (v_0^y)^{\alpha - 2\beta}} \mathbb{E} \left[ (v_T^x)^{1-\alpha} \right] \\
& + \frac{2 \kappa^2 (\eta^x)^2}{\mu (\eta^y)^2 C^{2\alpha - 2\beta}} \mathbb{E} \left[ \sum_{t=t_3}^{T-1} \| \grad_x f(x_t, y_t) \|^2 \right] + \left( \frac{1}{\mu} + \frac{\eta^y}{(v_0^y)^\beta} \right) \frac{4 \kappa \eta^x G^2}{\eta^y (v_0^y)^{\alpha}} \mathbb{E} \left[ (v_T^y)^\beta \right].
\end{align*}
\end{lemma}
\begin{proof}
From Lemma \ref{le.trigonometric distance bound}, we have
    \begin{equation}\label{eq.equation contain term D}
        \begin{split}
            d^2( y_{t+1}, y_{t+1}^*) \le &d^2( y_{t+1}, y_t^*) + \zeta d^2( y_t^*, y_{t+1}^*) - 2 \langle \Exp^{-1}_{y^*_t}(y_{t+1}), \Exp^{-1}_{y^*_t}(y_{t+1}^*) \rangle \\
    \leq &d^2( y_{t+1}, y_t^*) + \zeta\kappa^2\bar{c} \eta_t^2 \| \grad_x \tilde{f}(x_t, y_t) \|^2 - 2 \langle \Exp^{-1}_{y^*_t}(y_{t+1}), \Exp^{-1}_{y^*_t}(y_{t+1}^*) \rangle \\
    = &d^2(y_{t+1}, y_t^*) + \zeta\bar{c}\kappa^2 \eta_t^2 \| \grad_x \tilde{f}(x_t, y_t) \|^2 \underbrace{-2\left\langle\Exp^{-1}_{y^*_t}(y_{t+1}), \operatorname{D} y^*(x_t)[-\eta_t\grad_x\tilde{f}(x_t,y_t)]\right\rangle}_{\rm (C)} \\
    &\underbrace{-2 \left\langle \Exp^{-1}_{y^*_t}(y_{t+1}), \Exp^{-1}_{y^*_t}(y_{t+1}^*)  - \operatorname{D}y^*(x_t)[-\eta_t\grad_x\tilde{f}(x_t,y_t)] \right\rangle}_{\rm (D)},
        \end{split}
    \end{equation}
    where the second inequality uses Lemma \ref{le.lips smooth} and Assumption \ref{ass.retraction}. 
    For the inner product Term \text{(C)}, we apply Lemma \ref{le.lips smooth} and Cauchy‘s inequality, yielding for any $\lambda_t > 0$
    \begin{equation}\label{eq.bound Term C}
        \begin{split}
            &-2\left\langle\Exp^{-1}_{y^*_t}(y_{t+1}), \operatorname{D} y^*(x_t)[-\eta_t\grad_x\tilde{f}(x_t,y_t)]\right\rangle\\
        =&2\eta_t\left\langle\Exp^{-1}_{y^*_t}(y_{t+1}), \operatorname{D} y^*(x_t)[\grad_x f(x_t,y_t)]\right\rangle+2\eta_t\left\langle\Exp^{-1}_{y^*_t}(y_{t+1}), \operatorname{D} y^*(x_t)[\grad_x\tilde{f}(x_t,y_t)-\grad_x f(x_t,y_t)]\right\rangle\\
        \le &2\eta_td(y_{t+1}, y^*_t)\|\operatorname{D}y^*(x_t)\|\|\grad_xf(x_t,y_t)\|+2\eta_t\left\langle\Exp^{-1}_{y^*_t}(y_{t+1}), \operatorname{D} y^*(x_t)[\grad_x\tilde{f}(x_t,y_t)-\grad_x f(x_t,y_t)]\right\rangle\\
        \le &2d(y_{t+1},y_t^*)\kappa\eta_t\|\grad_xf(x_t,y_t)\|+2\eta_t\left\langle\Exp^{-1}_{y^*_t}(y_{t+1}), \operatorname{D} y^*(x_t)[\grad_x\tilde{f}(x_t,y_t)-\grad_x f(x_t,y_t)]\right\rangle\\
        \le &\lambda_td^2(y_{t+1},y^*_t)+\frac{\kappa^2\eta_t^2}{\lambda_t}\|\grad_xf(x_t,y_t)\|^2+2\eta_t\left\langle\Exp^{-1}_{y^*_t}(y_{t+1}), \operatorname{D} y^*(x_t)[\grad_x\tilde{f}(x_t,y_t)-\grad_x f(x_t,y_t)]\right\rangle.
        \end{split}
    \end{equation}
    For the Term \text{(D)}, applying Young's inequality, we have for any $\tau > 0$
\begin{align*}
    &-2 \left\langle \Exp^{-1}_{y^*_t}(y_{t+1}), \Exp^{-1}_{y^*_t}(y_{t+1}^*)  - \operatorname{D}y^*(x_t)[-\eta_t\grad_x\tilde{f}(x_t,y_t)] \right\rangle\\
    \le &2d(y_{t+1},y_t^*)\left\|\Exp^{-1}_{y^*_t}(y_{t+1}^*)  - \operatorname{D}y^*(x_t)[-\eta_t\grad_x\tilde{f}(x_t,y_t)] \right\|\\
    \le
    &2d(y_{t+1},y_t^*)\left(\left\|\Exp^{-1}_{y^*_t}(y_{t+1}^*)  - \operatorname{D}y^*(x_t)[\Exp_{x_t}^{-1}(x_{t+1})] \right\|+\left\|\operatorname{D}y^*(x_t)[\Exp_{x_t}^{-1}(x_{t+1})-\eta_t\grad_x\tilde{f}(x_t,y_t)]\right\|\right)\\
    \le 
    &2d(y_{t+1},y_t^*)\left(\kappa d(x_t,x_{t+1})+\|\operatorname{D}y^*(x_t)\|d(x_t,x_{t+1})+\|\operatorname{D}y^*(x_t)\|\|\Exp_{x_t}^{-1}(x_{t+1})-\eta_t\grad_x\tilde{f}(x_t,y_t)\|\right)\\
    \le &2d(y_{t+1},y_t^*)\left(2\kappa\sqrt{\bar{c}}\eta_t\|\grad_x \tilde{f}(x_t, y_t)\|+\kappa c_R\eta_t^2\|\grad_x \tilde{f}(x_t, y_t)\|^2\right)\\
    \le &(2\kappa\sqrt{\bar{c}}+\kappa Gc_R\eta_t)\eta_td(y_{t+1},y_t^*)\|\grad_x\tilde{f}(x_t,y_t)\|\\
    \le &\frac{\tau (2\kappa\sqrt{\bar{c}}+\kappa Gc_R\eta_t)^2\eta_t^{\beta/\alpha}}{2}d^2(y_{t+1}, y_t^*) + \frac{\eta_t^{2-\beta/\alpha}}{2 \tau} \| \grad_x \tilde{f}(x_t, y_t)\|^2.
\end{align*}
where the fourth inequality is from Assumption \ref{ass.retraction} and Lemma \ref{le.lips smooth}.

Combining the previous results leads to the inequality
\begin{align*}
d^2(y_{t+1}, y_{t+1}^*) \leq &\left(1 + \lambda_t + \frac{\tau (2\kappa\sqrt{\bar{c}}+\kappa Gc_R\eta_t)^2\eta_t}{2}\right) d^2(y_{t+1}, y_t^*)  + \left(\zeta\bar{c}{\kappa}^2\eta_t + \frac{1}{2 \tau}\right) \eta_t \| \grad_x \tilde{f}(x_t, y_t)\|^2 \\
& + \frac{{\kappa}^2 \eta_t^2}{\lambda_t} \| \grad_x f(x_t, y_t)\|^2 +2\eta_t\left\langle\Exp^{-1}_{y^*_t}(y_{t+1}), \operatorname{D} y^*(x_t)[\grad_x\tilde{f}(x_t,y_t)-\grad_x f(x_t,y_t)]\right\rangle.
\end{align*}
We further note that the adaptive learning rate \(\eta_t\) satisfies
\[
\eta_t = \frac{\eta^x}{\max \{v_{t+1}^x, v_{t+1}^y\}^{\alpha}} \leq \frac{\eta^x}{(v_{t+1}^y)^{\alpha}} \leq \frac{\eta^x}{(v_0^y)^{\alpha}},
\]
and
\[
\eta_t \leq \frac{\eta^x}{(v_{t+1}^y)^{\alpha}} = \frac{\eta^x}{(v_{t+1}^y)^{\alpha - 2\beta} (v_{t+1}^y)^{2\beta}} \leq \frac{\eta^x}{(v_0^y)^{\alpha - 2\beta} (v_{t+1}^y)^{2\beta}}.
\]
Plugging the these into the previous result implies
\begin{align*}
&d^2(y_{t+1}, y_{t+1}^*)\\
\leq &\left(1 + \lambda_t + \frac{\tau (2\kappa\sqrt{\bar{c}}(v_0^y)^{\alpha}+\kappa Gc_R\eta^x)^2(\eta^x)^{\beta/\alpha}}{2(v_0^y)^{2\alpha}(v^y_{t+1})^\beta}\right) d^2(y_{t+1}, y_t^*)  + \left(\frac{\zeta\bar{c}{\kappa}^2\eta^x}{(v^y_0)^{\alpha}}+ \frac{1}{2 \tau}\right) \eta_t^{2-\beta/\alpha} \| \grad_x \tilde{f}(x_t, y_t)\|^2 \\
& + \frac{{\kappa}^2 \eta_t^2}{\lambda_t} \| \grad_x f(x_t, y_t)\|^2 +2\eta_t\left\langle\Exp^{-1}_{y^*_t}(y_{t+1}), \operatorname{D} y^*(x_t)[\grad_x\tilde{f}(x_t,y_t)-\grad_x f(x_t,y_t)]\right\rangle.
\end{align*}
We now choose the parameters \(\lambda_t = \frac{\mu \eta^y}{4(v_{t+1}^y)^{\beta}}\) and \(\tau = \frac{\mu \eta^y (v_0^y)^{2\alpha}}{2(2\kappa\sqrt{\bar{c}}(v_0^y)^{\alpha}+\kappa Gc_R\eta^x)^2(\eta^x)^{\beta/\alpha}}\). Substituting these into the above inequality gives
\begin{align*}
&d^2(y_{t+1}, y_{t+1}^*)\\
\leq &\left(1 + \frac{\mu \eta^y}{2(v_{t+1}^y)^{\beta}}\right) d^2(y_{t+1}, y_t^*)  + \left(\zeta\bar{c}{\kappa}^2 + \frac{(2\kappa\sqrt{\bar{c}}(v_0^y)^{\alpha}+\kappa Gc_R\eta^x)^2(\eta^x)^{\beta/\alpha}}{\mu \eta^y (v_0^y)^{2\alpha}}\right) \eta_t^{2-\beta/\alpha} \| \grad_x \tilde{f}(x_t, y_t)\|^2\\
&+\frac{4{\kappa}^2 \left( v_{t+1}^y \right)^{\beta} \eta_t^2}{\mu \eta^y} \| \grad_x f(x_t, y_t) \|^2 + 2\eta_t\left\langle\Exp^{-1}_{y^*_t}(y_{t+1}), \operatorname{D} y^*(x_t)[\grad_x\tilde{f}(x_t,y_t)-\grad_x f(x_t,y_t)]\right\rangle.
\end{align*}
We now apply Lemma \ref{le.stochastic auxiliary lemma} to obtain
\begin{align*}
&\mathbb{E} \left[ \sum_{t=t_3}^{T-1} \left( f(x_t, y_t^*) - f(x_t, y_t) \right) \right] \\
\leq &\mathbb{E} \left[ \sum_{t=t_3}^{T-1} \left( \frac{1 - \gamma_t \mu}{2 \gamma_t}d^2(y_t, y_t^*) - \frac{1}{\gamma_t (2 + \mu \gamma_t)}d^2( y_{t+1}, y_{t+1}^*) \right) \right] + \mathbb{E} \left[ \sum_{t=t_3}^{T-1} \frac{(\zeta\bar{c}+\frac{2G}{\mu}c_R)\gamma_t}{2} \| \grad_y \tilde{f}(x_t, y_t) \|^2 \right] \\
&+ \underbrace{\mathbb{E} \left[ \sum_{t=t_3}^{T-1} \frac{1}{\gamma_t (2 + \mu \gamma_t)} \left(\zeta\bar{c}{\kappa}^2 + \frac{(2\kappa\sqrt{\bar{c}}(v_0^y)^{\alpha}+\kappa Gc_R\eta^x)^2(\eta^x)^{\beta/\alpha}}{\mu \eta^y (v_0^y)^{2\alpha}}\right) \eta_t^{2-\beta/\alpha} \| \grad_x \tilde{f}(x_t, y_t) \|^2 \right]}_{\rm (E)} \\ 
&+ \underbrace{\mathbb{E} \left[ \sum_{t=t_3}^{T-1} \frac{4{\kappa}^2 \left( v_{t+1}^y \right)^{\beta} \eta_t^2}{\gamma_t (2 + \mu \gamma_t) \mu \eta^y} \| \grad_x f(x_t, y_t) \|^2 \right]}_{\rm (F)} \\
& + \underbrace{\mathbb{E} \left[ \sum_{t=t_3}^{T-1} \frac{2\eta_t}{\gamma_t (2 + \mu \gamma_t)}\left\langle\Exp^{-1}_{y^*_t}(y_{t+1}), \operatorname{D} y^*(x_t)[\grad_x\tilde{f}(x_t,y_t)-\grad_x f(x_t,y_t)]\right\rangle\right]}_{\rm (G)}
\end{align*}
We now derive upper bounds for each of the remaining terms.

For the Term {\rm (E)}
\begin{align*}
\text{\rm Term (E)} \leq &\left(\zeta\bar{c}{\kappa}^2 + \frac{(2\kappa\sqrt{\bar{c}}(v_0^y)^{\alpha}+\kappa Gc_R\eta^x)^2(\eta^x)^{\beta/\alpha}}{\mu \eta^y (v_0^y)^{2\alpha}}\right) \mathbb{E} \left[ \sum_{t=t_3}^{T-1} \frac{\eta_t^{2-\beta/\alpha}}{2 \gamma_t} \| \grad_x \tilde{f}(x_t, y_t) \|^2 \right] \\
= &\left(\zeta\bar{c}{\kappa}^2 + \frac{(2\kappa\sqrt{\bar{c}}(v_0^y)^{\alpha}+\kappa Gc_R\eta^x)^2(\eta^x)^{\beta/\alpha}}{\mu \eta^y (v_0^y)^{2\alpha}}\right) \mathbb{E} \left[ \sum_{t=t_3}^{T-1} \frac{(\eta^x)^{2-\beta/\alpha} (v_{t+1}^y)^\beta}{2 \eta^y \max \{ v_{t+1}^x, v_{t+1}^y \}^{2\alpha-\beta}} \| \grad_x \tilde{f}(x_t, y_t) \|^2 \right] \\
\leq &\left(\zeta\bar{c}{\kappa}^2 + \frac{(2\kappa\sqrt{\bar{c}}(v_0^y)^{\alpha}+\kappa Gc_R\eta^x)^2(\eta^x)^{\beta/\alpha}}{\mu \eta^y (v_0^y)^{2\alpha}}\right) \mathbb{E} \left[ \sum_{t=t_3}^{T-1} \frac{(\eta^x)^{2-\beta/\alpha} (v_{t+1}^y)^\beta}{2 \eta^y (v_{t+1}^y)^{\beta}(v_{t+1}^y)^{\alpha-2\beta}(v_{t+1}^x)^{\alpha}} \| \grad_x \tilde{f}(x_t, y_t) \|^2 \right] \\
\leq &\left(\zeta\bar{c}{\kappa}^2 + \frac{(2\kappa\sqrt{\bar{c}}(v_0^y)^{\alpha}+\kappa Gc_R\eta^x)^2(\eta^x)^{\beta/\alpha}}{\mu \eta^y (v_0^y)^{2\alpha}}\right) \mathbb{E} \left[ \sum_{t=t_3}^{T-1} \frac{(\eta^x)^{2-\beta/\alpha}}{2 \eta^y (v_0^y)^{\alpha-2\beta}(v_{t+1}^x)^{\alpha}} \| \grad_x \tilde{f}(x_t, y_t) \|^2 \right]\\
\leq &\left(\zeta\bar{c}{\kappa}^2 + \frac{(2\kappa\sqrt{\bar{c}}(v_0^y)^{\alpha}+\kappa Gc_R\eta^x)^2(\eta^x)^{\beta/\alpha}}{\mu \eta^y (v_0^y)^{2\alpha}}\right)
\mathbb{E} \left[\frac{(\eta^x)^{2-\beta/\alpha}}{2\eta^y(v_0^y)^{\alpha-2\beta}} \left( \frac{v_t^x}{(v_0^x)^\alpha} + \sum_{t = 0}^{T - 1} \frac{1}{(v_{t+1}^x)^\alpha} \left\| \grad_x \tilde{f}(x_t, y_t) \right\|^2 \right) \right]\\
\leq &\left(\zeta\bar{c}{\kappa}^2 + \frac{(2\kappa\sqrt{\bar{c}}(v_0^y)^{\alpha}+\kappa Gc_R\eta^x)^2(\eta^x)^{\beta/\alpha}}{\mu \eta^y (v_0^y)^{2\alpha}}\right)
\frac{(\eta^x)^{2-\beta/\alpha}}{2(1 - \alpha)\eta^y (v_0^y)^{\alpha - 2\beta}} \mathbb{E} \left[ (v_T^x)^{1 - \alpha} \right]
\end{align*}
where we used Lemma \ref{le.adaptive bound} in the last step.

For the term {\rm (F)}, by using bounds on $\eta_t$ and monotonicity of $v_t^y$, we obtain:
\begin{equation}\label{eq.bound term F}
    \begin{split}
        \text{\rm Term (F)}  \leq &\mathbb{E} \left[ \sum_{t=t_3}^{T-1} \frac{2{\kappa}^2 \left( v_{t+1}^y \right)^\beta \eta_t^2}{\gamma_t \mu \eta^y} \| \grad_x f(x_t, y_t) \|^2 \right] \\
= &\frac{2{\kappa}^2 (\eta^x)^2}{\mu (\eta^y)^2} \mathbb{E} \left[ \sum_{t=t_3}^{T-1} \frac{(v_{t+1}^y)^{2\beta}}{\max \{ v_{t+1}^x, v_{t+1}^y \}^{2\alpha}} \| \grad_x f(x_t, y_t) \|^2 \right] \\
\leq &\frac{2{\kappa}^2 (\eta^x)^2}{\mu (\eta^y)^2} \mathbb{E} \left[ \sum_{t=t_3}^{T-1} \frac{(v_{t+1}^y)^{2\beta}}{(v_{t+1}^y)^{2\alpha}} \| \grad_x f(x_t, y_t) \|^2 \right] \\
\leq &\frac{2{\kappa}^2 (\eta^x)^2}{\mu (\eta^y)^2} \mathbb{E} \left[ \frac{1}{(v_{t_3 + 1}^y)^{2\alpha - 2\beta}} \sum_{t=t_3}^{T-1} \| \grad_x f(x_t, y_t) \|^2 \right] \\
\leq &\frac{2{\kappa}^2 (\eta^x)^2}{\mu (\eta^y)^2 C^{2\alpha - 2\beta}} \mathbb{E} \left[ \sum_{t=t_3}^{T-1} \| \grad_x f(x_t, y_t) \|^2 \right].
    \end{split}
\end{equation}

For the term {\rm (G)}, denote 
\(
m_t := \frac{2}{\gamma_t (2+\mu \gamma_t)} \left\langle\Exp^{-1}_{y^*_t}(y_{t+1}), \operatorname{D} y^*(x_t)[\grad_x\tilde{f}(x_t,y_t)-\grad_x f(x_t,y_t)]\right\rangle.
\)
Since \( y^*(\cdot) \) is \({\kappa}\)-Lipschitz as in Lemma \ref{le.lips smooth}, we can bound \( |m_t| \) as
\begin{align*}
|m_t| &\leq \frac{1}{\gamma_t} d(y_{t+1},y_t^*)\| \operatorname{D} y^*(x_t) \| \left( \| \grad_x \tilde{f}(x_t, y_t) \| + \| \grad_x f(x_t, y_t) \| \right) \\
&\leq \frac{{\kappa}}{\gamma_t}d(y_{t+1}, y_t^*) \left( \| \grad_x \tilde{f}(x_t, y_t) \| + \| \grad_x f(x_t, y_t) \| \right) \\
&\leq \frac{{\kappa}}{\gamma_t}\left(d(y_{t},y_t^*)+d(y_{t+1},y_{t})\right) \left( \| \grad_x \tilde{f}(x_t, y_t) \| + \| \grad_x f(x_t, y_t) \| \right) \\
&\leq \frac{{\kappa}}{\gamma_t} \left( \frac{1}{\mu} \| \grad_y f(x_t, y_t) \| + \| \gamma_t \grad_y \tilde{f}(x_t, y_t) \| \right) \left( \| \grad_x \tilde{f}(x_t, y_t) \| + \| \grad_x f(x_t, y_t) \| \right) \\
&\leq \underbrace{\frac{2G {\kappa}}{\gamma_{T-1}} \left( \frac{G}{\mu} + \frac{\eta^y G}{(v_0^y)^\beta} \right)}_{M}.
\end{align*}
Moreover, since \(\gamma_t\) and \(y_{t+1}\) are independent of \(\xi_t^x\), it holds that \(\mathbb{E}_{\xi_t^x} [m_t] = 0\). Next, We now evaluate Term (G).
\begin{equation}\label{eq.bound term G}
    \begin{split}
        \text{Term } (G) &= \mathbb{E} \left[ \sum_{t=t_3}^{T-1} \eta_t m_t \right] \\
&= \mathbb{E} \left[ \eta_{t_3} m_{t_3} + \sum_{t=t_3+1}^{T-1} \eta_{t-1} m_t + \sum_{t=t_3+1}^{T-1} (\eta_t - \eta_{t-1}) m_t \right] \\
&\leq \mathbb{E} \left[ \frac{\eta^x}{(v_0^y)^\alpha} M + \sum_{t=t_3+1}^{T-1} \eta_{t-1} \mathbb{E}_{\xi_t^x} [m_t] + \sum_{t=t_3+1}^{T-1} (\eta_{t-1} - \eta_t) (-m_t) \right]\\
&\leq \mathbb{E} \left[ \frac{\eta^x}{(v_0^y)^\alpha} M + \sum_{t=t_3+1}^{T-1} (\eta_{t-1} - \eta_t) M \right] \\
&\leq \mathbb{E} \left[ \frac{2\eta^x}{(v_0^y)^\alpha} M \right] \\
&= \left( \frac{1}{\mu} + \frac{\eta^y}{(v_0^y)^\beta} \right) \frac{4{\kappa} \eta^x G^2}{\eta^y (v_0^y)^\alpha} \mathbb{E} \left[ (v_T^y)^\beta \right].
    \end{split}
\end{equation}
Combining all the bounds derived above concludes the proof.
\end{proof}
\begin{lemma}\label{le.4}
Suppose the assumptions in Lemma~\ref{le.v<C} hold. Then the following inequality holds:
\[
\mathbb{E} \left[ \sum_{t=0}^{T-1} \left( \frac{1 - \gamma_t \mu}{2 \gamma_t}d^2( y_t, y_t^*) - \frac{1}{\gamma_t (2 + \mu \gamma_t)}  d^2(y_{t+1}, y_{t+1}^*) \right) \right] 
\leq \frac{(v_0^y)^\beta G^2}{2 \mu^2 \eta^y} + \frac{(2 \beta G)^{\frac{1}{1-\beta} + 2} G^2}{4 \mu^{\frac{1}{1-\beta} + 3} (\eta^y)^{\frac{1}{1-\beta} + 2} (v_0^y)^{2 - 2\beta}}.
\]
\end{lemma}
\begin{proof}
From $\frac{2(v_t^y)^{2\beta}}{2(v_t^y)^\beta +  \mu \eta^y}\ge (v_t^y)^\beta-\frac{\mu \eta^y}{2}$, it follows
\begin{align*}
    &\mathbb{E} \left[ \sum_{t=0}^{T-1} \left( \frac{1 - \gamma_t \mu}{2 \gamma_t} d^2( y_t, y_t^*) - \frac{1}{\gamma_t (2 + \mu \gamma_t)} d^2(y_{t+1}, y_{t+1}^*) \right) \right] \\
\leq &\left( \frac{(v_0^y)^\beta}{2 \eta^y} - \frac{\mu}{2} \right) d^2( y_0, y_0^*) 
+ \frac{1}{2 \eta^y} \sum_{t=1}^{T-1} \left( (v_{t+1}^y)^\beta - \mu \eta^y - \frac{2(v_t^y)^{2\beta}}{2(v_t^y)^\beta +  \mu \eta^y} \right) d^2(y_t, y_t^*)\\
\leq &\frac{(v_0^y)^\beta G^2}{2 \mu^2 \eta^y} + \frac{1}{2 \eta^y} \underbrace{\sum_{t=1}^{T-1} \left( (v_{t+1}^y)^\beta - \frac{\mu \eta^y}{2} - (v_t^y)^\beta \right)d^2(y_t, y_t^*)}_{\rm (H)}.
\end{align*}
To bound Term {\rm (H)}, we follow the strategy of~
\cite{yang2022nest, li2022tiada}
, which hinges on the observation that the quantity $(v_{t+1}^y)^\beta - \frac{\mu \eta^y}{2} - (v_t^y)^\beta$ can be positive for only a constant number of iterations. Specifically, if the term is positive at iteration $t$, then
\begin{equation}\label{eq.positive upper bound}
    \begin{split}
        0 < &(v_{t+1}^y)^\beta - (v_t^y)^\beta - \frac{\mu \eta^y}{2}\\
        = &\left( v_t^y + \| \grad_y \tilde{f}(x_t, y_t) \|^2 \right)^\beta - (v_t^y)^\beta - \frac{\mu \eta^y}{2}\\
        \leq &(v_t^y)^\beta \left( 1 + \frac{ \| \grad_y \tilde{f}(x_t, y_t) \|^2 }{v_t^y} \right)^\beta-(v_t^y)^\beta - \frac{\mu \eta^y}{2}\\
        \le &(v_t^y)^\beta \left( 1 + \frac{\beta \| \grad_y \tilde{f}(x_t, y_t) \|^2 }{v_t^y} \right)-(v_t^y)^\beta - \frac{\mu \eta^y}{2}\\
        = & \frac{\beta \| \grad_y \tilde{f}(x_t, y_t) \|^2 }{(v_t^y)^{1-\beta}}- \frac{\mu \eta^y}{2},
    \end{split}
\end{equation}
where the last inequality follows from Bernoulli’s inequality. Rearranging this inequality yields two conditions:
$$
    \begin{cases}
        \left\lVert \grad_y \tilde{f}(x_t, y_t) \right\rVert^2 > \frac{\mu \eta^y}{2 \beta} (v_t^y)^{1 - \beta} \geq \frac{\mu \eta^y}{2 \beta} (v_0^y)^{1 - \beta} \\
(v_t^y)^{1 - \beta} < \frac{2 \beta}{\mu \eta^y} \left\lVert \grad_y \tilde{f}(x_t, y_t) \right\rVert^2 \leq \frac{2 \beta G^2}{\mu \eta^y},
    \end{cases}
$$
This implies that whenever the term is positive, the gradient norm must be sufficiently large while the accumulated gradient norm \( v_{t+1}^y \) remains sufficiently small. Consequently, the number of such iterations is at most
\(
\frac{\left( \frac{2 \beta G^2}{\mu \eta^y} \right)^{\frac{1}{1 - \beta}}}{\frac{\mu \eta^y}{2 \beta} (v_0^y)^{1 - \beta}}.
\)
Moreover, when the term is positive, it is upper bounded by Equation \eqref{eq.positive upper bound}
\begin{align*}
\left( (v_{t+1}^y)^\beta - \frac{\mu \eta^y}{2} - (v_t^y)^\beta \right) d^2( y_t, y_t^*) 
&\leq \frac{\beta \left\lVert \grad_y \tilde{f}(x_t, y_t) \right\rVert^2}{(v_t^y)^{1 - \beta}} d^2(y_t, y_t^*) \\
&\leq \frac{\beta G^2}{(v_0^y)^{1 - \beta}} d^2(y_t, y_t^*) \\
&\leq \frac{\beta G^2}{\mu^2 (v_0^y)^{1 - \beta}} \| \grad_y f(x_t, y_t) \|^2 \\
&\leq \frac{\beta G^4}{\mu^2 (v_0^y)^{1 - \beta}}
\end{align*}
which is a constant. Thus, the total contribution of Term \text{(H)} is bounded by
\[
    \frac{(2 \beta G^2)^{\frac{1}{1 - \beta} + 2}}{2 \mu^{\frac{1}{1 - \beta} + 3} (\eta^y)^{\frac{1}{1 - \beta} + 1} (v_0^y)^{2 - 2 \beta}}.
\]
Substituting this into the earlier bound yields the desired result.
\end{proof}
\textbf{Proof of Lemma \ref{le.bound conclusion}}
\begin{proof}
By Lemma \ref{le.v<C} and Lemma \ref{le.v>C}, we have for any constant \(C\),
\begin{align*}
&\mathbb{E}\left[\sum_{t=0}^{T-1}\left(f(x_{t},y^{*}_{t})-f(x_{t}, y_{t})\right)\right] \\
\leq&\mathbb{E}\left[\sum_{t=0}^{T-1}\left(\frac{1-\gamma_{t}\mu}{2 \gamma_{t}}\|y_{t}-y^{*}_{t}\|^{2}-\frac{1}{\gamma_{t}(2+\mu\gamma_{t})}\|y_{t+1}-y^{*}_{t+1}\|^{2}\right)\right] \\
& +\mathbb{E}\left[\sum_{t=0}^{T-1}\frac{(\zeta\bar{c}+\frac{2G}{\mu}c_R)\gamma_{t}}{2}\Big{\|}\grad_{y}\widetilde{f}(x_{t},y_{t})\Big{\|}^{2}\right]+\frac{2{\kappa}^{2}\left(\eta^{x}\right)^{2}}{\mu\left(\eta^{y}\right)^{2}C^{2\alpha-2\beta}}\mathbb{E}\left[\sum_{t=0}^{T-1}\|\grad_x f(x_{t},y_{t})\|^{2}\right] \\
&+\frac{{\kappa}^{2}\bar{c}\left(\mu\eta^{y}C^{\beta}+2C^{2\beta}\right)\left(\eta^{x}\right)^{2}}{2\mu\left(\eta^{y}\right)^{2}}\mathbb{E}\left[\frac{1+\log v_{T}^{x}-\log v_{0}^{x}}{\left(v_{0}^{x}\right)^{2\alpha-1}}\cdot\mathbf{1}_{\alpha\geq 0.5}+\frac{\left(v_{T}^{x}\right)^{1-2\alpha}}{1-2\alpha}\cdot\mathbf{1}_{\alpha<0.5}\right] \\
&+\left(\zeta\bar{c}{\kappa}^2 + \frac{(2\kappa\sqrt{\bar{c}}(v_0^y)^{\alpha}+\kappa Gc_R\eta^x)^2(\eta^x)^{\beta/\alpha}}{\mu \eta^y (v_0^y)^{2\alpha}}\right)
\frac{(\eta^x)^{2-\beta/\alpha}}{2(1 - \alpha)\eta^y (v_0^y)^{\alpha - 2\beta}}\mathbb{E}\left[\left(v_{T}^{x}\right)^{1-\alpha}\right] \\
& +\left(\frac{1}{\mu}+\frac{\eta^{y}}{\left(v_{0}^{y}\right)^{\beta}}\right)\frac{4\kappa\eta^{x}G^{2}}{\eta^{y}\left(v_{0}^{y}\right)^{\alpha}}\mathbb{E}\left[\left(v_{T}^{y}\right)^{\beta}\right].
\end{align*}

The first term can be bounded by Lemma \ref{le.4}. For the second term, we have
\begin{align*}
\mathbb{E}&\left[\sum_{t=0}^{T-1}\frac{(\zeta\bar{c}+\frac{2G}{\mu}c_R)\gamma_{t}}{2}\Big{\|}\grad_{y}\widetilde{f}(x_{t},y_{t})\Big{\|}^{2}\right] \\
&= \mathbb{E}\left[\sum_{t=0}^{T-1}\frac{(\zeta\bar{c}+\frac{2G}{\mu}c_R)\eta^{y}}{2\left(v_{t+1}^{y}\right)^{\beta}}\Big{\|}\grad_{y}\widetilde{f}(x_{t},y_{t})\Big{\|}^{2}\right] \\
&\leq \frac{(\zeta\bar{c}+\frac{2G}{\mu}c_R)\eta^{y}}{2}\mathbb{E}\left[\frac{v_{0}^{y}}{\left(v_{0}^{y}\right)^{\beta}}+\sum_{t=0}^{T-1}\frac{1}{\left(v_{t+1}^{y}\right)^{\beta}}\Big{\|}\grad_{y}\widetilde{f}(x_{t},y_{t})\Big{\|}^{2}\right] \\
&\leq \frac{(\zeta\bar{c}+\frac{2G}{\mu}c_R)\eta^{y}}{2(1-\beta)}\mathbb{E}\left[\left(v_{T}^{y}\right)^{1-\beta}\right],
\end{align*}
where the last inequality follows from Lemma \ref{le.adaptive bound}. Then the proof is completed.
\end{proof}
\textbf{Proof of Lemma \ref{le.bound EvTy}}
\begin{proof}
    Note that since \( f(x, y) \) is \( L_1 \)-smooth and concave in \( y \), we have
\[
\|\grad_y f(x_t, y_t)\|^2 \le 2L_1 \left( f(x_t, y_t^*) - f(x_t, y_t) \right).
\]
Combine the above inequality with Lemma~\ref{le.bound conclusion}, yields the desired result.
\end{proof}
\textbf{Proof of Lemma \ref{le.bound EvTx}}
\begin{proof}
    From Lemma \ref{le.lips smooth}, we have
\begin{align*}
    \Phi(x_{t+1}) - \Phi(x_t) \leq & \left\langle \grad \Phi(x_t), \Exp_{x_t}^{-1}(x_{t+1}) \right\rangle + \frac{L_{\Phi}}{2} d^2(x_t,x_{t+1})\\
    \le & -\eta_t \left\langle \grad \Phi(x_t), \grad_x \widetilde{f}(x_t, y_t) \right\rangle+\frac{L_{\Phi}\bar{c}\eta_t^2}{2}\left\| \grad_x \widetilde{f}(x_t, y_t) \right\|^2 \\
    &+\left\langle \grad \Phi(x_t), \Exp_{x_t}^{-1}(x_{t+1})-\eta_t\grad_x \widetilde{f}(x_t, y_t) \right\rangle\\
    \le & -\eta_t \left\langle \grad \Phi(x_t), \grad_x \widetilde{f}(x_t, y_t) \right\rangle+\frac{L_{\Phi}\bar{c}\eta_t^2}{2}\left\| \grad_x \widetilde{f}(x_t, y_t) \right\|^2 \\
    &+\| \grad \Phi(x_t)\|\| \Exp_{x_t}^{-1}(x_{t+1})-\eta_t\grad_x \widetilde{f}(x_t, y_t)\|\\
    \le & -\eta_t \left\langle \grad \Phi(x_t), \grad_x \widetilde{f}(x_t, y_t) \right\rangle+\frac{(L_{\Phi}\bar{c}+2Gc_R)\eta_t^2}{2}\left\| \grad_x \widetilde{f}(x_t, y_t) \right\|^2 \\
\end{align*}
where the second inequality follows from Assumption \ref{ass.retraction}, and the last inequality uses Assumption \ref{ass.retraction} again together with the bound \( \| \grad \Phi(x_t) \| \le G \).

By multiplying both sides by \( \frac{1}{\eta_t} \) and taking expectation over the stochasticity at iteration \( t \), we obtain
\begin{align*}
    &\mathbb{E}\left[\frac{\Phi(x_{t+1}) - \Phi(x_t)}{\eta_t}\right]\\
    \leq &-\left\langle \grad \Phi(x_t), \grad_x f(x_t, y_t) \right\rangle + \frac{(L_{\Phi}\bar{c}+2Gc_R)}{2}\mathbb{E}\left[\eta_t \left\| \grad_x \widetilde{f}(x_t, y_t) \right\|^2 \right]\\
    = &-\left\| \grad_x f(x_t, y_t) \right\|^2 + \left\langle \grad_x f(x_t, y_t) - \grad \Phi(x_t), \grad_x f(x_t, y_t) \right\rangle + \frac{(L_{\Phi}\bar{c}+2Gc_R)}{2} \mathbb{E}\left[\eta_t \left\| \grad_x \widetilde{f}(x_t, y_t) \right\|^2 \right]\\
    \le & -\left\| \grad_x f(x_t, y_t) \right\|^2 + \frac{1}{2} \left\| \grad_x f(x_t, y_t) \right\|^2 + \frac{1}{2} \left\| \grad_x f(x_t, y_t) - \grad \Phi(x_t) \right\|^2\\
    &+ \frac{(L_{\Phi}\bar{c}+2Gc_R)}{2} \mathbb{E}\left[\eta_t \left\| \grad_x \widetilde{f}(x_t, y_t) \right\|^2 \right]\\
    =&-\frac{1}{2} \left\| \grad_x f(x_t, y_t) \right\|^2 + \frac{1}{2} \left\| \grad_x f(x_t, y_t) - \grad \Phi(x_t) \right\|^2 + \frac{(L_{\Phi}\bar{c}+2 Gc_R)}{2} \mathbb{E}\left[\eta_t \left\| \grad_x \widetilde{f}(x_t, y_t) \right\|^2 \right].
\end{align*}

Summing over \( t = 0 \) to \( T - 1 \) and taking the total expectation yields
\begin{align*}
    &\mathbb{E}\left[\sum_{t=0}^{T-1} \left\| \grad_x f(x_t, y_t) \right\|^2 \right]\\  
\leq &\underbrace{2 \mathbb{E} \left[ \sum_{t=0}^{T-1} \frac{\Phi(x_t) - \Phi(x_{t+1})}{\eta_t} \right]}_{\rm (I)}
+ \underbrace{(L_{\Phi}\bar{c}+2 Gc_R) \mathbb{E} \left[ \sum_{t=0}^{T-1} \eta_t \left\| \grad_x \widetilde{f}(x_t, y_t) \right\|^2 \right]}_{\rm (J)}\\
&+ \underbrace{\mathbb{E} \left[ \sum_{t=0}^{T-1} \left\| \grad_x f(x_t, y_t) - \grad \Phi(x_t) \right\|^2 \right]}_{\rm (K)}.     
\end{align*}
For the term {\rm (I)}, using the definition of the adaptive step size $\eta_t$, we have 
\begin{align*}
    2\mathbb{E} \left[ \sum_{t=0}^{T-1} \frac{\Phi(x_t) - \Phi(x_{t+1})}{\eta_t} \right]
= &2\mathbb{E} \left[ \frac{\Phi(x_0)}{\eta_0}-\frac{\Phi(x_T)}{\eta_{T-1}} + \sum_{t=1}^{T-1} \Phi(x_t) \left( \frac{1}{\eta_t} - \frac{1}{\eta_{t-1}} \right) \right]\\
\leq &2\mathbb{E} \left[ \frac{\Phi_{\max}}{\eta_0}-\frac{\Phi^*}{\eta_{T-1}} + \sum_{t=1}^{T-1} \Phi_{\max} \left( \frac{1}{\eta_t} - \frac{1}{\eta_{t-1}} \right) \right]\\
= &2\mathbb{E} \left[ \frac{\Delta \Phi}{\eta_{T-1}} \right]
= 2\mathbb{E} \left[ \frac{\Delta \Phi}{\eta^x} \max \left\{ (v_T^x)^\alpha, (v_T^y)^\alpha \right\} \right].
\end{align*}
To bound term {\rm (J)}, applying Lemma \ref{le.stochastic auxiliary lemma} yields
\begin{align*}
    (L_{\Phi}\bar{c}+2Gc_R) \mathbb{E} \left[ \sum_{t=0}^{T-1} \eta_t \left\| \grad_x \widetilde{f}(x_t, y_t) \right\|^2 \right]
= &(L_{\Phi}\bar{c}+2Gc_R) \mathbb{E} \left[ \sum_{t=0}^{T-1} \frac{\eta^x}{\max \left\{ (v_{t+1}^x)^\alpha, (v_{t+1}^y)^\alpha \right\}} \left\| \grad_x \widetilde{f}(x_t, y_t) \right\|^2 \right]\\
\leq &(L_{\Phi}\bar{c}+2Gc_R) \eta^x \mathbb{E} \left[ \sum_{t=0}^{T-1} \frac{1}{(v_{t+1}^x)^\alpha} \left\| \grad_x \widetilde{f}(x_t, y_t) \right\|^2 \right]\\
\leq &(L_{\Phi}\bar{c}+2Gc_R) \eta^x \mathbb{E} \left[ \left( \frac{v_T^x}{(v_0^x)^\alpha} + \sum_{t=0}^{T-1} \frac{1}{(v_{t+1}^x)^\alpha} \right) \left\| \grad_x \widetilde{f}(x_t, y_t) \right\|^2 \right]\\
\leq &\frac{(L_{\Phi}\bar{c}+2Gc_R) \eta^x}{1 - \alpha} \mathbb{E} \left[ (v_T^x)^{1 - \alpha} \right].
\end{align*}
For the Term \( \rm{(K)} \), using the smoothness and concavity in \( y \), we have
\[
\mathbb{E} \left[ \sum_{t=0}^{T-1} \left\| \grad_x f(x_t, y_t) - \grad \Phi(x_t) \right\|^2 \right]
\leq L_1^2 \mathbb{E} \left[ \sum_{t=0}^{T-1}  d^2( y_t, y_t^* ) \right]
\leq 2L_1\kappa \mathbb{E} \left[ \sum_{t=0}^{T-1} \left( f(x_t, y_t^*) - f(x_t, y_t) \right) \right],
\]
Now define
\(C = \left( \frac{8L_1{\kappa}^3 (\eta^x)^2}{\mu (\eta^y)^2} \right)^{\frac{1}{2\alpha - 2\beta}}\)
and apply Lemma~\ref{le.bound conclusion}. Putting all terms together, we obtain
\begin{align*}
    &\mathbb{E} \left[ \sum_{t=0}^{T-1} \left\| \grad_x f(x_t, y_t) \right\|^2 \right]\\
\leq &\frac{1}{2} \mathbb{E} \left[ \sum_{t=0}^{T-1} \left\| \grad_x f(x_t, y_t) \right\|^2 \right]
+ 2\mathbb{E} \left[ \frac{\Delta \Phi}{\eta^x} \max \left\{ (v_T^x)^\alpha, (v_T^y)^\alpha \right\} \right]
+ \frac{(L_{\Phi}\bar{c}+2Gc_R) \eta^x}{1 - \alpha} \mathbb{E} \left[ (v_T^x)^{1 - \alpha} \right]\\
& + \frac{L_1\kappa (\zeta\bar{c}+\frac{2G}{\mu}c_R)\eta^y}{1 - \beta} \mathbb{E} \left[ (v_T^y)^{1 - \beta} \right]
+ \left( \frac{1}{\mu} + \frac{\eta^y}{(v_0^y)^\beta} \right) \cdot \frac{8L_1\kappa^2\eta^x G^2}{\eta^y (v_0^y)^\alpha} \mathbb{E} \left[ (v_T^y)^\beta \right]\\
&+ \frac{\kappa^4\bar{c} ( \eta^y C^\beta + 2C^{2\beta}) (\eta^x)^2}{\mu(\eta^y)^2}\mathbb{E} \left[ \frac{1 + \log v_T^x - \log v_0^x}{(v_0^x)^{2\alpha - 1}} \cdot \mathbb{I}_{\alpha \geq 0.5}
+ \frac{(v_T^x)^{1 - 2\alpha}}{1 - 2\alpha} \cdot \mathbb{I}_{\alpha < 0.5} \right]\\
&+ \left(\zeta\bar{c}{\kappa}^2 + \frac{(2\kappa\sqrt{\bar{c}}(v_0^y)^{\alpha}+\kappa Gc_R\eta^x)^2(\eta^x)^{\beta/\alpha}}{\mu \eta^y (v_0^y)^{2\alpha}}\right)
\frac{L_1\kappa(\eta^x)^{2-\beta/\alpha}}{(1 - \alpha)\eta^y (v_0^y)^{\alpha - 2\beta}} \mathbb{E} \left[ (v_T^x)^{1 - \alpha} \right]
+ \frac{\kappa^2 (v_0^y)^\beta G^2}{\mu \eta^y}\\
&+ \frac{L_1 \kappa \left( 2 \beta G \right)^{\frac{1}{1 - \beta} + 2} G^2}
{2 \mu^{\frac{1}{1 - \beta} + 3} \left( \eta^y \right)^{\frac{1}{1 - \beta} + 2} \left( v_0^y \right)^{2 - 2 \beta}}.
\end{align*}
Rearranging the above inequality yields the desired result.
\end{proof}
\begin{theorem}[Restatement of Theorem \ref{the.stochastic convergence}]
    Under Assumptions~\ref{ass.l-smooth}, \ref{ass.g-convex},  \ref{ass.optimal exist and stationary condition hold}, \ref{ass.manifold}, \ref{ass.retraction} and \ref{ass.grad bound}, for any $0<2\beta \le \alpha<1$, the following bounds hold
    \begin{align*}
&\mathbb{E}\left[\frac{1}{T}\sum_{t=0}^{T-1}\|\grad_x f(x_t,y_t)\|^2\right] \\
\leq& \frac{4\Delta\Phi G^{2\alpha}}{\eta^x T^{1-\alpha}}+\left(\frac{4(L_{\Phi}\bar{c}+2Gc_R)\eta^x}{1-\alpha}+\left(\zeta\bar{c}{\kappa}^2 + \frac{(2\kappa\sqrt{\bar{c}}(v_0^y)^{\alpha}+\kappa Gc_R\eta^x)^2(\eta^x)^{\beta/\alpha}}{\mu \eta^y (v_0^y)^{2\alpha}}\right)
\frac{2L_1\kappa(\eta^x)^{2-\beta/\alpha}}{(1 - \alpha)\eta^y (v_0^y)^{\alpha - 2\beta}}\right)\frac{G^{2(1-\alpha)}}{T^\alpha} \\
& + \frac{2L_1\kappa(\zeta\bar{c}+\frac{2G}{\mu}c_R)\eta^y G^{2(1-\beta)}}{(1-\beta)T^\beta} + \left(\frac{1}{\mu} + \frac{\eta^y}{\left(v_0^y\right)^\beta}\right)\frac{16L_1\kappa^2\eta^x G^{2(1+\beta)}}{\eta^y\left(v_0^y\right)^\alpha T^{1-\beta}} \\
& + \frac{2\kappa^4\bar{c}\left(\mu\eta^y C^\beta + 2C^{2\beta}\right)\left(\eta^x\right)^2}{\left(\eta^y\right)^2}\left(\frac{1+\log(G^2T) - \log v_0^x}{\left(v_0^x\right)^{2\alpha-1}T}\cdot\mathbf{1}_{\alpha\geq 0.5} + \frac{G^{2(1-2\alpha)}}{(1-2\alpha)T^{2\alpha}}\cdot\mathbf{1}_{\alpha<0.5}\right)\\
&+ \frac{2\kappa^2\left(v_0^y\right)^\beta G^2}{\mu\eta^y T} + \frac{L_1\kappa\left(2\beta G\right)^{\frac{1}{1-\beta}+2}G^2}{\mu^{\frac{1}{1-\beta}+3}\left(\eta^y\right)^{\frac{1}{1-\beta}+2}\left(v_0^y\right)^{2-2\beta}T},
\end{align*}
and
\begin{align*}
&\mathbb{E}\left[ \frac{1}{T}\sum_{t=0}^{T-1}\|\grad_y f(x_t,y_t)\|^2 \right] \\
\leq &\frac{4\kappa^3\left(\eta^x\right)^2}{\left(\eta^y\right)^2C^{2\alpha-2\beta}}\mathbb{E}\left[ \frac{1}{T}\sum_{t=0}^{T-1}\|\grad_x f(x_t,y_t)\|^2 \right] + \frac{L_1(\zeta\bar{c}+\frac{2G}{\mu}c_R)\eta^y G^{2-2\beta}}{(1-\beta)T^\beta} + \left(\frac{1}{\mu} + \frac{\eta^y}{\left(v_0^y\right)^\beta}\right)\frac{8L_1{\kappa}\eta^x G^{2+2\beta}}{\eta^y\left(v_0^y\right)^\alpha T^{1-\beta}} \\
&+ \frac{\kappa^3\bar{c}\left(\mu\eta^y C^\beta + 2C^{2\beta}\right)\left(\eta^x\right)^2}{\left(\eta^y\right)^2}\left(\frac{1+\log TG^2 - \log v_0^x}{\left(v_0^x\right)^{2\alpha-1}T}\cdot\mathbf{1}_{\alpha\geq 0.5} + \frac{G^{2-4\alpha}}{(1-2\alpha)T^{2\alpha}}\cdot\mathbf{1}_{\alpha<0.5}\right)\\
&+ \left(\zeta\bar{c}{\kappa}^2 + \frac{(2\kappa\sqrt{\bar{c}}(v_0^y)^{\alpha}+\kappa Gc_R\eta^x)^2(\eta^x)^{\beta/\alpha}}{\mu \eta^y (v_0^y)^{2\alpha}}\right)\frac{L_1\left(\eta^x\right)^{2-\beta/\alpha}G^{2-2\alpha}}{(1-\alpha)\eta^y\left(v_0^y\right)^{\alpha-2\beta}T^\alpha} + \frac{\kappa\left(v_0^y\right)^\beta G^2}{\mu\eta^y T}\\
&+ \frac{2L_1\left(2\beta G\right)^{\frac{1}{1-\beta}+2}G^2}{4\mu^{\frac{1}{1-\beta}+3}\left(\eta^y\right)^{\frac{1}{1-\beta}+2}\left(v_0^y\right)^{2-2\beta}T}.
\end{align*}
\end{theorem}
\begin{proof}
    From assumption \ref{ass.grad bound}, we can get
   \[ \max\{(v_T^x)^m, (v_T^y)^m\} \le (TG)^m \] completes the proof of the first inequality.  By substituting the above inequality into the inequalities of Lemma \ref{le.bound EvTy}, \ref{le.bound EvTx} and dividing both sides by $T$, we obtain the desired result.  
\end{proof}
\section{Proof of Theorem \ref{the.improved stochastic convergence}}\label{app.improved stochastic convergence}
The proof of the following lemma can be found in 
\cite{han2024framework}
. We include it here for the sake of completeness.
\begin{lemma}[Lemma 4 in 
\cite{han2024framework}
]\label{le.ystar lips smooth}
Suppose Assumptions \ref{ass.l-smooth}, \ref{ass.g-convex}, \ref{ass.optimal exist and stationary condition hold} and \ref{ass.second smooth} hold. Then the mapping \( y^*(x) \) is Lipschitz smooth. Specifically, for any \( x_1, x_2 \in \cM_x \), we have
\[
\left\| \grad y^*(x_1) - \Gamma_{y^*(x_2)}^{y^*(x_1)} \grad y^*(x_2) \Gamma_{x_1}^{x_2} \right\| \leq L_y \, d(x_1, x_2),
\]
where the Lipschitz constant is given by
\[
L_y = \frac{\kappa^2 L_2}{\mu} + \frac{2 \kappa L_2}{\mu} + \frac{L_2}{\mu}.
\]
\end{lemma}
\begin{proof}
    From Assumption~\ref{ass.l-smooth}, for any \( v \in \operatorname{T}_x\cM_x \), we have
    \begin{align*}
        \|\grad_{yx} f(x, y)[v]\|_y 
        &= \|\operatorname{D}_x \grad_y f(x, y)[v]\|_y \\
        &\le \lim_{t \to 0} \frac{\|\grad_y f(\Exp_x(tv), y) - \grad_y f(x, y)\|_y}{|t|} \\
        &\le \lim_{t \to 0} \frac{L_1 \|tv\|_x}{|t|} = L_1 \|v\|_x.
    \end{align*}
    Hence, \( \|\grad_{yx} f(x, y)\| \le L_1 \). 

    According to Proposition 1 in~\cite{han2024framework}
    , we have
    \[
    \operatorname{D}y^*(x) = -\hess_y f(x, y^*(x))^{-1} \circ \grad_{yx} f(x, y^*(x)).
    \]
    Then, it follows that
    \begin{align*}
        &\|\operatorname{D}y^*(x_1)-\Gamma^{y^*(x_1)}_{y^*(x_2)}\operatorname{D}y^*(x_2)\Gamma^{x_2}_{x_1}\|_{y^*(x_1)}\\
        \le & \|\grad_{yx}f(x_1,y^*(x_1))\|\|\hess_y^{-1}f(x_1,y^*(x_1))-\Gamma^{y^*(x_1)}_{y^*(x_2)}\hess_y^{-1}f(x_2,y^*(x_2))\Gamma^{y^*(x_2)}_{y^*(x_1)}\|\\
        &+\|\hess^{-1}_yf(x_2,y^*(x_2))\|\|\Gamma^{y^*(x_2)}_{y^*(x_1)}\grad_{yx}f(x_1,y^*(x_1))-\grad_{yx}f(x_2,y^*(x_2))\Gamma_{x_1}^{x_2}\|\\
        \le &\frac{L_1L_2}{\mu^2}(d(x_1,x_2)+d(y^*(x_1),y^*(x_2)))+\frac{L_2}{\mu}(d(x_1,x_2)+d(y^*(x_1),y^*(x_2)))\\
        \le &(\frac{L_1^2L_2}{\mu^3}+\frac{2L_1L_2}{\mu^2}+\frac{L_2}{\mu})d(x_1,x_2)
    \end{align*}
    where the second inequality uses Assumptions~\ref{ass.l-smooth}, \ref{ass.g-convex}, and~\ref{ass.second smooth}, as well as the identity \( \|A^{-1} - B^{-1}\| \le \|A^{-1}\| \cdot \|A - B\| \cdot \|B^{-1}\| \), and the last inequality applies Lemma~\ref{le.lips smooth}.
\end{proof}
\begin{lemma}\label{le.v>C'}
    Suppose the assumptions in Lemma \ref{le.v<C} and Assumption \ref{ass.second smooth} hold. Let \( t_3 \) be defined as in Lemma \ref{le.v<C} and $0<\beta\le\alpha<1$, then we have
\begin{align*}
&\mathbb{E}\left[ \sum_{t=t_3}^{T-1} (f(x_t, y_t^*) - f(x_t, y_t)) \right] \\
\leq &\mathbb{E} \left[ \sum_{t=t_3}^{T-1} \frac{1 - \gamma_t \mu}{2\gamma_t} d^2( y_t , y_t^* ) - \frac{1}{\gamma_t (2 + \mu \gamma_t)} d^2(y_{t+1}, y_{t+1}^*) \right] + \mathbb{E} \left[ \sum_{t=t_3}^{T-1} \frac{(\zeta\bar{c}+\frac{2G}{\mu}c_R)\gamma_t}{2} \| \grad_y \tilde{f}(x_t, y_t) \|^2 \right] \\
& + \left( \zeta\bar{c}\kappa^2 + \frac{(L_y\bar{c}+2c_R\kappa)^2 G^2 (\eta^x)^2}{\mu \eta^y (v_0^y)^{2\alpha - \beta}} \right) \frac{(\eta^x)^2}{2(1 - \alpha) \eta^y (v_0^y)^{\alpha - \beta}} \mathbb{E} \left[ (v_T^x)^{1-\alpha} \right] \\
& + \frac{4 \kappa^2 (\eta^x)^2}{\mu (\eta^y)^2 C^{2\alpha - 2\beta+\delta}} \mathbb{E} \left[ \sum_{t=t_3}^{T-1} \| \grad_x f(x_t, y_t) \|^2 \right] + \left( \frac{1}{\mu} + \frac{\eta^y}{(v_0^y)^\beta} \right) \frac{4 \kappa \eta^x G^2}{\eta^y (v_0^y)^{\alpha}} \mathbb{E} \left[ (v_T^y)^\beta \right],
\end{align*}
for any $\delta\le \log(2)/\log(TG)$.
\end{lemma}
\begin{proof}
To bound Term (D) in \eqref{eq.equation contain term D}, we consider an arbitrary $\tau > 0$ and proceed as follows
\begin{align*}
    &-2 \left\langle \Exp^{-1}_{y^*_t}(y_{t+1}), \Exp^{-1}_{y^*_t}(y_{t+1}^*)  - \operatorname{D}y^*(x_t)[-\eta_t\grad_x\tilde{f}(x_t,y_t)] \right\rangle\\
    \le &2d(y_{t+1},y_t^*)\left\|\Exp^{-1}_{y^*_t}(y_{t+1}^*)  - \operatorname{D}y^*(x_t)[-\eta_t\grad_x\tilde{f}(x_t,y_t)] \right\|\\
    \le
    &2d(y_{t+1},y_t^*)\left(\left\|\Exp^{-1}_{y^*_t}(y_{t+1}^*)  - \operatorname{D}y^*(x_t)[\Exp_{x_t}^{-1}(x_{t+1})] \right\|+\left\|\operatorname{D}y^*(x_t)[\Exp_{x_t}^{-1}(x_{t+1})-\eta_t\grad_x\tilde{f}(x_t,y_t)]\right\|\right)\\
    \le 
    &2d(y_{t+1},y_t^*)\left(\frac{L_y}{2}\|\Exp_{x_t}^{-1}(x_{t+1})\|^2+\|\operatorname{D}y^*(x_t)\|\|\Exp_{x_t}^{-1}(x_{t+1})-\eta_t\grad_x\tilde{f}(x_t,y_t)\|^2\right)\\
    \le &2d(y_{t+1},y_t^*)\left(\frac{L_y\eta_t^2\bar{c}}{2}\|\grad_x \tilde{f}(x_t, y_t)\|^2+\kappa c_R\eta_t^2\|\grad_x \tilde{f}(x_t, y_t)\|\right)\\
    \le &(L_y\bar{c}+2c_R\kappa)\eta_t^2d(y_{t+1},y_t^*)G\|\grad_x\tilde{f}(x_t,y_t)\|\\
    \le &\frac{\tau (L_y\bar{c}+2c_R\kappa) G^2 \eta_t^2}{2}d^2(y_{t+1}, y_t^*) + \frac{(L_y\bar{c}+2c_R\kappa) \eta_t^2}{2 \tau} \| \grad_x \tilde{f}(x_t, y_t)\|^2.
\end{align*}
Here, the third inequality uses Lemma \ref{le.ystar lips smooth}, the fourth follows from Assumption \ref{ass.retraction}, and the final step applies Young’s inequality.

By combining the above inequality with \eqref{eq.equation contain term D} and \eqref{eq.bound Term C}, we obtain
\begin{align*}
d^2(y_{t+1}, y_{t+1}^*) \leq &\left(1 + \lambda_t + \frac{\tau (L_y\bar{c}+2c_R{\kappa}) G^2 \eta_t^2}{2}\right) d^2(y_{t+1}, y_t^*)  + \left(\zeta\bar{c}{\kappa}^2 + \frac{L_y\bar{c}+2c_R{\kappa}}{2 \tau}\right) \eta_t^2 \| \grad_x \tilde{f}(x_t, y_t)\|^2 \\
& + \frac{{\kappa}^2 \eta_t^2}{\lambda_t} \| \grad_x f(x_t, y_t)\|^2 +2\eta_t\left\langle\Exp^{-1}_{y^*_t}(y_{t+1}), \operatorname{D} y^*(x_t)[\grad_x\tilde{f}(x_t,y_t)-\grad_x f(x_t,y_t)]\right\rangle.
\end{align*}
To bound \(\eta_t\), we note that
\[
\eta_t = \frac{\eta^x}{\max \{v_{t+1}^x, v_{t+1}^y\}^{\alpha}} \leq \frac{\eta^x}{(v_{t+1}^y)^{\alpha}} \leq \frac{\eta^x}{(v_0^y)^{\alpha}},
\]
and
\[
\eta_t \leq \frac{\eta^x}{(v_{t+1}^y)^{\alpha}} = \frac{\eta^x}{(v_{t+1}^y)^{\alpha - \beta} (v_{t+1}^y)^{\beta}} \leq \frac{\eta^x}{(v_0^y)^{\alpha - \beta} (v_{t+1}^y)^{\beta}}.
\]
Substituting these bounds into the previous inequality yields
\begin{align*}
d^2(y_{t+1}, y_{t+1}^*) \leq &\left(1 + \lambda_t + \frac{\tau (L_y\bar{c}+2c_R{\kappa}) G^2 (\eta^x)^2}{2(v_0^y)^{2\alpha - \beta} (v_{t+1}^y)^{\beta}}\right) d^2(y_{t+1}, y_t^*) + \left(\zeta\bar{c}{\kappa}^2 + \frac{(L_y\bar{c}+2c_R{\kappa})}{2 \tau}\right) \eta_t^2 \| \grad_x \tilde{f}(x_t, y_t)\|^2 \\
& + \frac{{\kappa}^2 \eta_t^2}{\lambda_t} \| \grad_x f(x_t, y_t)\|^2  +2\eta_t\left\langle\Exp^{-1}_{y^*_t}(y_{t+1}), \operatorname{D} y^*(x_t)[\grad_x\tilde{f}(x_t,y_t)-\grad_x f(x_t,y_t)]\right\rangle.
\end{align*}
Now, let we choose \(\lambda_t = \frac{\mu \eta^y}{8(v_{t+1}^y)^{\beta-\delta}}\) with $\delta\le \log(2)/\log(TG)$ and \(\tau = \frac{\mu \eta^y (v_0^y)^{2\alpha - \beta}}{2(L_y\bar{c}+2c_R{\kappa}) G^2 (\eta^x)^2}\). Substituting them gives
\begin{align*}
&d^2(y_{t+1}, y_{t+1}^*)\\
\leq &\left(1 + \frac{\mu \eta^y}{2(v_{t+1}^y)^{\beta}}\right) d^2(y_{t+1}, y_t^*)  + \left(\zeta\bar{c}{\kappa}^2 + \frac{(L_y\bar{c}+2c_R{\kappa})^2 G^2 (\eta^x)^2}{\mu \eta^y (v_0^y)^{2\alpha - \beta}}\right) \eta_t^2 \| \grad_x \tilde{f}(x_t, y_t)\|^2\\
&+\frac{8{\kappa}^2 \left( v_{t+1}^y \right)^{\beta-\delta} \eta_t^2}{\mu \eta^y} \| \grad_x f(x_t, y_t) \|^2 + 2\eta_t\left\langle\Exp^{-1}_{y^*_t}(y_{t+1}), \operatorname{D} y^*(x_t)[\grad_x\tilde{f}(x_t,y_t)-\grad_x f(x_t,y_t)]\right\rangle.
\end{align*}
By applying Lemma \ref{le.stochastic auxiliary lemma}, we arrive at
\begin{align*}
&\mathbb{E} \left[ \sum_{t=t_3}^{T-1} \left( f(x_t, y_t^*) - f(x_t, y_t) \right) \right] \\
\leq &\mathbb{E} \left[ \sum_{t=t_3}^{T-1} \left( \frac{1 - \gamma_t \mu}{2 \gamma_t}d^2(y_t, y_t^*) - \frac{1}{\gamma_t (2 + \mu \gamma_t)}d^2( y_{t+1}, y_{t+1}^*) \right) \right] + \mathbb{E} \left[ \sum_{t=t_3}^{T-1} \frac{(\zeta\bar{c}+\frac{2G}{\mu}c_R)\gamma_t}{2} \| \grad_y \tilde{f}(x_t, y_t) \|^2 \right] \\
&+ \underbrace{\mathbb{E} \left[ \sum_{t=t_3}^{T-1} \frac{1}{\gamma_t (2 + \mu \gamma_t)} \left( \zeta\bar{c}{\kappa}^2 + \frac{(L_y\bar{c}+2c_R{\kappa})^2 G^2 (\eta^x)^2}{\mu \eta^y (v_0^y)^{2\alpha - \beta}} \right) \eta_t^2 \| \grad_x \tilde{f}(x_t, y_t) \|^2 \right]}_{\rm (E)} \\ 
&+ \underbrace{\mathbb{E} \left[ \sum_{t=t_3}^{T-1} \frac{8{\kappa}^2 \left( v_{t+1}^y \right)^{\beta-\delta} \eta_t^2}{\gamma_t (2 + \mu \gamma_t) \mu \eta^y} \| \grad_x f(x_t, y_t) \|^2 \right]}_{\rm (F)} \\
& + \underbrace{\mathbb{E} \left[ \sum_{t=t_3}^{T-1} \frac{2\eta_t}{\gamma_t (2 + \mu \gamma_t)}\left\langle\Exp^{-1}_{y^*_t}(y_{t+1}), \operatorname{D} y^*(x_t)[\grad_x\tilde{f}(x_t,y_t)-\grad_x f(x_t,y_t)]\right\rangle\right]}_{\rm (G)}
\end{align*}
For the term {\rm (E)}
\begin{align*}
\text{\rm Term (E)} \leq &\left( \zeta\bar{c}{\kappa}^2 + \frac{(L_y\bar{c}+2c_R{\kappa})^2 G^2 (\eta^x)^2}{\mu \eta^y (v_0^y)^{2\alpha - \beta}} \right) \mathbb{E} \left[ \sum_{t=t_3}^{T-1} \frac{\eta_t^2}{2 \gamma_t} \| \grad_x \tilde{f}(x_t, y_t) \|^2 \right] \\
= &\left( \zeta\bar{c}{\kappa}^2 + \frac{(L_y\bar{c}+2c_R{\kappa})^2 G^2 (\eta^x)^2}{\mu \eta^y (v_0^y)^{2\alpha - \beta}} \right) \mathbb{E} \left[ \sum_{t=t_3}^{T-1} \frac{(\eta^x)^2 (v_{t+1}^y)^\beta}{2 \eta^y \max \{ v_{t+1}^x, v_{t+1}^y \}^{2\alpha}} \| \grad_x \tilde{f}(x_t, y_t) \|^2 \right] \\
\leq &\left( \zeta\bar{c}{\kappa}^2 + \frac{(L_y\bar{c}+2c_R{\kappa})^2 G^2 (\eta^x)^2}{\mu \eta^y (v_0^y)^{2\alpha - \beta}} \right) \mathbb{E} \left[ \sum_{t=t_3}^{T-1} \frac{(\eta^x)^2 (v_{t+1}^y)^\beta}{2 \eta^y (v_{t+1}^y)^{\beta}(v_{t+1}^y)^{\alpha-\beta}(v_{t+1}^x)^{\alpha}} \| \grad_x \tilde{f}(x_t, y_t) \|^2 \right] \\
\leq &\left( \zeta\bar{c}{\kappa}^2 + \frac{(L_y\bar{c}+2c_R{\kappa})^2 G^2 (\eta^x)^2}{\mu \eta^y (v_0^y)^{2\alpha - \beta}} \right) \mathbb{E} \left[ \sum_{t=t_3}^{T-1} \frac{(\eta^x)^2}{2 \eta^y (v_0^y)^{\alpha-\beta}(v_{t+1}^x)^{\alpha}} \| \grad_x \tilde{f}(x_t, y_t) \|^2 \right]\\
\leq &\left( \zeta\bar{c}{\kappa}^2 + \frac{(L_y\bar{c}+2c_R{\kappa})^2 G^2 (\eta^x)^2}{\mu \eta^y (v_0^y)^{2\alpha - \beta}} \right)
\mathbb{E} \left[\frac{(\eta^x)^2}{2\eta^y(v_0^y)^{\alpha-\beta}} \left( \frac{v_t^x}{(v_0^x)^\alpha} + \sum_{t = 0}^{T - 1} \frac{1}{(v_{t+1}^x)^\alpha} \left\| \grad_x \tilde{f}(x_t, y_t) \right\|^2 \right) \right]\\
\leq &\left( \zeta\bar{c}{\kappa}^2 + \frac{(L_y\bar{c}+2c_R{\kappa})^2 G^2 (\eta^x)^2}{\mu \eta^y (v_0^y)^{2\alpha - \beta}} \right)
\frac{(\eta^x)^2}{2(1 - \alpha)\eta^y (v_0^y)^{\alpha - \beta}} \mathbb{E} \left[ (v_T^x)^{1 - \alpha} \right]
\end{align*}
where we used Lemma \ref{le.adaptive bound} in the last step.

For the term {\rm (F)}, by using bounds on $\eta_t$ and monotonicity of $v_t^y$, we obtain:
\begin{equation}\label{eq.bound term F'}
    \begin{split}
        \text{\rm Term (F)}  \leq &\mathbb{E} \left[ \sum_{t=t_3}^{T-1} \frac{4{\kappa}^2 \left( v_{t+1}^y \right)^{\beta-\delta} \eta_t^2}{\gamma_t \mu \eta^y} \| \grad_x f(x_t, y_t) \|^2 \right] \\
= &\frac{4{\kappa}^2 (\eta^x)^2}{\mu (\eta^y)^2} \mathbb{E} \left[ \sum_{t=t_3}^{T-1} \frac{(v_{t+1}^y)^{2\beta-\delta}}{\max \{ v_{t+1}^x, v_{t+1}^y \}^{2\alpha}} \| \grad_x f(x_t, y_t) \|^2 \right] \\
\leq &\frac{4{\kappa}^2 (\eta^x)^2}{\mu (\eta^y)^2} \mathbb{E} \left[ \sum_{t=t_3}^{T-1} \frac{(v_{t+1}^y)^{2\beta-\delta}}{(v_{t+1}^y)^{2\alpha}} \| \grad_x f(x_t, y_t) \|^2 \right] \\
\leq &\frac{4{\kappa}^2 (\eta^x)^2}{\mu (\eta^y)^2} \mathbb{E} \left[ \frac{1}{(v_{t_3 + 1}^y)^{2\alpha - 2\beta+\delta}} \sum_{t=t_3}^{T-1} \| \grad_x f(x_t, y_t) \|^2 \right] \\
\leq &\frac{4{\kappa}^2 (\eta^x)^2}{\mu (\eta^y)^2 C^{2\alpha - 2\beta+\delta}} \mathbb{E} \left[ \sum_{t=t_3}^{T-1} \| \grad_x f(x_t, y_t) \|^2 \right].
    \end{split}
\end{equation}
Term {\rm (G)} can be bounded using inequality and \eqref{eq.bound term G}, which completes the proof.
\end{proof}
\textbf{Proof of Lemma \ref{le.imporve bound conclusion}}
\begin{proof}
    By combining Lemmas \ref{le.v<C}, \ref{le.4} and \ref{le.v>C'}, and following a similar argument as in the proof of Lemma \ref{le.bound conclusion}, we obtain the bound.
\end{proof}
\begin{theorem}[Restatement of Theorem \ref{the.improved stochastic convergence}]
    Under Assumptions~\ref{ass.l-smooth}, \ref{ass.g-convex}, \ref{ass.optimal exist and stationary condition hold}, \ref{ass.manifold}, \ref{ass.retraction}, \ref{ass.grad bound} and \ref{ass.second smooth}, for any $0<\beta\le\alpha<1$, the following bounds hold
    \begin{align*}
&\mathbb{E}\left[\frac{1}{T}\sum_{t=0}^{T-1}\|\grad_x f(x_t,y_t)\|^2\right] \\
\leq& \frac{4\Delta\Phi G^{2\alpha}}{\eta^x T^{1-\alpha}}+\left(\frac{4(L_{\Phi}\bar{c}+2Gc_R)\eta^x}{1-\alpha}+\left(\zeta\bar{c}{\kappa}^2+\frac{(L_y\bar{c}+2c_R{\kappa})^2G^2\left(\eta^x\right)^2}{\mu\eta^y\left(v_0^y\right)^{2\alpha-\beta}}\right)\frac{2L_1\kappa\left(\eta^x\right)^2}{(1-\alpha)\eta^y\left(v_0^y\right)^{\alpha-\beta}}\right)\frac{G^{2(1-\alpha)}}{T^\alpha} \\
& + \frac{2L_1\kappa(\zeta\bar{c}+\frac{2G}{\mu}c_R)\eta^y G^{2(1-\beta)}}{(1-\beta)T^\beta} + \left(\frac{1}{\mu} + \frac{\eta^y}{\left(v_0^y\right)^\beta}\right)\frac{16L_1\kappa^2\eta^x G^{2(1+\beta)}}{\eta^y\left(v_0^y\right)^\alpha T^{1-\beta}} \\
& + \frac{2\kappa^4\bar{c}\left(\mu\eta^y C^\beta + 2C^{2\beta}\right)\left(\eta^x\right)^2}{\left(\eta^y\right)^2}\left(\frac{1+\log(G^2T) - \log v_0^x}{\left(v_0^x\right)^{2\alpha-1}T}\cdot\mathbf{1}_{\alpha\geq 0.5} + \frac{G^{2(1-2\alpha)}}{(1-2\alpha)T^{2\alpha}}\cdot\mathbf{1}_{\alpha<0.5}\right)\\
&+ \frac{2\kappa^2\left(v_0^y\right)^\beta G^2}{\mu\eta^y T} + \frac{L_1\kappa\left(2\beta G\right)^{\frac{1}{1-\beta}+2}G^2}{\mu^{\frac{1}{1-\beta}+3}\left(\eta^y\right)^{\frac{1}{1-\beta}+2}\left(v_0^y\right)^{2-2\beta}T},
\end{align*}
and
\begin{align*}
&\mathbb{E}\left[ \frac{1}{T}\sum_{t=0}^{T-1}\|\grad_y f(x_t,y_t)\|^2 \right] \\
\leq &\frac{8\kappa^3\left(\eta^x\right)^2}{\left(\eta^y\right)^2C^{2\alpha-2\beta+\delta}}\mathbb{E}\left[ \frac{1}{T}\sum_{t=0}^{T-1}\|\grad_x f(x_t,y_t)\|^2 \right] + \frac{L_1(\zeta\bar{c}+\frac{2G}{\mu}c_R)\eta^y G^{2-2\beta}}{(1-\beta)T^\beta} + \left(\frac{1}{\mu} + \frac{\eta^y}{\left(v_0^y\right)^\beta}\right)\frac{8L_1{\kappa}\eta^x G^{2+2\beta}}{\eta^y\left(v_0^y\right)^\alpha T^{1-\beta}} \\
&+ \frac{\kappa^3\bar{c}\left(\mu\eta^y C^\beta + 2C^{2\beta}\right)\left(\eta^x\right)^2}{\left(\eta^y\right)^2}\left(\frac{1+\log TG^2 - \log v_0^x}{\left(v_0^x\right)^{2\alpha-1}T}\cdot\mathbf{1}_{\alpha\geq 0.5} + \frac{G^{2-4\alpha}}{(1-2\alpha)T^{2\alpha}}\cdot\mathbf{1}_{\alpha<0.5}\right)\\
&+ \left(\zeta\bar{c}{\kappa}^2 + \frac{(L_y\bar{c}+2c_R{\kappa})^2G^2\left(\eta^x\right)^2}{\mu\eta^y\left(v_0^y\right)^{2\alpha-\beta}}\right)\frac{L_1\left(\eta^x\right)^2G^{2-2\alpha}}{(1-\alpha)\eta^y\left(v_0^y\right)^{\alpha-\beta}T^\alpha} + \frac{\kappa\left(v_0^y\right)^\beta G^2}{\mu\eta^y T} + \frac{2L_1\left(2\beta G\right)^{\frac{1}{1-\beta}+2}G^2}{4\mu^{\frac{1}{1-\beta}+3}\left(\eta^y\right)^{\frac{1}{1-\beta}+2}\left(v_0^y\right)^{2-2\beta}T}.
\end{align*}
\end{theorem}
\begin{proof}
Replacing Lemma \ref{le.bound conclusion} with Lemma \ref{le.imporve bound conclusion} and following the same derivation as in the proof of Theorem~\ref{the.stochastic convergence} yields the desired result.
\end{proof}

\end{document}